\documentclass[reqno,twoside,11pt,english]{amsart}
\usepackage{amsmath,amsfonts,amssymb,amsthm,epsfig}
\newcommand{\RNum}[1]{\uppercase\expandafter{\romannumeral #1\relax}}
\usepackage{comment}

\usepackage{color}

\usepackage[pagebackref,colorlinks,citecolor=blue,linkcolor=blue]{hyperref}
\usepackage{graphicx}
\usepackage{titletoc,epsf}
\usepackage{amsmath,amsfonts,latexsym,amsthm,amsxtra,amssymb,bbm}
\allowdisplaybreaks
\usepackage{graphicx,float}

\allowdisplaybreaks

\let\th=T
\let\pa=\partial

\def\h{\frak h}

\newtheorem{theorem}{Theorem}[section]
\newtheorem{lemma}[theorem]{Lemma}
\newtheorem{proposition}[theorem]{Proposition}
\newtheorem{corollary}[theorem]{Corollary}
\newtheorem{definition}[theorem]{Definition}
\newtheorem{remark}[theorem]{Remark}

\numberwithin{equation}{section}

\numberwithin{equation}{section}

\topmargin - .23 in

\begin{document}

\title[GWP of NLS on hyperbolic spaces]{Global well-posedness and scattering of energy-critical defocusing nonlinear Schr\"odinger equation on high-dimensional hyperbolic spaces}

\author{Shunqin Zhang, Jiqiang Zheng, and Yunrui Zheng}

\address[S. Zhang]{School of mathematics, Hohai University, Nanjing 210098, Jiangsu,  China { and Laboratory of Mathematical Modeling and Intelligent Computing for Water Systems, Hohai University, Nanjing, 211100, China}}
\email{{\tt 250219010006@hhu.edu.cn}}

\address[J. Zheng]{Institute of applied physics and computational mathematics and national key laboratory of computational physics, Beijing 100088, China}
\email{{\tt   zheng\_jiqiang@iapcm.ac.cn}}

\address[Y. Zheng]{ School of mathematics, Shandong University, Shandong 250100, Jinan,   China.}
\email{{\tt yunrui\_zheng@sdu.edu.cn}}

\thanks{S. Zhang has been partially supported by the NSF of Jiangsu Province (Grant No.~BK20221497). J. Zheng is supported by National Key R\&D program of China: 2021YFA1002500 and NSF Grant of China (No.~12271051). Y. Zheng is  partially supported by the NSF of
	Shandong Province (No.~ZR2026MS0070) and the National Natural Science Foundation of China( No.~11901350 and No.~12371172).}

\begin{abstract}
We prove global well-posedness and scattering for the energy-critical defocusing nonlinear Schr\"odinger equation
\begin{equation*}
\begin{cases} 
	i\partial_t u + \Delta u = |u|^{\frac{4}{d-2}} u, \\ 
	u(0) = u_0,
\end{cases}
\end{equation*}
on hyperbolic space $\mathbb{H}^d(d \geq 4)$. We mainly utilize the concentration-compactness/rigidity method. We establish a linear profile decomposition on high-dimensional hyperbolic spaces and prove that the nonlinear profile on hyperbolic spaces can be approximated by the solutions of the energy-critical defocusing nonlinear Schr\"odinger equation on Euclidean spaces, then obtaining the global well-posedness and scattering of the energy-critical defocusing nonlinear Schr\"odinger equation on the high-dimensional hyperbolic space.

\bigskip

\noindent \textbf{Keywords}: Nonlinear Schrödinger equation, scattering, hyperbolic space, energy-critical, concentration-compactness/rigidity method.

\bigskip

\noindent {\textbf{Mathematics Subject Classification (2020)} Primary: 35L05; Secondary: 35L71, 35Q40, 35P25, 35B40}
\end{abstract}

\maketitle

\section{Introduction}

In this article, we consider the nonlinear Schrödinger equation (NLS) posed on the hyperbolic space $\mathbb{H}^d$ when $d \geq 4$, which is the unique simply connected complete Riemannian manifold with a constant negative sectional curvature $-1$: 
\begin{align}\label{eq1.1v1} 
i\partial_t u + \Delta_g u = |u|^{\frac{4}{d-2}} u, \quad (t,x) \in \mathbb{R} \times \mathbb{H}^d,
\end{align} 
where the unknown function $u : \mathbb{R} \times \mathbb{H}^d \to \mathbb{C}$. Here, $\Delta_g$ is the Laplace-Beltrami operator on $\mathbb{H}^d$. This equation is the Hamiltonian flow of the energy functional:
\begin{align}\label{eq1.1v2} 
M(u(t)) = \int_{\mathbb{H}^d} |u(t)|^2 \, d\mu, 
\quad E(u(t)) = \frac{1}{2} \int_{\mathbb{H}^d} |\nabla_g u(t)|^2 \, d\mu + \frac{d-2}{2d} \int_{\mathbb{H}^d} |u(t)|^{\frac{2d}{d-2}} \, d\mu.
\end{align} 
The nonlinear Schrödinger equation is an important model in physics, with rich applications in nonlinear optics, plasma physics, and quantum field theory. In particular, it describes phenomena such as the refraction of electromagnetic waves through media under certain field strengths, the propagation of electromagnetic fields in plasma, and condensation \cite{Akhmanov1966,Chiao1964,Kelley1965}. Additionally, the inviscid limit of the complex Ginzburg-Landau equation leads to the nonlinear Schrödinger equation. For further reading on the complex Ginzburg-Landau equation and nonlinear Schrödinger equation, see \cite{Cheng2026a,Cheng2025b,Cheng2025a,Cheng2025c,Li2025,Liu2025,Yang2018,Zhang2026}. 

The global well-posedness and scattering of the nonlinear Schrödinger equation on hyperbolic spaces have been extensively studied over the past two decades. J.-P. Anker and V. Pierfelice \cite{Anker2009}, V. Banica \cite{Banica2007}, and V. Banica, R. Carles, and G. Staffilani \cite{Banica2008} proved the local well-posedness and small data scattering for the nonlinear Schrödinger equation on $\mathbb{H}^d$. V. Banica, R. Carles, and T. Duyckaerts \cite{Banica2009a}, V. Banica and T. Duyckaerts \cite{Banica2007b,Banica2015}, and A. Ionescu and G. Staffilani \cite{Ionescu2009} consider
 the energy-subcritical defocusing/focusing case on $\mathbb{H}^d$ and established global well-posedness and scattering for the nonlinear Schrödinger equation. In \cite{Ionescu2012}, A. D. Ionescu, B. Pausader, and G. Staffilani proved global well-posedness and scattering for the energy-critical defocusing nonlinear Schrödinger equation on $\mathbb{H}^3$ using the concentration-compactness/rigidity argument. In this paper, we aim to extend their result to higher dimensions $d\geq4$ as follows:

\begin{theorem}\label{th1v1} 
\emph{(Global well-posedness and scattering.)} If $\phi \in H^1(\mathbb{H}^d)$, then there exists a unique global solution $u \in C(\mathbb{R}; H^1(\mathbb{H}^d))$ of the initial-value problem
\begin{align}\label{eq1.1v3} 
\begin{cases} 
i\partial_t u + \Delta_g u = |u|^{\frac{4}{d-2}} u, & (t,x) \in \mathbb{R} \times \mathbb{H}^d, \\ 
u(0) = \phi. 
\end{cases} 
\end{align} 
In addition, the mapping $\phi \to u$ is continuous from $H^1(\mathbb{H}^d)$ to $C(\mathbb{R}; H^1(\mathbb{H}^d))$, and the quantities $M(u)$ and $E(u)$ defined in \eqref{eq1.1v2} are conserved. 
	
Moreover, we have the bound
\begin{align}\label{eq1.1v4} 
\|u\|_{L^{\frac{2(d+2)}{d-2}}(\mathbb{R} \times \mathbb{H}^d)} \leq C(\|\phi\|_{H^1(\mathbb{H}^d)}). 
\end{align} 	
As a consequence, there exist unique \(u_\pm \in H^1(\mathbb{H}^d)\) such that
\begin{align}\label{eq1.1v5} 
\|u(t) - e^{it\Delta_g} u_\pm\|_{H^1(\mathbb{H}^d)} \to 0, \quad \text{as } t \to \pm\infty
\end{align} 
in hyperbolic space $\mathbb{H}^d$. 
\end{theorem}

\subsection{Methodology}

In Euclidean spaces, the concentration compactness/rigidity method has been developed for both subcritical and critical problems, for instance, \cite{Bourgain1999, Cazenave2003, Colliander2008, Grillakis2000, Kenig2006, Ryckman2007, Tao2006, Visan2007}. Subsequently, this method has been applied to other large-data critical dispersive problems in hyperbolic spaces, such as \cite{Anker2009, Cheng2025c}. We continue to adopt this approach. The proof strategy of the main theorem uses linear profile decomposition and the rigidity theorem. To prove linear profile decomposition, we need the heat kernel and to establish heat kernel approximation for general dimensions.  The heat kernels are much complicated due to different dimensions. Here we give a very comprehensive analysis for the asymptotics of the heat kernels on $\mathbb{H}^d$ as the limit $k \to \infty$ for the profile scale $N_k$, compared to the heat kernels on $\mathbb{R}^d$. 

For critical problems, the global geometry of the manifold cannot provide any improvements. To illustrate this, consider data of the form
\begin{align}\label{eq1.1v6} 
\phi_N(x) = N^{\frac{d-2}{2}} \phi(N\Psi_I^{-1}(x)), 
\end{align} 
where $\phi \in C_0^\infty(\mathbb{R}^d)$ is fixed. Assuming $\phi$ is fixed and letting $N \to \infty$, the functions $\phi_N \in C_0^\infty(\mathbb{H}^d)$ have uniformly bounded $H^1$ norms.

Let $v$ be the solution on $(-T, T)$ of the Euclidean nonlinear Schr\"odinger equation(NLS)
\begin{align}\label{eq1.1v8} 
\begin{cases} 
i \partial_t v + \Delta v = |v|^{\frac{4}{d-2}} v, & (t, x) \in \mathbb{R} \times \mathbb{R}^d, \\ 
v(0) = \phi. 
\end{cases} 
\end{align} 
For any fixed $T \geq 0$, we can prove that $v$ corresponding to initial data $\phi_N$ is well approximated by
\begin{align}\label{eq1.1v7} 
N^{\frac{d-2}{2}} v(N\Psi_I^{-1}(x), N^2 t) 
\end{align} 
on the time interval $(-TN^{-2}, TN^{-2})$. In other words, the solution of the hyperbolic NLS \eqref{eq1.1v3} with data $\phi_N$ can be regular on the interval $(-TN^{-2}, TN^{-2})$ only if the solution of the Euclidean NLS \eqref{eq1.1v8} is regular on $(-T, T)$. This shows that understanding the Euclidean scale-invariant problem is a prerequisite for understanding the problem on any other manifold. Fortunately, we can use the main theorem of E. Ryckman and M. Visan \cite{Ryckman2007, Visan2007} as a black box.

The structure of this paper is as follows. Section \ref{sec:BasLWP} introduces the notation and summarizes the key dispersive estimates for the linear Schr\"odinger flow on hyperbolic spaces; several auxiliary lemmas used later are also established there. In Section \ref{sec:Eucapp}, we study nonlinear solutions of \eqref{eq1.1v3} whose initial data are concentrated at a point as in \eqref{eq1.1v7}. By applying the main results of \cite{Ryckman2007, Visan2007}, we prove that such solutions exist globally in time and satisfy the corresponding dispersive bounds.  Section \ref{sec:Prodechs} provides the core profile decomposition for $H^1$-bounded sequences on hyperbolic spaces. There we distinguish two types of profiles: Euclidean profiles, which may concentrate after time and space translations, and hyperbolic profiles, which essentially live at frequency $N=1$; hyperbolic geometry ensures that low‑frequency profiles ($N\ll 1$) can be treated as perturbations. Finally, Section \ref{sec:ProofMainThm} completes the proof of Proposition \ref{pr.2v8} by combining the profile decomposition with orthogonality arguments.



\section{Basic lemmas and local well-posedness results}\label{sec:BasLWP}

In this section, we first recall some geometric and harmonic analysis aspects of hyperbolic space, which can be found in \cite{Bray1994, Helgason1994}. Then, we will revisit the Strichartz estimates and present the local well-posedness of the nonlinear Schr\"odinger equation (NLS) on $\mathbb{H}^d$.

\subsection{Basic harmonic analysis on $\mathbb{H}^d$}
\textbf{Hyperbolic Space and Isometries.} 
For integers $d \geq 2$, consider the Minkowski space $\mathbb{R}^{d+1}$ with the standard Minkowski metric
$$-(dx_0)^2 + (dx_1)^2 + \cdots + (dx_d)^2,$$ 
and define the bilinear form
$$[x, y] = x_0y_0 - x_1y_1 - \cdots - x_dy_d.$$
The hyperbolic space $\mathbb{H}^d$ is defined as
$$\mathbb{H}^d = \{x \in \mathbb{R}^{d+1} : [x, x] = 1,\; x_0 > 0\}.$$
Denote the origin by $\mathbf{0} = (1,0,\ldots,0)$. The Minkowski metric induces a Riemannian metric $g$ on $\mathbb{H}^d$. Let $D$ be the covariant derivative, and $\mathrm{d}\mu$ be the induced measure. 

Let $\mathbb{G} := SO(d,1) = SO_e(d,1)$ be the connected Lie group of $(d+1) \times (d+1)$ matrices preserving $[\cdot, \cdot]$. A matrix $X$ belongs to $SO(d,1)$ if
$$X^T \cdot I_{d,1} \cdot X = I_{d,1},\quad \det X = 1,\quad X_{00} > 0,$$ 
where $I_{d,1} = \operatorname{diag}[-1, 1, \ldots, 1]$ (such that $[x,y] = -x^T I_{d,1} y$). Let $\mathbb{K} = SO(d)$ be the subgroup fixing the origin $\mathbf{0}$, which acts as compact rotations on $(x_1,\ldots,x_d)$. Define the commutative subgroup
\[\mathbb{A} := \left\{ a_s = \begin{pmatrix}
	\cosh s & \sinh s & 0\\
	\sinh s & \cosh s & 0\\
	0 & 0 & I_{d-1}
\end{pmatrix} : s \in \mathbb{R} \right\},\]
and the Cartan decomposition
\[\mathbb{G} = \mathbb{K} \mathbb{A}_+ \mathbb{K} ,\qquad \mathbb{A}_+ := \{a_s : s \ge 0\}.\]
The group $\mathbb{G}$ acts transitively on $\mathbb{H}^d$; indeed, $\mathbb{H}^d \simeq \mathbb{G}/\mathbb{K}$. For any $h \in \mathbb{G}$, the map $L_h : \mathbb{H}^d \to \mathbb{H}^d$ defined by $L_h(x) = h \cdot x$ is an isometry. This induces an isometry on functions
$$\pi_h : L^2(\mathbb{H}^d) \to L^2(\mathbb{H}^d),\quad \pi_h(f)(x) = f(h^{-1}x). $$ 
For any $h \in \mathbb{G}$, define the diffeomorphism
\begin{align}\label{diff.2v4} 
\Psi_I : \mathbb{R}^d \to \mathbb{H}^d,\quad \Psi_I(v^1,\ldots,v^d) = h \cdot \bigl(\sqrt{1+|v|^2}, v^1, \ldots, v^d\bigr).
\end{align} 
Using this, we set
$$\tilde{\pi}_h : C(\mathbb{R}^d) \to C(\mathbb{H}^d),\quad \tilde{\pi}_h(f)(x) = f(\Psi_I^{-1}(x)). $$ 
The integration formula
$$\int_{\mathbb{H}^d} f(x)\, d\mu(x) = \int_{\mathbb{R}^d} f(\Psi_I(v)) (1+\vert{}v\vert{}^2)^{-\frac{1}{2}}\, dv,\qquad f\in C_0(\mathbb{H}^d),$$ 
holds.\\
\textbf{Fourier Analysis on $\mathbb{H}^d$.} 
The Fourier transform maps functions on $\mathbb{H}^d$ to functions on $\mathbb{R}_+ \times \mathbb{S}^{d-1}$. For $\omega \in \mathbb{S}^{d-1}$ and $\lambda \in \mathbb{C}$, set
$$b(\omega) = (1, \omega) \in \mathbb{R}^{d+1},\qquad h_{\lambda,\omega}(x) = [x, b(\omega)]^{i\lambda - \frac{d-1}{2}}. $$ 
These are eigenfunctions of the Laplace–Beltrami operator
$$\Delta_g h_{\lambda,\omega} = -(\lambda^2 + \rho^2)h_{\lambda,\omega},\quad \rho = \frac{d-1}{2}. $$ 
For $f \in C_0(\mathbb{H}^d)$, define the Fourier transform
$$
\hat{f}(\lambda,\omega) = \int_{\mathbb{H}^d} f(x) h_{\lambda,\omega}(x)\, d\mu = \int_{\mathbb{H}^d} f(x)[x, b(\omega)]^{-i\lambda - \frac{d-1}{2}}\, d\mu. 
$$ 
The inversion formula states that for $f \in C_0^\infty(\mathbb{H}^d)$
$$
f(x) = \int_0^\infty \int_{\mathbb{S}^{d-1}} \hat{f}(\lambda,\omega)[x, b(\omega)]^{-i\lambda - \frac{d-1}{2}} \vert{}c(\lambda)\vert{}^{-2}\, d\lambda\, d\omega, 
$$ 
where the Harish–Chandra $c$-function is given by
$$
c(\lambda) = C\,\frac{\Gamma(i\lambda)}{\Gamma\bigl(\frac{d-1}{2}+i\lambda\bigr)} 
$$ 
with a suitable constant $C$, and the measure on $\mathbb{S}^{d-1}$ is normalized to $1$. The Fourier transform diagonalizes $\Delta_g$ according to
\begin{align}\label{eq1.2v10} 
\widehat{\Delta_g f}(\lambda,\omega) = -\bigl(\lambda^2 + \rho^2\bigr)\hat{f}(\lambda,\omega).
\end{align} 
The Plancherel identity states that the map $f \mapsto \hat{f}$ extends to an isometry from $L^2(\mathbb{H}^d)$ onto $L^2(\mathbb{R}_+ \times \mathbb{S}^{d-1}, \vert{}c(\lambda)\vert{}^{-2} d\lambda\, d\omega)$. Moreover,
$$
\int_{\mathbb{H}^d} f_1(x)\overline{f_2(x)}\, d\mu = \frac{1}{2} \iint_{\mathbb{R}_+ \times \mathbb{S}^{d-1}} \hat{f}_1(\lambda,\omega)\overline{\hat{f}_2(\lambda,\omega)}\, \vert{}c(\lambda)\vert{}^{-2}\, d\lambda\, d\omega. 
$$ 
\textbf{Multiplier Operators and $L^p$ Boundedness.} 
A bounded multiplier $m : \mathbb{R}_+ \to \mathbb{C}$ defines a bounded operator $T_m$ on $L^2(\mathbb{H}^d)$ by
$$
\widehat{T_m f}(\lambda,\omega) = m(\lambda)\hat{f}(\lambda,\omega).
$$ 
For $p \neq 2$, the question of $L^p$ boundedness is more delicate. A necessary condition for $T_m$ to be bounded on $L^p(\mathbb{H}^d)$ is that $m$ extends to an even analytic function in the interior of the strip
$$
\mathcal{F}_p = \{\lambda \in \mathbb{C} : \vert{}\Im\lambda\vert{} < \vert{}2/p-1\vert{}/\rho\}, 
$$ 
(see \cite{Clerc1974}). Conversely, if $p \in (1,\infty)$ and $m : \mathcal{F}_p \to \mathbb{C}$ is an even analytic function satisfying the symbol-type bounds
\begin{align}\label{eq1.2v13} 
	|\partial^\alpha m(\lambda)| \leq C(1+|\lambda|)^{-\alpha} \quad \text{for all }\alpha\in[0,d+2]\cap\mathbb{Z},\;\lambda\in\mathcal{F}_p, 
\end{align} 
then $T_m$ extends to a bounded operator on $L^p(\mathbb{H}^d)$ (\cite{Stanton1978}).\\ 
\textbf{Convolution and Radial Functions.} 
Normalize the Haar measures on $\mathbb{K}$ and $\mathbb{G}$ so that $\int_{\mathbb{K}} 1\, dk = 1$ and
$$
\int_{\mathbb{G}} f(g\cdot \mathbf{0})\, dg = \int_{\mathbb{H}^d} f(x)\, d\mu,\qquad f\in C_0(\mathbb{H}^d).
$$ 
For $f_1, f_2 \in C_0(\mathbb{G})$, define convolution
$$
(f_1 * f_2)(h) = \int_{\mathbb{G}} f_1(g) f_2(g^{-1}h)\, dg. 
$$ 
A function $K:\mathbb{G}\to\mathbb{C}$ is $\mathbb{K}$-biinvariant if $K(k_1 g k_2) = K(g)$ for all $k_1, k_2 \in \mathbb{K}$. A function $K:\mathbb{H}^d\to\mathbb{C}$ is $\mathbb{K}$-invariant if $K(k\cdot x) = K(x)$ for all $k \in \mathbb{K}, x \in \mathbb{H}^d$. For $f, K \in C_0(\mathbb{H}^d)$ with $K$ radial, define
$$
(f*K)(x) = \int_{\mathbb{G}} f(g\cdot \mathbf{0})\, K(g^{-1}\cdot x)\, dg.
$$ 
For radial $K$, the Fourier transform reduces to a spherical transform
\begin{align}\label{eq1.2v17} 
	\tilde{K}(\lambda,\omega) = \tilde{K}(\lambda) = \int_{\mathbb{H}^d} K(x)\Phi_{-\lambda}(x)\, d\mu,
\end{align} 
where
$$
\Phi_\lambda(x) = \int_{\mathbb{S}^{d-1}} [x, b(\omega)]^{-i\lambda - \frac{d-1}{2}}\, d\omega
$$ 
is the elementary spherical function. The inversion formula becomes
$$
K(x) = \int_0^\infty \tilde{K}(\lambda)\Phi_\lambda(x)\, \vert{}c(\lambda)\vert{}^{-2}\, d\lambda. 
$$ 
The convolution identity holds
$$
\widehat{f*K}(\lambda,\omega) = \hat{f}(\lambda,\omega)\cdot \tilde{K}(\lambda).
$$ 
\textbf{Sobolev Spaces.} 
For $s \in \mathbb{C}$, define the operator $(-\Delta)^{s/2}$ via the Fourier multiplier $\lambda \mapsto (\lambda^2+\rho^2)^{s/2}$. For $p \in (1,\infty)$ and $s \in \mathbb{R}$, the Sobolev space $W^{s,p}(\mathbb{H}^d)$ is the closure of $C_0^\infty(\mathbb{H}^d)$ under the norm
$$
\|f\|_{W^{s,p}(\mathbb{H}^d)} = \|(-\Delta)^{s/2}f\|_{L^p(\mathbb{H}^d)}. 
$$ 
Set $H^s = W^{s,2}$. Because the symbol of $(-\Delta)^{s/2}$ satisfies \eqref{eq1.2v13}, this definition coincides with the usual one on Riemannian manifolds. For $p \in (1,\infty)$ and $s = 1$, we have
$$
\|f\|_{W^{1,p}(\mathbb{H}^d)} \approx_p \Bigl( \int_{\mathbb{H}^d} \vert{}\nabla_g f\vert{}^p\, d\mu \Bigr)^{1/p}, 
$$ 
where $\vert{}\nabla_g f\vert{} = (D^\alpha f D_\alpha \bar{f})^{1/2}$. The Sobolev embedding theorem holds
$$
W^{s,p} \hookrightarrow L^q\quad \text{if } 1 < p \leq q < \infty,\; s = \frac{d}{p} - \frac{d}{q}. 
$$ 
\textbf{Notation.} 
Let $\mathbb{N}$ denote the natural numbers. For a time interval $I\subset\mathbb{R}$, define
$$
S^0(I) := L_t^\infty L_x^2 \cap L_t^2 L_x^{\frac{2d}{d-2}}(I\times\mathbb{H}^d),\quad
Z(I) := L_{t,x}^{\frac{2(d+2)}{d-2}}(I\times\mathbb{H}^d), 
$$ 
and for $s \in \mathbb{R}$
$$
S^s(I) := \bigl\{ u : (-\Delta_g)^{s/2}u \in S^0(I) \bigr\}. 
$$ 
Let $N^0(I)$ be the dual Strichartz space, and
$$
N^s(I) := \bigl\{ F : (-\Delta_g)^{s/2}F \in N^0(I) \bigr\}. 
$$ 
\textbf{A Substitute for Littlewood-Paley Projections.} 
For $N > 0$, define the operator $P_N : L^2(\mathbb{H}^d)\to L^2(\mathbb{H}^d)$ by
$$
P_N := N^{-2}\Delta_g e^{N^{-2}\Delta_g}, 
$$ 
with Fourier symbol
$$
\widehat{P_N f}(\lambda,\omega) = -N^{-2}(\lambda^2+\rho^2)e^{-N^{-2}(\lambda^2+\rho^2)}\hat{f}(\lambda,\omega). 
$$ 
Since only real-analytic multipliers are $L^p$-bounded on $\mathbb{H}^d$ for $p \neq 2$ \cite{Clerc1974}, $P_N$ serves as a replacement for Euclidean Littlewood-Paley projections (localizing to frequencies $\sim N$). By Fourier inversion
\[
P_N f(x) = \int_{\mathbb{H}^d} f(y)\, P_N(d(x,y))\, d\mu(y),
\]
with 
\begin{align}\label{Fourier}
	\left|P_N(r)\right| \lesssim N^d (1+Nr)^{-\frac{d+2}{d-2}} e^{-\frac{d+1}{d-2}r}.
\end{align}

According to \cite{GN}, the explicit form of the heat kernel is as follows 
\begin{theorem}\label{th.2v1}
	Assume $m \in \mathbb{N}$. Then the heat kernel $p_t(x,y)$ is given by the following formulas.
	\begin{enumerate}
		\item[(1)] When $d = 2m + 1$, we have
		\begin{align*}
			p_t(x,y) = \frac{(-1)^m}{2^m \pi^m} \frac{1}{(4\pi t)^{\frac{1}{2}}} \left( \frac{1}{\sinh r} \frac{\partial}{\partial r} \right)^m e^{-mt-\frac{r^2}{4t}} .
		\end{align*}
		\item[(2)] When $d = 2m + 2$, we have
		\begin{align*}
			p_t(x,y) = \frac{(-1)^m}{2^{m+\frac{5}{2}}\pi^{m+\frac{3}{2}}} t^{-\frac{3}{2}} e^{-\frac{(2m+1)^2 t}{4}} \left( \frac{1}{\sinh r} \frac{\partial}{\partial r} \right)^m \int_r^\infty \frac{se^{-\frac{s^2}{4t}}}{(\cosh s - \cosh r)^{\frac{1}{2}}} \, ds.
		\end{align*}
	\end{enumerate}
\end{theorem}
\subsection{Strichartz estimates and local well-posedness theory.} \textbf{Dispersive Estimates.} 
For $\phi \in H^s(\mathbb{H}^d)$, define the free Schrödinger evolution by
$$
\widehat{e^{it\Delta_g}\phi}(\lambda,\omega) = \hat{\phi}(\lambda,\omega)\, e^{-it(\lambda^2+\rho^2)}. 
$$ 
The following dispersive estimate is essential (see \cite{Anker2009, Banica2007, Banica2008, Ionescu2009, Pierfelice2008}):
\begin{align}\label{dispersive-e} 
	\|e^{it\Delta_g}\|_{L^p\to L^{p'}} \lesssim |t|^{-d(1/p-1/2)},\qquad p \in \Bigl[\frac{2d}{d+2}, 2\Bigr],\; p' = \frac{p}{p-1}, 
\end{align} 
for any $t \neq 0$.
Then, by the argument in \cite{Keel1998,Miao2004,Miao2025}, we obtain the Strichartz estimate for NLS:

\begin{proposition}[Strichartz estimates]
	Assume that $d \geq 4$ and $I = (a, b) \subseteq \mathbb{R}$ is an interval.
	\begin{enumerate}
		\item[(1)] If $\phi \in L^2(\mathbb{H}^d)$, then
		\begin{align*}
			\|e^{it\Delta_g} \phi\|_{L^\infty_t L^2_x \cap L^2_t L^{\frac{2d}{d-2}}_x(I \times \mathbb{H}^d)} \lesssim \|\phi\|_{L^2}.
		\end{align*}
		\item[(2)] If $F \in L^1_t L^2_x + L^2_t L^{\frac{2d}{d+2}}_x(I \times \mathbb{H}^d)$, then
		\begin{align*}
			\left\| \int_a^t e^{i(t-s)\Delta_g} F(s) \, ds \right\|_{L^\infty_t L^2_x \cap L^2_t L^{\frac{2d}{d-2}}_x(I \times \mathbb{H}^d)} \lesssim \|F\|_{L^1_t L^2_x + L^2_t L^{\frac{2d}{d+2}}_x(I \times \mathbb{H}^d)}.
		\end{align*}
	\end{enumerate}
\end{proposition}

Next, by the decay property of Littlewood-Paley multiplier  \eqref{Fourier}, we obtain the following property.
\begin{lemma}\label{le.2v3}
	\begin{enumerate}
		\item[(1)] Let $\varepsilon \in (0, 1]$, $R_\varepsilon \geq 1$, for each $x \in \mathbb{H}^d$, $N \geq 1$,
		\begin{align*}
			|P_N f(x)| \lesssim N^{\frac{d-2}{2}} \left( \|f \cdot \mathbf{1}_{B(x,R_\varepsilon N^{-1})}\|_{L^{\frac{2d}{d-2}}(\mathbb{H}^d)} + \varepsilon \|f\|_{L^{\frac{2d}{d-2}}(\mathbb{H}^d)} \right).
		\end{align*}
		\item[(2)] If $f \in H^1(\mathbb{H}^d)$, then
		\begin{align*}
			\|f\|_{L^{\frac{2d}{d-2}}(\mathbb{H}^d)} \lesssim \|\nabla_g f\|_{L^2(\mathbb{H}^d)}^{\frac{2}{d}} \sup_{N \geq 1,\, x \in \mathbb{H}^d} \left[ N^{-\frac{d-2}{2}} \left|P_N f(x)\right| \right]^{\frac{d-2}{d}}.
		\end{align*}
	\end{enumerate}
\end{lemma}

\begin{proof}
Although \cite{Ionescu2012} only provides the proof for three dimensions, it can be extended to higher dimensions in a similar manner. 
\end{proof}

Next, we prove a Kato's type local smoothing estimate, which plays an important role in the proof of Proposition \ref{pr.2v8}.

\begin{lemma}[Local smoothing estimate]\label{le.2v4}
	Suppose that $\psi \in H^1(\mathbb{H}^d)$, for $R, T > 0$, and $\theta \in (0, 1]$, then
	\begin{align}\label{eq1.2v1}
		\| \nabla_g e^{it\Delta_g} \psi \|_{L^2_{t,x}([0,T] \times B(x_0,R))} \leq C(R, T, \theta) \|e^{it\Delta_g} \psi\|_{L^2([0,T])}^\frac{\theta}{1+2\theta} \|\psi\|_{H^1}^\frac{1+\theta}{1+2\theta} 
	\end{align}
	uniformly in $\psi$, where
	\begin{align*}
		C(R, T, \theta) = T^{\frac{2}{(d+2)(1+2\theta)}}R^{\frac{d(1+\theta)-2(1-\theta)+\theta(d+2)}{(d+2)(1+2\theta)}} + (1+R^3T) \left( T^{\frac{1}{d+2}} L^{\frac{2d(1+2\theta)-2(3-2\theta)}{4(d+2)}} \right)^{\frac{1}{1+2\theta}}, \quad \theta \in (0, 1].
	\end{align*}
\end{lemma}

\begin{proof}
	\text{Step 1.} The low frequency case. Let $L > 0$, it follows from H\"older's, Bernstein's (for instance, see \cite{Ivanovici2017,Klainerman2006}), and Strichartz inequalities that
	\[\begin{aligned}
		& \| \nabla (e^{it\Delta_g} P_L \psi) \|_{L^2_{t,x}([0,T] \times B(x_0,R))} \\
		 \leq &T^{\frac{2d(1+\theta)-4(1-\theta)}{2(d+2)}} \| \nabla (e^{it\Delta_g} P_L \psi) \|_{L^2_{t,x}([0,T])}^{\frac{2(d+2)}{d-2}} L^{\frac{2d(d+2)}{d^2+4-2d\theta-4\theta}} \\
		  \leq& T^{\frac{2d(1+\theta)-4(1-\theta)}{2(d+2)}} L^{\theta} \|(-\Delta_g)^{(1-\theta)/2} e^{it\Delta_g} P_L \psi \|_{L^2_{t,x}([0,T])}^{\frac{2(d+2)}{d-2}} L^{\frac{2d(d+2)}{d^2+4-2d\theta-4\theta}} \\
		 \leq &T^{\frac{2d(1+\theta)-4(1-\theta)}{2(d+2)}} L^{\theta} \|e^{it\Delta_g} \psi \|_{Z([0,T])}^{\theta} \|(-\Delta_g)^{1/2} e^{it\Delta_g} \psi \|_{L^2_{t,x}([0,T])}^{1-\theta} L^{\frac{2d(d+2)}{d^2+4}} \\
		 \leq& T^{\frac{2d(1+\theta)-4(1-\theta)}{2(d+2)}} L^{\theta} \|e^{it\Delta_g} \psi \|_{Z([0,T])}^{\theta} \|\psi\|_{H^1(\mathbb{H}^d)}^{1-\theta}.
	\end{aligned}
	\]
	\text{Step 2.} The high frequency case. Let $K \geq L$, we show that
	\begin{align}\label{eq1.2v2}
		\left\| \nabla_g \left(P_L e^{it\Delta_g} \psi\right) \right\|_{L^2_{t,x}([0,T] \times B(x_0,R))} \lesssim \frac{1+R^3T}{\sqrt{L}} \|\psi\|_{H^1(\mathbb{H}^d)}.
	\end{align}
	There exists a real-valued function \(a = a_R \in C^\infty(\mathbb{H}^d)\) that satisfies
	\begin{align}\label{eq1.2v3}
		&|D^\alpha a D_\alpha a| \lesssim 1 \text{ in } \mathbb{H}^d, \quad |\Delta_g^2 a| \lesssim R^3 \text{ in } \mathbb{H}^d, \nonumber\\
		&X^\alpha X_\alpha \cdot R\mathbf{1}_{B(x_0,R)} \lesssim X^\alpha X^\beta D_\alpha D_\beta a \text{ in } \mathbb{H}^d \text{ for any vector-field } X \in T(\mathbb{H}^d).
	\end{align}
	Next, let $u := P_K(e^{it\Delta_g} \psi), \text{ where } K \geq L \geq 1$ and define the Morawetz action
	\begin{align*}
		M_a(t) = 2\Im \int_{\mathbb{H}^d} D^\alpha a(x)\,\overline{u(x)}\, D_\alpha u(x) \, d\mu(x).
	\end{align*}
	After differentiation, we obtain
	\begin{align*}
		M'_a(t) = -\int_{\mathbb{H}^d} \Delta_g^2 a\,|u|^2 \, d\mu + 4\Re \int_{\mathbb{H}^d} D^\beta D^\alpha a\, D_\beta \overline{u}\, D_\alpha u \, d\mu.
	\end{align*}
Thus, by integrating over the time interval \([-T, T]\) and using the first two properties from \eqref{eq1.2v3}, we have
\begin{align*}
 4 \int_{-T}^T \int_{\mathbb{H}^d} \Re\bigl(D^\alpha D^\beta a \cdot D_\alpha u \,\overline{D_\beta u}\bigr) \, d\mu \, dt  \leq& 2 \sup_{|t| \leq T} |M_a(t)| + \int_{-T}^T \int_{\mathbb{H}^d} |\Delta_g(a)| \cdot |u|^2 \, d\mu \, dt \\
 \lesssim& \int_{\mathbb{H}^d} |D^\alpha a|\,|\overline{u}|\,|D_\alpha u| \, d\mu + R^3 \int_{-T}^T \|u(t)\|_{L^2(\mathbb{H}^d)}^2 \, dt \\
\lesssim& \sup_{|t| \leq T} \|u(t)\|_{L^2(\mathbb{H}^d)} \|u(t)\|_{H^1(\mathbb{H}^d)} + R^3 \int_{-T}^T \|u(t)\|_{L^2(\mathbb{H}^d)}^2 \, dt \\
 \lesssim& K^{-1} \|\psi\|_{H^1(\mathbb{H}^d)}^2 + R^3 T K^{-1} \|\psi\|_{H^1(\mathbb{H}^d)}^2.
	\end{align*}
	It follows from Plancherel's identity and techniques of Fourier analysis in $H^1(\mathbb{H}^d)$ that
	\begin{align*}
		\|u(t)\|_{H^1(\mathbb{H}^d)} \lesssim \|P_K e^{it\Delta_g} \psi\|_{H^1(\mathbb{H}^d)}.
	\end{align*}
Moreover, applying Bernstein's inequality, we have
	\begin{align*}
		\|P_K e^{it\Delta_g} \psi\|_{H^1(\mathbb{H}^d)}^2 \lesssim  \|\psi\|_{H^1(\mathbb{H}^d)}^2.
	\end{align*}
	Given that $a$ is real valued, the assumption $K \geq L$ holds, and using the inequality in the last line of \eqref{eq1.2v3}, we obtain the desired bound \eqref{eq1.2v2}. Finally, the estimate is achieved by setting
	\begin{align*}
		L = \left(T^{\frac{2}{d+2}} R^{\frac{d(1+2\theta)-2(3-2\theta)}{2(d+2)}} \|e^{it\Delta_g} \psi\|_{Z([0,T])}^{\theta} \|\psi\|_{H^1(\mathbb{H}^d)}^{-\theta}\right)^{-\frac{1}{1+2\theta}}.
	\end{align*}
The remainder of the proof is similar to Lemma 2.3 in \cite{Ionescu2012}, so we omit the details here.
\end{proof}
Below, we give the theory of local well-posedness and stability, the global Morawetz inequality, and a compactness argument that is used in the proof of Theorem \ref{th1v1} in Section 5. Suppose that $\mathbb{R}$ is an open interval, then
\begin{align*}
	\mathcal{P} = \{(I, u) : I \subseteq \mathbb{R},\; u \in C(I : H^1(\mathbb{H}^d))\}.
\end{align*}
Moreover, when $I \subseteq I'$, $u'(t) = u(t)$ for any $t \in I$,
\begin{align*}
	(I, u) \leq (I', u').
\end{align*}

By the Strichartz estimate, we easily obtain the local well-posedness and stability theory as follows:
\begin{proposition}[Local well-posedness]\label{pr.2v5}
	\label{prop:local}
	Suppose that $\phi \in H^1(\mathbb{H}^d)$. Then the initial-value problem
	\begin{align}\label{pr.2v51}
		\begin{cases}
			(i\partial_t + \Delta_g)u = u|u|^{\frac{4}{d-2}}, & (t,x) \in I \times \mathbb{H}^d, \\
			u(0) = \phi
		\end{cases} 
	\end{align}
	has a unique maximal solution $(I,u) = (I(\phi), u(\phi)) \in \mathcal{P}$ with $0 \in I$. On the interval $I$, both the mass $M(u)$ and the energy $E(u)$ as defined in \eqref{eq1.1v2} remain constant. Furthermore, for every compact subinterval $J \subseteq I$, we have $\|u\|_{S^1(J)} < \infty$. Moreover, if $I_+ := I \cap [0, \infty)$ is bounded, then $\|u\|_{Z(I_+)} = \infty$, and if $I_- := I \cap (-\infty, 0]$ is bounded, then $\|u\|_{Z(I_-)} = \infty$.
	
	That is, a local-in-time solution exists and can be continued as a strong solution unless its spacetime norm blows up.
\end{proposition}

\begin{proposition}[Stability]\label{pr.2v6}
	\label{prop:stability}
	Suppose that $I \subseteq \mathbb{R}$ is an open interval, $\rho \in \{0, 1\}$ and $\tilde{u} \in C(I; H^1(\mathbb{H}^d))$ satisfies the approximate NLS
	\begin{align*}
		i\partial_t \tilde{u} + \Delta_g \tilde{u} = \rho |\tilde{u}|^{\frac{4}{d-2}} \tilde{u} + e \quad \text{in } I \times \mathbb{H}^d.
	\end{align*}
	Moreover
	\begin{align}\label{pr.2v61}
		\|\tilde{u}\|_{L_{t,x}^{\frac{2(d+2)}{d-2}}} + \sup_{t \in I} \|\tilde{u}(t)\|_{H^1(\mathbb{H}^d)} \leq M, 
	\end{align}
	for some $M > 0$.
	For $0 < \epsilon < \epsilon_1$, $\epsilon_1 = \epsilon_1(M) > 0$ is a small constant. Let $t_0 \in I$ and $u(t_0) \in H^1(\mathbb{H}^d)$ be such that the smallness condition
	\begin{align}\label{pr.2v62}
		\|u(t_0) - \tilde{u}(t_0)\|_{H^1(\mathbb{H}^d)} + \|e\|_{N^1(I)} \leq \epsilon 
	\end{align}
	holds. Then there exists a solution $u \in C(I; H^1(\mathbb{H}^d))$ of the NLS
	\begin{align*}
		i\partial_t u + \Delta_g u = \rho |u|^{\frac{4}{d-2}} u \quad \text{in } I \times \mathbb{H}^d,
	\end{align*}
	and
	\begin{align*}
		\|u\|_{S^1(I)} + \|\tilde{u}\|_{S^1(I)} \leq C(M), \qquad \|u - \tilde{u}\|_{S^1(I)} \leq C(M)\epsilon.
	\end{align*}
\end{proposition}
For $4 \le d \le 6$, the stability can be obtained by the usual critical iteration. When $d\ge 7$, we can use the arguments similar to the exotic Strichartz estimates obtained by Tao and Visan \cite{TV2007} to close Proposition \ref{pr.2v6}.

Next, we will state the global Morawetz estimate which was shown in  \cite{Ionescu2009} for sub-critical case. After a careful checking, we believe that the proof is still carried out for critical setting. So we omit the details here.

\begin{proposition}[Global Morawetz estimate, \cite{Ionescu2009}]\label{pr.2v7}
	Let $I \subseteq \mathbb{R}$ be an open interval, then
	\begin{align*}
		i\partial_t u + \Delta_g u = |u|^{\frac{4}{d-2}} u \quad \text{in } I \times \mathbb{H}^d,
	\end{align*}
	admits a unique solution $u \in C(I: H^1(\mathbb{H}^d))$. In addition
	\begin{align}\label{pr.2v71}
		\|u\|_{L^{\frac{2d}{d-2}}([t_1, t_2]\times \mathbb{H}^d)}^{\frac{2d}{d-2}} \lesssim \sup_{t \in [t_1, t_2]} \|u(t)\|_{L^2(\mathbb{H}^d)} \|u(t)\|_{H^1(\mathbb{H}^d)},
	\end{align}
	where $t_1, t_2\in I$.
\end{proposition}

Now, let us recall the conserved energy $E(u)$ from \eqref{eq1.1v2}. For each $E \in [0, \infty)$, we set
\begin{align*}
	S(E) = \sup\{ \|u\|_{Z(I)} : E(u) \leq E \},
\end{align*}
where the supremum is taken over all solutions $u \in C(I : H^1(\mathbb{H}^d))$ defined on some interval $I$ with energy not exceeding $E$. We also define
\begin{align*}
	E_{\max} = \sup\{ E : S(E) < \infty \}.
\end{align*}
By applying Proposition \ref{pr.2v6} with $\tilde{u} \equiv 0$, $e \equiv 0$, $I = \mathbb{R}$, $M = 1$, $\epsilon \ll 1$, it follows that $E_{\max} > 0$. Proposition \ref{pr.2v5} then implies that if $u$ solves \eqref{eq1.1v1} and $E(u) < E_{\max}$, then $u$ can be extended to a global solution that scatters.

If $E_{\max} = +\infty$, then Theorem \ref{th1v1} follows immediately from Propositions \ref{pr.2v5} and \ref{pr.2v6}. On the other hand, assuming $E_{\max} < +\infty$, there exists a sequence of solutions that satisfy the assumptions of the following key proposition to be proved later.

\begin{proposition}\label{pr.2v8}
	Consider a sequence of nonlinear solutions $u_k \in C((-T_k, T^k) : H^1(\mathbb{H}^d))$, $k = 1, 2, \dots$, to the nonlinear equation
	\begin{align*}
		i\partial_t u + \Delta_g u = |u|^{\frac{4}{d-2}} u,
	\end{align*}
	each defined on an open interval $(-T_k, T^k)$. Suppose that their energies satisfy $E(u_k) \to E_{\max}$. Choose a sequence of times $t_k \in (-T_k, T^k)$ such that
	\begin{align}\label{pr.2v81}
		\lim_{k \to \infty} \|u_k\|_{Z(-T_k, t_k)} = \lim_{k \to \infty} \|u_k\|_{Z(t_k, T^k)} = +\infty. 
	\end{align}
	Then, after possibly passing to a subsequence, there exist a function $w_0 \in H^1(\mathbb{H}^d)$ and a sequence of isometries $h_k \in \mathbb{G}$ for which
	\begin{align*}
		u_k(t_k, h_k^{-1} \cdot x) \to w_0(x)
	\end{align*}
	 strongly in $H^1(\mathbb{H}^d)$.
\end{proposition}

 We will gather the main tools in the next two sections and complete the proof of Proposition \ref{pr.2v8} in Section \ref{sec:ProofMainThm}.

\section{Euclidean Approximations}\label{sec:Eucapp}

This section is devoted to deriving precise estimates that allow us to compare, for both linear and nonlinear Schr\"odinger equations, solutions on Euclidean spaces with those on hyperbolic spaces. Since the Euclidean and hyperbolic geometries differ greatly at large scales, such a comparison makes sense only for rescaled data that are concentrated at a point.
We take a spherically symmetric function $\eta \in C_0^\infty(\mathbb{R}^d)$ supported in the ball of radius $2$ and identically equal to $1$ on the ball of radius $1$. For any $\phi \in \dot{H}^1(\mathbb{R}^d)$ and any real number $N \geq 1$, we set
\begin{equation}\label{eq1.3v1}
	\begin{aligned}
		 Q_N \phi \in C_0^\infty(\mathbb{R}^d), \quad& (Q_N \phi)(x) = \eta(x/N^{1/2}) \cdot (e^{\Delta/N} \phi)(x), \\
		 \phi_N \in C_0^\infty(\mathbb{R}^d),\quad &\phi_N(x) = N^{\frac{d-2}{2}} (Q_N \phi)(Nx),  \\
		 f_N \in C_0^\infty(\mathbb{R}^d),\quad &f_N(y) = \phi_N(\Psi_I^{-1}(y)), 
	\end{aligned}
\end{equation}
where $\Psi_I$ is defined in \eqref{diff.2v4}. Thus, $Q_N \phi$ is a regularized, compactly supported approximation of the profile $\phi$, $\phi_N$ is an $\dot{H}^1$-invariant rescaling of $Q_N \phi$, and $f_N$ is the function obtained by pulling $\phi_N$ back to a neighbourhood of $\mathbf{0}$ in $\mathbb{H}^d$. We also define
\begin{align*}
	E_{\mathbb{R}^d}(\phi) = \frac{1}{2} \int_{\mathbb{R}^d} |\nabla \phi|^2 \, dx + \frac{d-2}{2d} \int_{\mathbb{R}^d} |\phi|^{\frac{2d}{d-2}} \, dx.
\end{align*}

We shall employ the main result of \cite{Ryckman2007,Visan2007} in the following formulation.

\begin{theorem}\label{th.3v1}
	Let $\psi \in \dot{H}^1(\mathbb{R}^d)$. Then there exists a unique global solution $v \in C(\mathbb{R} : \dot{H}^1(\mathbb{R}^d))$ of the initial value problem
	\begin{align}\label{eq1.3v2}
		\begin{cases}
			i \partial_t v + \Delta v = |v|^{\frac{4}{d-2}} v, \\
			v(0) = \psi,
		\end{cases}
	\end{align}
	and it satisfies
	\begin{align}\label{eq1.3v3}
		\|\nabla v\|_{L_t^\infty L_x^2 \cap L_t^2 L_x^{\frac{2d}{d-2}}(\mathbb{R}^d \times \mathbb{R})} \leq \tilde{C}(E_{\mathbb{R}^d}(\psi)).
	\end{align}
	Moreover, the solution scatters, and there exist $\psi^{\pm\infty} \in \dot{H}^1(\mathbb{R}^d)$ such that
	\begin{align}\label{eq1.3v4}
		\|v(t) - e^{it\Delta} \psi^{\pm\infty}\|_{\dot{H}^1(\mathbb{R}^d)} \to 0 \quad \text{as } t \to \pm\infty.
	\end{align}
	In addition, if $\psi \in H^s(\mathbb{R}^d)$ with $s \geq 0$ when $d=3,4$, or   $0 \leq s \leq 1 + \frac{4}{d-2}$ when $d\geq5$, then
	\begin{align*}
		\sup_{t \in \mathbb{ R }} \|v(t)\|_{H^s(\mathbb{R}^d)} \lesssim _{\|\psi\|_{H^s}}1.
	\end{align*}
\end{theorem}

\begin{lemma}\label{le.3v2}
	For any compact time interval $I$, we assume $v$ is the solution of \eqref{eq1.3v2} that yields
	\begin{align*}
		\left\| e^{\Delta/R} |v|^{\frac{4}{d-2}} v - |e^{\Delta/R} v|^{\frac{4}{d-2}} e^{\Delta/R} v \right\|_{L_t^2 W_x^{1, \frac{2d}{d+2}}(I \times \mathbb{R}^d)} \to 0, \quad \text{as } R \to \infty.
	\end{align*}
\end{lemma}
\begin{proof}
	
The proof revolves around confirming the estimates
	\begin{align}\label{eq1.3v5}
		\lim_{R \to \infty} \|e^{\Delta/R} |v|^{\frac{4}{d-2}} v - |e^{\Delta/R} v|^{\frac{4}{d-2}} e^{\Delta/R} v\|_{L^2_t L_x^{\frac{2d}{d+2}}} = 0, 
	\end{align}
and
	\begin{align}\label{eq1.3v6}
		\lim_{R \to \infty} \left\| \nabla \left( e^{\Delta/R} |v|^{\frac{4}{d-2}} v - |e^{\Delta/R} v|^{\frac{4}{d-2}} e^{\Delta/R} v \right) \right\|_{L^2_t L_x^{\frac{2d}{d+2}}} = 0. 
	\end{align}
It's easy to see
	\begin{align}\label{eq1.3v7}
		&\left\| \nabla \left( e^{\Delta/R} |v|^{\frac{4}{d-2}} v - |e^{\Delta/R} v|^{\frac{4}{d-2}} e^{\Delta/R} v \right) \right\|_{L^2_t L_x^{\frac{2d}{d+2}}} \nonumber\\
		\leq &\left\| \nabla (1 - e^{\Delta/R})|v|^{\frac{4}{d-2}} v  \right\|_{L^2_t L_x^{\frac{2d}{d+2}}} + \left\| \nabla \left( |v|^{\frac{4}{d-2}} v - |e^{\Delta/R}v|^{\frac{4}{d-2}} e^{\Delta/R}v \right) \right\|_{L^2_t L_x^{\frac{2d}{d+2}}}.
	\end{align}
By \eqref{eq1.3v3}, it follows that
	\begin{align*}
		\|\nabla v\|_{S^0(\mathbb{R})} \leq C \bigl( \|\psi\|_{\dot{H}^1} \bigr),
	\end{align*}
which implies
	\begin{align*}
		\left\| \nabla |v|^{\frac{4}{d-2}} v \right\|_{L^2_t L_x^{\frac{2d}{d+2}}} \lesssim \left\| v \right\|_{L^{\frac{2(d+2)}{d-2}}_{t,x}}^\frac{4}{d-2} \left\| \nabla v \right\|_{L^{\frac{2(d+2)}{d-2}}_{t} L^{\frac{2d(d+2)}{d^2+4}}_{x}} \leq C \bigl( \left\| v(0) \right\|_{\dot{H}^1} \bigr),
	\end{align*}
This together with the dominated convergence theorem yields
	\begin{align*}
		\lim_{R \to \infty} \left\| \nabla (1 - e^{\Delta/R})|v|^\frac{4}{d-2}v \right\|_{L^2_t L_x^{\frac{2d}{d+2}}} = 0,
	\end{align*}
	and so the first term on the right side of \eqref{eq1.3v7} turns to zero.
	For $d \ge 6$, H\"older's inequality together with the dominated convergence theorem implies that the second term on the right side of \eqref{eq1.3v7} satisfies
	\begin{align}\label{eq1.3v8}
	&	\left\| \nabla \left( |v|^{\frac{4}{d-2}} v - |e^{\Delta/R}v|^{\frac{4}{d-2}} e^{\Delta/R}v \right) \right\|_{L^2_t L_x^{\frac{2d}{d+2}}} \nonumber\\
		 \lesssim &\left\| (1 - e^{\Delta/R})v \right\|_{L_{t,x}^{\frac{2(d+2)}{d-2}}}^{\frac{4}{d-2}} \left\| \nabla v \right\|_{L_t^{\frac{2(d+2)}{d}} L_x^{\frac{2d(d+2)}{d^2+4}}} \nonumber\\
		 &+ \left\| v \right\|_{L_{t,x}^{\frac{2(d+2)}{d-2}}}^{\frac{4}{d-2}} \left\| (1 - e^{\Delta/R})\nabla v \right\|_{L_t^{\frac{2(d+2)}{d}} L_x^{\frac{2d(d+2)}{d^2+4}}} \to 0, \text{ as } R \to \infty. 
	\end{align}	
	For $d = 4$ or $d = 5$, the first term in the second line of \eqref{eq1.3v8} is replaced by
\begin{align*}
	\left\| (1 - e^{\Delta/R})v \right\| \left( |v|^{\frac{6-d}{d-2}} + |e^{\Delta/R}v|^{\frac{6-d}{d-2}} \right) |\nabla u| \Big\|_{L_t^2 L_x^{\frac{2d}{d+2}}} \lesssim \left\| (1 - e^{\Delta/R})v \right\|_{L_{t,x}^{\frac{2(d+2)}{d-2}}} \|v\|_{L_{t,x}^{\frac{2(d+2)}{d-2}}}^{\frac{6-d}{d-2}} \|\nabla v\|_{L_t^{\frac{2(d+2)}{d}} L_x^{\frac{2d(d+2)}{d^2+4}}}.
\end{align*}
Consequently, the dominated convergence theorem implies the convergence of this term to 0, thereby proving \eqref{eq1.3v6}. A similar argument gives \eqref{eq1.3v5}. Indeed, we have
\begin{align*}
		&\left\| e^{\Delta/R} |v|^{\frac{4}{d-2}} v - |e^{\Delta/R}v|^{\frac{4}{d-2}} e^{\Delta/R}v \right\|_{L_t^2 L_x^{\frac{2d}{d+2}}} \\
		\lesssim& \left\| (1 - e^{\Delta/R})|v|^{\frac{4}{d-2}}v \right\|_{L_t^2 L_x^{\frac{2d}{d+2}}} + \left\| |v|^{\frac{4}{d-2}} v - |e^{\Delta/R}v|^{\frac{4}{d-2}} e^{\Delta/R}v \right\|_{L_t^2 L_x^{\frac{2d}{d+2}}}.
\end{align*}
It follows from H\"older's inequality, Bernstein's inequality, and the chain rule that
$$\left\| (1 - e^{\Delta/R})|v|^{\frac{4}{d-2}} v  \right\|_{L_t^2 L_x^{\frac{2d}{d+2}}} \lesssim R^{-\frac{1}{2}} \|v\|_{L_{t,x}^{\frac{2(d+2)}{d-2}}}^{\frac{4}{d-2}} \| \nabla v \|_{L_t^{\frac{2(d+2)}{d}} L_x^{\frac{2d(d+2)}{d^2+4}}} = O\left(R^{-\frac{1}{2}}\right).$$	
	Then, applying Bernstein's inequality, H\"older's inequality, and Sobolev's inequality, we have
	\begin{align*}
		\left\| |v|^{\frac{4}{d-2}} v - |e^{\Delta/R}v|^{\frac{4}{d-2}} e^{\Delta/R}v \right\|_{L_t^2 L_x^{\frac{2d}{d+2}}} \lesssim& \left\| \left(|v|^{\frac{4}{d-2}} + |e^{\Delta/R}v|^{\frac{4}{d-2}} \right) (1 - e^{\Delta/R})v \right\|_{L_t^2 L_x^{\frac{2d}{d+2}}} \nonumber \\
		\lesssim&_{|I|} \| \nabla v \|_{L_t^\infty L_x^2}^{\frac{4}{d-2}} \left\| \nabla (1 - e^{\Delta/R})v \right\|_{L_t^\infty L_x^2}.
	\end{align*}	
	Given that $v \in C^0_t \dot{H}^1_x(I \times \mathbb{R}^d)$, we conclude that the set $\{v(t)\}_{t \in I}$ is compact in $\dot{H}^1(\mathbb{R}^d)$. By the Riesz characterization of compactness in $L^2$, it follows that the right-hand side tends to 0 as $R \to \infty$.
\end{proof}

		The main result of this section is the following lemma.
		
		\begin{lemma}\label{le.3v3}
			Suppose we are given $\phi \in \dot{H}^1(\mathbb{R}^d)$, $T_0 \in (0, \infty)$, and $\rho \in \{0, 1\}$, and define $f_N$ as in \eqref{eq1.3v1}.
			
			\begin{enumerate}
				\item[(i)] There exists $N_0 = N_0(\phi, T_0)$ sufficiently large such that for every $N \ge N_0$, it can be found a unique solution $U_N \in C([-T_0N^{-2}, T_0N^{-2}] : H^1(\mathbb{H}^d))$ of the initial value problem
				\begin{align}\label{le.3v31}
				\begin{cases}
					i\partial_t U_N + \Delta_g U_N = \rho |U_N|^{\frac{4}{d-2}}U_N, \\
					U_N(0) = f_N.
				\end{cases}
				\end{align}
				Moreover, for every $N \ge N_0$,
				\begin{align*}
					\|U_N\|_{S^1([-T_0N^{-2}, T_0N^{-2}])} \lesssim_ {E_{\mathbb{R}^d}(\phi)}1. 
				\end{align*}
				
				\item[(ii)] Let $\varepsilon_1 \in (0, 1)$ be sufficiently small, depending only on $E_{\mathbb{R}^d}(\phi)$. Choose $\phi' \in H^s(\mathbb{R}^d)$, $s \ge d+2$, such that
				\[
				\|\phi - \phi'\|_{\dot{H}^1(\mathbb{R}^d)} \le \varepsilon_1.
				\]
				Denote by $v' \in C(\mathbb{R} : H^s)$ the solution of the initial-value problem
				\begin{align*}
				\begin{cases}
					i\partial_t v' + \Delta v' = \rho |v'|^{\frac{4}{d-2}} v', \\
					v'(0) = \phi'.
				\end{cases}
				\end{align*}
				Then there exists $N, R_0 \ge 1$ depending on $T_0, \phi'$, and $\varepsilon_1$ such that for every $R \ge R_0$,
				\begin{align}\label{le.3v32}
					\limsup_{N \to \infty} \|U_N - V_{R,N}\|_{S^1([-T_0N^{-2}, T_0N^{-2}])} \lesssim_{E_{\mathbb{R}^d}(\phi)} \varepsilon_1,
				\end{align}
			\begin{align*}
					V_{R,N}(t,x) = N^{\frac{d-2}{2}} \eta \left( \frac{N \Psi_I^{-1}(h^{-1}x)}{R} \right) \left( e^{\Delta/R} v' \right) \left( N^2 t, N \Psi_I^{-1}(h^{-1}x) \right), \text{ for } -T_0N^{-2} \le t \le T_0N^{-2}.
			\end{align*}
			\end{enumerate}
		\end{lemma}
		
		\begin{proof}
			We now show that for any sufficiently large $R_0$, there exists an $N_0$ such that for all $N \ge N_0$, the function $V_{R_0,N}$ is an approximate solution to \eqref{le.3v31}. Then, by applying Proposition \ref{pr.2v6}, we can turn this approximate solution into an exact solution of the initial-value problem \eqref{le.3v31}, which completes the proof of the lemma.
			
			 For $k$ large enough, it follows from the Strichartz inequality and Theorem \ref{th.3v1} that $V_{R_0,N}$ has a finite spacetime bound and that $f_N$ is close to $V_{R_0,N}(0)$.
			Because of
			\begin{align}\label{3v14}
				E_{R,N}(t,x) = \left( (i\partial_t + \Delta_g)V_{R,N} - |V_{R,N}|^{\frac{4}{d-2}}V_{R,N} \right)(t,x),
			\end{align}
			$$
			v'_{R,N}(t,x) = N^{\frac{d-2}{2}} \eta \left( \frac{Nx}{R} \right) \left( e^{\Delta/R}v' \right) (N^2 t, N x), \text{ for } -N^{-2}T_0 \le t \le N^{-2}T_0,
			$$
			and
		$$
			v'_R(t,x) = \eta \left( \frac{x}{R} \right) \left( e^{\Delta/R}v' \right)(t,x).
		$$	
			Then
			\[
			V_{R,N}(t,x) = v'_{R,N}(t, \Psi_I^{-1}(h^{-1}x)), \text{ for } -N^{-2}T_0 \le t \le N^{-2}T_0.
			\]
			Moreover
			\begin{align}\label{eq1.3v15}
				E_{R,N}(t,x) = e_{R,N}(t, \Psi_I^{-1}(h^{-1}x)) + \Delta_g V_{R,N}(t,x) - (\Delta v'_{R,N})(t, \Psi_I^{-1}(h^{-1}x)),
			\end{align}
			where $e_{R,N}(t,x) := \left( (i\partial_t + \Delta)v'_{R,N} - \rho |v'_{R,N}|^{\frac{4}{d-2}} v'_{R,N} \right)(t,x).$
			For estimating the difference in the preceding formula, we introduce the standard vector fields $\partial_j$ on $\mathbb{R}^d$, $j=1,\dots,d$ and the induced vector fields $\tilde{\partial}_j := (\Psi_T)_*(\partial_j)$ on $\mathbb{H}^d$, $j=1,\dots,d$. Then, for each $k=1,2,3$, define
			\begin{align*}
					\left| \tilde{\nabla}^k h(x) \right| := \sum_{k_1+k_2+\dots+k_d=k} \left| \tilde{\partial}_1^{k_1} \tilde{\partial}_2^{k_2} \dots \tilde{\partial}_d^{k_d} h(x) \right|.
			\end{align*}
			Then the following estimate holds:
			\begin{align*}
					\left| \tilde{\nabla}^1 (\Delta_g f(x) - (\Delta (f \circ \Psi_I)) (\Psi_T^{-1}(x))) \right|  \lesssim \sum_{k=1}^3 \left| \Psi_I^{-1}(x) \right|^{k-1} \left| \tilde{\nabla}^k f(x) \right|,
				\end{align*}
				for every $C^3$ function $f: \mathbb{H}^d \to \mathbb{C} \text{ supported in the ball of radius } 1 \text{ centered at } \mathbf{0}$.  Hence,
				\begin{align}\label{eq1.3v16}
					&\left| \tilde{\nabla}^1(\Delta_g V_{R,N}(t,x) - (\Delta v'_{R,N})(t, \Psi_I^{-1}(h^{-1}x))) \right|  \nonumber\\
					\lesssim &\sum_{k=1}^3 \sum_{k_1+k_2+\dots+k_d=k} |\Psi_I^{-1}(h^{-1}x)|^{k-1}  \left| \left( \tilde{\partial}_1^{k_1} \tilde{\partial}_2^{k_2} \dots \tilde{\partial}_d^{k_d} u_k \right)(t, \Psi_I^{-1}(h^{-1}x)) \right| \nonumber\\
					\lesssim& R^3 N^{\frac{d-2}{2}} \sum_{k_1+k_2+\dots+k_d \in \{1,2,3\}}  \left| \left( \tilde{\partial}_1^{k_1} \tilde{\partial}_2^{k_2} \dots \tilde{\partial}_d^{k_d} \left( \eta \left( \frac{\cdot}{R} \right) u(t,\cdot) \right) \right) (N^2 t, N \Psi_I^{-1}(h^{-1}x)) \right|. 
			\end{align}
			We obtain $e_{R,N}(t,x) = N^{\frac{d+2}{2}} e_R(N^2 t, Nx)$, where
			\begin{align*}
				e_R(t,x) := &\left( (i\partial_t + \Delta)v'_R - \rho |v'_R|^{\frac{4}{d-2}} v'_R \right)(t,x) \\
				= &\rho \left( \eta\left(\frac{x}{R}\right) e^{\Delta/R}|v' |^{\frac{4}{d-2}} v' - \eta\left(\frac{x}{R}\right)^{\frac{d+2}{d-2}} |e^{\Delta/R}v'|^{\frac{4}{d-2}} e^{\Delta/R}v' \right) \\
				&+ \frac{1}{R^2} e^{\Delta/R}v'(t,x)(\Delta \eta) \left( \frac{x}{R} \right) + \frac{2}{R} \nabla e^{\Delta/R}v'(t,x) \cdot (\nabla \eta) \left( \frac{x}{R} \right)\\
				:=& e_R^1(t,x) + e_R^2(t,x) + e_R^3(t,x).
			\end{align*}
			Let
		\begin{align*}
				e_R^1(t,x) = &\rho \eta \left( \frac{x}{R} \right) \left( e^{\Delta/R}|v'|^{\frac{4}{d-2}} v' - |e^{\Delta/R}v'|^{\frac{4}{d-2}} e^{\Delta/R}v' \right) + \rho \eta \left( \frac{x}{R} \right) \left( 1 - \eta \left( \frac{x}{R} \right)^{\frac{4}{d-2}} \right) |e^{\Delta/R}v'|^{\frac{4}{d-2}} e^{\Delta/R}v'\\
				=&e_R^{1,1}(t,x)+e_R^{1,2}(t,x).
		\end{align*}
			It follows from the Leibniz rule and a change of variables that
		\begin{align*}
			\left\| \nabla \left( e_R^{1,1}(t,x) \right) \right\|_{L_t^2 L_x^{\frac{2d}{d+2}} ([-T_0,T_0] \times \mathbb{R}^d)} \le &\left\| \nabla \left( e^{\Delta/R}|v'|^{\frac{4}{d-2}} v' - |e^{\Delta/R}v'|^{\frac{4}{d-2}} e^{\Delta/R}v' \right) \right\|_{L_t^2 L_x^{\frac{2d}{d+2}}}\\
			 &+ \left\| \nabla \left( \eta \left( \frac{x}{R} \right) \right) \cdot \left( e^{\Delta/R}|v'|^{\frac{4}{d-2}} v' - |e^{\Delta/R}v'|^{\frac{4}{d-2}} e^{\Delta/R}v' \right) \right\|_{L_t^2 L_x^{\frac{2d}{d+2}}}.
		\end{align*}
			Lemma \ref{le.3v2} makes the first term vanish in the limit $R \to \infty$. The same lemma also covers the second term.
			\begin{align*}
					&\left\| \left( e^{\Delta/R}|v'|^{\frac{4}{d-2}} v' - |e^{\Delta/R}v'|^{\frac{4}{d-2}} e^{\Delta/R}v' \right) \right\|_{L_t^2 L_x^{\frac{2d}{d+2}}} \\
					\lesssim& \frac{1}{R} \left\| e^{\Delta/R}|v|^{\frac{4}{d-2}} v - |e^{\Delta/R}v|^{\frac{4}{d-2}} e^{\Delta/R}v \right\|_{L_t^2 L_x^{\frac{2d}{d+2}}} \to 0, \text{ as } R \to \infty.
			\end{align*}
			Thus
		$$
			\left\| \nabla \left( e_R^{1,1}(t,x) \right) \right\|_{L_t^2 L_x^{\frac{2d}{d+2}} ([-T_0,T_0] \times \mathbb{R}^d)} \to 0, \text{ as } R \to \infty.
		$$
			In order to estimate $e_R^{1,2}(t,x)$, we make use of the Leibniz rule, a change of variables, Hölder's inequality, Sobolev's inequality, and finally the dominated convergence theorem to deduce
			\begin{align*}
				&\left\| \nabla \left( e_R^{1,2}(t,x) \right) \right\|_{L_t^2 L_x^{\frac{2d}{d+2}} ([0,T] \times \mathbb{R}^d)} \\
				& \lesssim \left\| |e^{\Delta/R}v'|^{\frac{4}{d-2}} \nabla e^{\Delta/R}v' \right\|_{L_t^2 L_x^{\frac{2d}{d+2}} ([0,T] \times \{|x| \sim R\})} + \frac{1}{R} \left\| e^{\Delta/R}v' \right\|_{L_t^\infty L_x^{\frac{2d}{d-2}}}^{\frac{d+2}{d-2}} \\
				& \lesssim \left\| \nabla e^{\Delta/R}v' \right\|_{L_t^{\frac{2(d+2)}{d-2}} L_x^{\frac{2d(d+2)}{d^2+4}}} \left\| e^{\Delta/R}v' \right\|_{L_{t,x}^{\frac{2(d+2)}{d-2}}([0,T] \times \{|x| \sim R\})}^{\frac{4}{d-2}}  + O \left( \frac{1}{R} \right) \\
				& \le C(E(v'(0))) \left( \|v'\|_{L_{t,x}^{\frac{2(d+2)}{d-2}} ([0,T] \times \{|x| \sim R\})} + \left\| (1-e^{\Delta/R})v' \right\|_{L_{t,x}^{\frac{2(d+2)}{d-2}}} \right)^{\frac{4}{d-2}} + O \left( \frac{1}{R} \right) \\
				& = o(1) + O \left( \frac{1}{R} \right).
			\end{align*}
			For the term $e_R^2(t,x)$, we apply H\"older's inequality and Sobolev's inequality to obtain
			\begin{align*}
				\left\| \nabla \left( e_R^2(t,x) \right) \right\|_{L_t^1 L_x^2 ([-T_0,T_0] \times \mathbb{R}^d)} & \sim \frac{1}{R^2} \left\| \nabla \left( e^{\Delta/R}v'(t,x)(\Delta \eta) \left( \frac{x}{R} \right) \right) \right\|_{L_t^1 L_x^2 ([-T_0,T_0] \times \mathbb{R}^d)} \\
				& \lesssim \frac{T_0}{R^2} \|\nabla v'\|_{L_t^\infty L_x^2} + \frac{T_0}{R^2} \|v'\|_{L_t^\infty L_x^{2} }\\
				&\lesssim \frac{T_0}{R^2} \|\nabla v'\|_{L_t^\infty L_x^2} \to 0, \text{ as } R \to \infty.
			\end{align*}
			Finally, to estimate $e_R^3(t,x)$, we using Hölder's inequality and Bernstein's inequality to have
			\begin{align*}
				&\left\| \nabla \left( e_R^3(t,x) \right) \right\|_{L_t^1 L_x^2 ([-T_0,T_0] \times \mathbb{R}^d)}\\
				\lesssim& \frac{1}{R} \left\| \Delta e^{\Delta/R}v'(t,x)(\nabla \eta) \left( \frac{x}{R} \right) \right\|_{L_t^1 L_x^2 ([-T_0,T_0] \times \mathbb{R}^d)} + \frac{1}{R} \left\| \nabla e^{\Delta/R}v'(t,x) \frac{1}{R} (\nabla^2 \eta) \left( \frac{x}{R} \right) \right\|_{L_t^1 L_x^2 ([-T_0,T_0] \times \mathbb{R}^d)}\\
				\lesssim& \frac{T_0}{R^2} \sqrt{R} \|\nabla v'\|_{L_t^\infty L_x^2} + \frac{T_0}{R^2} \|\nabla v'\|_{L_t^\infty L_x^2} \to 0, \text{ as } R \to \infty.
			\end{align*}
			Hence,
			\begin{align}\label{eq1.3v17}
				\|\nabla e_R\|_{N^0([-T_0,T_0] \times \mathbb{R}^d)} \to 0, \text{ as } R \to \infty. 
			\end{align}
			Let $N^0([-T_0,T_0] \times \mathbb{R}^d)$ be the dual Strichartz space on $\mathbb{R}^d$, whose definition mirrors that of the dual Strichartz space on $\mathbb{H}^d$. From \eqref{eq1.3v17} we deduce the existence of some $R_0 \ge 1$ such that, for any $R \ge R_0$ and $N \ge 1$,
\begin{align}\label{eq1.3v18}
\limsup_{R \to \infty} \|\nabla e_{R,N}\|_{N^0([-N^{-2}T_0, N^{-2}T_0] \times \mathbb{R}^d)} \to 0. 
\end{align}
Then, from \eqref{eq1.3v15} and \eqref{eq1.3v16}, we have
\begin{align*}
\left| \tilde{\nabla}^1 E_{R,N}(t,x) \right| \lesssim& \left| (\nabla e_{R,N}) (t, \Psi_I^{-1}(h^{-1}x)) \right| + \left| \tilde{\nabla}^1(\Delta_g V_{R,N}(t,x) - (\Delta v'_{R,N})(t, \Psi_I^{-1}(h^{-1}x)) )\right| \\
\lesssim& \left| (\nabla e_{R,N}) (t, \Psi_I^{-1}(h^{-1}x)) \right| \\
&+ R^3 N^{\frac{d-2}{2}} \sum_{k_1+k_2+\dots+k_d \in \{1,2,3\}} \left| \left( \tilde{\partial}_1^{k_1} \tilde{\partial}_2^{k_2} \dots \tilde{\partial}_d^{k_d} \left( \eta \left( \frac{\cdot}{R} \right) u(t,\cdot) \right) \right) (N^2 t, N \Psi_I^{-1}(h^{-1}x)) \right|.
\end{align*}
Combining this with \eqref{eq1.3v18}, we obtain that for every sufficiently large $R_0$, there exists $N_0$ (depending on $R_0$) such that for all $N \geq N_0$,
$$\lim_{T_0\to \infty}\|\nabla_g E_{R_0,N}\|_{N^1([-N^{-2}T_0, N^{-2}T_0])} \to 0.$$
From Proposition \ref{pr.2v6}, we can deduce the existence of a global solution $U_N$ to $\text{NLS}_{\mathbb{H}^d}$ with initial data $U_N(0) = \pi_{h}T_N\phi$ and (i).
\end{proof}
	
As a consequence, we have the following corollary using the approach outlined in  \cite[Corollary 4.3]{Ionescu2012}.
			
\begin{corollary}\label{corollary}
 Assume $\psi \in \dot{H}^1(\mathbb{R}^d)$, $\varepsilon > 0$, and $I \subseteq \mathbb{R}$ is an interval such that
\begin{align*}
\|\nabla (e^{it\Delta}\psi)\|_{L^p_t L^q_x(I\times\mathbb{R}^d )} \le \varepsilon,
\end{align*}
with $\frac{2}{p} + \frac{d}{q} = \frac{d}{2}$ and $q \in (2, \frac{2d}{d-2}]$. For $N \ge 1$ define as before
\begin{align*}
(Q_N\psi)(x) = \eta \left( \frac{x}{N^{\frac{1}{2}}} \right) \cdot (e^{\Delta/N}\psi)(x), \quad \psi_N(x) = N^{\frac{d-2}{2}}(Q_N\psi)(Nx), \quad \tilde{\psi}_N(y) = \psi_N(\Psi_I^{-1}(y)).
\end{align*}
Then there exists $N_1 = N_1(\psi, \varepsilon)$ such that for every $N \ge N_1$,
\begin{align*}
\|\nabla_g (e^{it\Delta_g}\tilde{\psi}_N)\|_{L^p_t L^q_x(N^{-2}I\times\mathbb{H}^d)} \lesssim_q \varepsilon.
\end{align*}
\end{corollary}


			
\section{Profile decomposition on hyperbolic spaces}\label{sec:Prodechs}
			
The definitions of the hyperbolic and Euclidean frames, along with their associated profiles, are provided at the beginning of this section. We then prove basic properties of the profiles related to equivalent and orthogonal frames. After establishing the inverse Strichartz estimate and performing a finitary decomposition, we ultimately construct the linear profile decomposition on hyperbolic space.
			
For a triple \((f, t_0, h_0) \in L^2(\mathbb{H}^d) \times \mathbb{R} \times \mathbb{G}\), we set
\begin{align}\label{eq1.4v1}
\Pi_{t_0, h_0} f(x) = (e^{-it_0 \Delta_g} f)(h_0^{-1} x) = (\pi_{h_0} e^{-it_0 \Delta_g} f)(x). 
\end{align}
As in Section \ref{sec:Eucapp}, see \eqref{eq1.3v1}, for $\phi \in \dot{H}^1(\mathbb{R}^d)$ and $N \geq 1$, we introduce
\begin{align}\label{eq1.4v2}
T_N \phi(x) := N^{\frac{d-2}{2}} \tilde{\phi}(N \Psi_I^{-1}(x)), \quad \text{where } \tilde{\phi}(y) := \eta\left(\frac{y}{N^{\frac{1}{2}}}\right) \cdot (e^{\Delta/N} \phi)(y), 
\end{align}
and we observe that $T_N : \dot{H}^1(\mathbb{R}^d) \to H^1(\mathbb{H}^d)$ is a bounded linear operator satisfying
			\begin{align*}
				\|T_N \phi\|_{H^1(\mathbb{H}^d)} \lesssim \|\phi\|_{\dot{H}^1(\mathbb{R}^d)}.
			\end{align*}
			
			\begin{definition}\label{def.4v1}
				\begin{enumerate}
					\item[(1)] A \emph{frame} is a sequence $\mathcal{O}_k = (N_k, t_k, h_k) \in [1, \infty) \times \mathbb{R} \times \mathbb{G}$ for $k = 1, 2, \ldots$, where $N_k \geq 1$ is a scale, $t_k \in \mathbb{R}$ represents time, and $h_k \in \mathbb{G}$ denotes a translation. We require either that $N_k = 1$ for every $k$ (then we call $\{\mathcal{O}_k\}_{k \geq 1}$ a \emph{hyperbolic frame}) or that $N_k \nearrow \infty$ (then we call it a \emph{Euclidean frame}). Let $\mathcal{F}_e$ denote the collection of Euclidean frames
					\begin{align*}
						\mathcal{F}_e = \{\mathcal{O} = \{(N_k, t_k, h_k)\}_{k \geq 1} : N_k \in [1, \infty),\; t_k \in \mathbb{R},\; h_k \in \mathbb{G},\; N_k \nearrow \infty\},
					\end{align*}
					and let $\mathcal{F}_h$ denote the collection of hyperbolic frames
					\begin{align*}
						\mathcal{F}_h = \{\tilde{\mathcal{O}} = \{(1, t_k, h_k)\}_{k \geq 1} : t_k \in \mathbb{R},\; h_k \in \mathbb{G}\}.
					\end{align*}
					
					\item[(2)] Two frames $\{(N_k, t_k, h_k)\}_{k \geq 1}$ and $\{(N'_k, t'_k, h'_k)\}_{k \geq 1}$ are said to be \emph{orthogonal} if
					\begin{align}\label{def.4v11}
						\lim_{k \to \infty} \left[ \left|\ln\left(\frac{N_k}{N'_k}\right) \right| + N_k^2 |t_k - t'_k| + N_k d(h_k \cdot 0, h'_k \cdot 0) \right] = +\infty. 
					\end{align}
					If they are not orthogonal, they are called \emph{equivalent}.
					
					\item[(3)] Given $\phi \in \dot{H}^1(\mathbb{R}^d)$ and a Euclidean frame $\mathcal{O} = \{\mathcal{O}_k\}_{k \geq 1} = \{(N_k, t_k, h_k)\}_{k \geq 1} \in \mathcal{F}_e$, we define the \emph{Euclidean profile} associated with $(\phi, \mathcal{O})$ as the sequence $\tilde{\phi}_{\mathcal{O}_k}$ given by
					\begin{align}\label{def.4v12}
						\tilde{\phi}_{\mathcal{O}_k} := \Pi_{t_k, h_k}(T_{N_k} \phi), 
					\end{align}
					where the operators $\Pi$ and $T$ are as in \eqref{eq1.4v1} and \eqref{eq1.4v2}.
					
					\item[(4)] Given $\psi \in H^1(\mathbb{H}^d)$ and a hyperbolic frame $\tilde{\mathcal{O}} = \{\tilde{\mathcal{O}}_k\}_{k \geq 1} = \{(1, t_k, h_k)\}_{k \geq 1} \in \mathcal{F}_h$, we define the \emph{hyperbolic profile} associated with $(\psi, \tilde{\mathcal{O}})$ as the sequence $\tilde{\psi}_{\tilde{\mathcal{O}}_k}$ given by
					\begin{align*}
						\tilde{\psi}_{\tilde{\mathcal{O}}_k} := \Pi_{t_k, h_k} \psi.
					\end{align*}
				\end{enumerate}
			\end{definition}
			\begin{definition}\label{def.4v2}
 A bounded sequence $\{f_{k}\}_{k}$ in $H^{1}\left(\mathbb{H}^{d}\right)$ is said to be absent from a frame $\mathcal{O}=\left\{\left(N_{k}, t_{k}, h_{k}\right)\right\}_{k}$ if its restriction to $\mathcal{O}$ converges weakly to zero. More precisely, for every profile $\tilde{\phi}_{\mathcal{O}_{k}}$ associated with $\mathcal{O}$,
				\begin{align}\label{def.4v21}
					\lim _{k \rightarrow \infty}\left\langle f_{k}, \tilde{\phi}_{\mathcal{O}_{k}}\right\rangle_{H^{1} \times H^{1}\left(\mathbb{H}^{d}\right)}=0. 
				\end{align}
			\end{definition}
	\begin{remark}\label{remark}
		\begin{enumerate}
			\item[(i)] When $\mathcal{O}=(1, t_k, h_k)_k$ is a hyperbolic frame, condition \eqref{def.4v21} is equivalent to
			\begin{align*}
				\Pi_{-t_k, h_k^{-1}} f_k \rightharpoonup 0 \quad \text{in } H^1(\mathbb{H}^d) \text{ as } k \to \infty.
			\end{align*}
			\item[(ii)] For a Euclidean frame $\mathcal{O}=(N_k, t_k, h_k)_k$, \eqref{def.4v21} is equivalent to the following: for every $R>0$,
			\begin{align*}
				g_k^R(v) = \eta\left(\frac{v}{R}\right) N_k^{-\frac{d-2}{2}} \bigl(\Pi_{-t_k, h_k^{-1}} f_k\bigr)\left(\Psi_I\left(\frac{v}{N_k}\right)\right) \to 0 \quad \text{in } \dot{H}^1(\mathbb{R}^d) \text{ as } k \to \infty.
			\end{align*}
		\end{enumerate}
	\end{remark}
	
			We first establish some basic facts about profiles coming from equivalent or orthogonal frames.
			\begin{lemma}\label{le.4v4}
				\begin{enumerate}
					\item[(i)] Suppose $\left\{\mathcal{O}_k\right\}_{k \geq 1} = \left\{\left(N_k, t_k, h_k\right)\right\}_{k \geq 1}$ and $\left\{\mathcal{O}_k'\right\}_{k \geq 1} = \left\{\left(N_k', t_k', h_k'\right)\right\}_{k \geq 1}$ are two equivalent Euclidean frames (or both hyperbolic frames), and let $\phi \in \dot{H}^1(\mathbb{R}^d)$ (or $\phi \in H^1(\mathbb{H}^d)$ respectively). Then there exists an element $\phi' \in \dot{H}^1(\mathbb{R}^d)$ (or $\phi' \in H^1(\mathbb{H}^d)$) such that, after passing to a subsequence,
					\begin{equation}\label{le.4v41}
						\lim_{k \to \infty} \bigl\| \tilde{\phi}_{\mathcal{O}_k} - \tilde{\phi}'_{\mathcal{O}_k'} \bigr\|_{H^1(\mathbb{H}^d)} = 0 ,
					\end{equation}
					where $\tilde{\phi}_{\mathcal{O}_k}, \tilde{\phi}_{\mathcal{O}_k'}$ are the profiles defined in Definition \ref{def.4v1}.
					
					\item[(ii)] Let $\left\{\mathcal{O}_k\right\}_{k \geq 1} = \left\{\left(N_k, t_k, h_k\right)\right\}_{k \geq 1}$ and $\left\{{\mathcal{O}}_k'\right\}_{k \geq 1} = \left\{\left(N_k', t_k', h_k'\right)\right\}_{k \geq 1}$ be two orthogonal frames (either Euclidean or hyperbolic), and let $\tilde{\phi}_{\mathcal{O}_k}, \tilde{\phi}_{\mathcal{O}_k'}$ be the corresponding profiles. Then
					\begin{equation}\label{le2.4v42}
						\lim_{k \to \infty} \left| \int_{\mathbb{H}^d} D^{\alpha} \tilde{\phi}_{\mathcal{O}_k} \overline{D_{\alpha} \tilde{\psi}_{\mathcal{O}_k'}} \, d\mu \right| + \lim_{k \to \infty} \bigl\| \tilde{\phi}_{\mathcal{O}_k} \tilde{\psi}_{\mathcal{O}_k'} \bigr\|_{L^{\frac{d}{d-2}}(\mathbb{H}^d)} = 0.
					\end{equation}
					
					\item[(iii)] If $\tilde{\phi}_{\mathcal{O}_k}$ and $\tilde{\psi}_{\mathcal{O}_k}$ are two Euclidean profiles associated with the same frame, then
					\begin{align*}
						\lim_{k \to \infty} \bigl\langle \nabla_g \tilde{\phi}_{\mathcal{O}_k}, \nabla_g \tilde{\psi}_{\mathcal{O}_k} \bigr\rangle_{L^2 \times L^2(\mathbb{H}^d)}
						&= \lim_{k \to \infty} \int_{\mathbb{H}^d} D^{\alpha} \tilde{\phi}_{\mathcal{O}_k} \overline{D_{\alpha} \tilde{\psi}_{\mathcal{O}_k}} \, d\mu \\
						&= \int_{\mathbb{R}^d} \nabla \phi(x) \cdot \nabla \psi(x) \, dx = \bigl\langle \nabla \phi, \nabla \psi \bigr\rangle_{L^2 \times L^2(\mathbb{R}^d)}.
					\end{align*}
				\end{enumerate}
			\end{lemma}
	
	\begin{proof}
		 This proof is similar to Lemma 5.4 in \cite{Ionescu2012}, and we can generalize its three‑dimensional version to higher dimensions.
	\end{proof}
	
	We next give the main result.
	
	\begin{proposition}[Linear profile decomposition]\label{pr.4v5}
		Suppose that $\{f_k\}_{k\ge1}$ is a bounded sequence in $H^1(\mathbb{H}^d)$. Then, after passing to a subsequence, there exist sequences of pairs $(\phi^\mu,\mathcal{O}^\mu)\in\dot{H}^1(\mathbb{R}^d)\times\mathcal{F}_e$ and $(\psi^\nu,\tilde{\mathcal{O}}^\nu)\in H^1(\mathbb{H}^d)\times\mathcal{F}_h$ ($\mu,\nu=1,2,\dots$) such that for every $J\ge1$ the following decomposition holds:
		$$
		f_k = \sum_{\mu=1}^J \tilde{\phi}_{\mathcal{O}_k^\mu}^{\mu} + \sum_{\nu=1}^J \tilde{\psi}_{\mathcal{\tilde{O}}_k^\nu}^{\nu} + r_k^J,
	$$
		where $\tilde{\phi}_{\mathcal{O}_k^\mu}^{\mu}$ and $\tilde{\psi}_{\tilde{\mathcal{O}}_k^\nu}$ are the corresponding profiles as in Definition \ref{def.4v1}. Moreover, one has
		\begin{align}\label{pr.4v51}
			\lim_{J\to\infty}\limsup_{k\to\infty}\sup_{\substack{N\ge1\\ t\in\mathbb{R},\ x\in\mathbb{H}^d}} N^{-\frac{d-2}{2}}|P_N e^{it\Delta_g} r_k^J|(x)=0.
		\end{align}
		In addition, the families $\{\mathcal{O}^\mu\}_{\mu\ge1}$ and $\{\tilde{\mathcal{O}}^\nu\}_{\nu\ge1}$ are pairwise orthogonal. Finally, the decomposition is asymptotically orthogonal in the sense that
		\begin{align}\label{pr.4v52}
			\lim_{J\to\infty}\limsup_{k\to\infty}\left( E(f_k) - \sum_{\mu=1}^J E(\tilde{\phi}_{\mathcal{O}_k^\mu}^{\mu}) - \sum_{\nu=1}^J E(\tilde{\psi}_{\tilde{\mathcal{O}}_k^\nu}^{\nu}) - E(r_k^J) \right)=0,
		\end{align}
		where $E$ is defined in \eqref{eq1.1v2}.
	\end{proposition}

For a bounded sequence $\{g_k\}_k$ in $H^1(\mathbb{H}^d)$, define
\begin{align}\label{eq1.4v24}
\delta ( ( g _ { k } ) _ { k } ) = \operatorname* { l i m } _ { k \to \infty } \sup_{\substack{N\ge1\\ t\in\mathbb{R},\ h\in\mathbb{H}^d}}  N^{-\frac{d-2}{2}}| P _ { N } ( e ^{i t \Delta _ { g } } g _ { k } ) ( h \cdot 0 ) | .
	\end{align}
	When $\delta((f_k)_{k}) \leq \delta$, we simply set $J = 0$ and $f_k = r_k$, Lemma \ref{le.4v6} holds trivially. Otherwise we proceed inductively using the following Lemma \ref{le.4v7}. If the Lemma \ref{le.4v7} is proved, Lemma \ref{le.4v6} follows by iteration. Indeed, define $(f_k^0)_k = (f_k)_k$. When $\delta((f_k^\alpha)_k) \geq \delta$, apply the claim to $(f_k^\alpha)_k$ to obtain
	\begin{align*}
		f _ { k } ^{\alpha + 1} = f _ { k } ^{\alpha} - \tilde { \phi } _ { \mathcal { O } _ { k } ^{\alpha + 1} } .
	\end{align*}
	By construction, the sequence $(f_k^\alpha)_k$ is absent from all frames $\mathcal{O}^\beta$ with $\beta \leq \alpha$. The process stops after $O(\delta^{-2})$ steps because each profile contributes at least $\delta$ to the energy: from \eqref{def.4v21} and the orthogonality,
	\begin{align*}
		\| \nabla _ { g } f _ { k } ^{\alpha - 1} \| _ { L ^{2} } ^{2} = \| \nabla _ { g } f _ { k } ^{\alpha} \| _ { L ^{2} } ^{2} + \| \nabla _ { g } \tilde { \phi } _ { \mathcal { O } _ { k } ^{\alpha} } \| _ { L ^{2} } ^{2} + 2 \left\langle f _ { k } ^{\alpha} , \tilde { \phi } _ { \mathcal { O } _ { k } ^{\alpha} } \right\rangle _ { H ^{1} \times H ^{1} ( \mathbb { H } ^{d} ) } = \| \nabla _ { g } f _ { k } ^{\alpha} \| _ { L ^{2} } ^{2} + \| \nabla _ { g } \tilde { \phi } _ { \mathcal { O } _ { k } ^{\alpha} } \| _ { L ^{2} } ^{2} + o _ { k } ( 1 ) ,
	\end{align*}
	hence by induction
	\begin{align*}
		\| \nabla _ { g } f _ { k } \| _ { L ^{2} } ^{2} = \sum _ { 1 \leq \alpha \leq A } \| \nabla _ { g } \tilde { \phi } _ { \mathcal { O } ^{\alpha} } \| _ { L ^{2} } ^{2} + \| \nabla _ { g } f _ { k } ^{A} \| _ { L ^{2} } ^{2} + o _ { k } ( 1 ) .
	\end{align*}
	Since each $\|\nabla_{g}\tilde{\phi}_{\mathcal{O}^{\alpha}}\|_{L^{2}}^{2}\geq\delta^{2}$ asymptotically, the number of steps is bounded by $\delta^{-2}\|\nabla_{g}f_{k}\|_{L^{2}}^{2}$.

\begin{lemma}\label{le.4v7}
Let $\{g_k\}_k$ be a bounded sequence in $H^1(\mathbb{H}^d)$ that is absent from a family of frames $\{\mathcal{O}^\alpha\}_{\alpha\leq A}$, such that $\delta((g_k)_k)\geq\delta$. Then, after passing to a subsequence, there exists a new frame $\mathcal{O}'$, orthogonal to each $\mathcal{O}^\alpha$ ($\alpha\leq A$) and a profile $\tilde
		{\phi}_{\mathcal{O}'_k}$ with free energy
		\begin{align}\label{le.4v71}
			\operatorname* { l i m } _ { k \to \infty } \| \nabla _ { g } \tilde { \phi } _ { \mathcal { O } _ { k } ^{\prime} } \| _ { L ^{2} } \geq \delta
		\end{align}
		such that $g_k - \tilde{\phi}_{\mathcal{O}_k'}$ is absent from $\mathcal{O}'$ and from all $\mathcal{O}^\alpha(\alpha \leq A)$.
	\end{lemma} 
\begin{proof}
	According to the hypothesis, we can find a sequence $\tilde{\mathcal{O}}_k = (N_k, t_k, h_k)_k$ such that the $\limsup_{k \to \infty}$ in \eqref{eq1.4v24} exceeds $\frac{\delta}{2}$. If $\limsup N_k = \infty$, after extracting a subsequence we may take $\{\tilde{\mathcal{O}}_k\}_k = \tilde{\mathcal{O}}$ to be a Euclidean frame. Otherwise, $N_k \to N \geq 1$ and we may take $\mathcal{O}' = \{(1, t_k, h_k)_k\}_{k \geq 1}$ to be a hyperbolic frame. In all case, we obtain a frame $\mathcal{O}' = \{(M_k, t_k, h_k)_k\}_{k \geq 1}$ with $M_k$ comparable to $N_k$ satisfying
	\begin{align}\label{eq1.4v26}
		\frac { \delta } { 2 } \leq \operatorname* { l i m } _ { k \to \infty } N _ { k } ^{- \frac { d - 2} { 2 } } \, | P _ { N _ { k } } ( e ^{i t _ { k} \Delta _ { g } } g _ { k } ) ( h _ { k } \cdot 0 ) | = \operatorname* { l i m } _ { k \to \infty } | \langle \Pi _ { - t _ { k } , h _ { k } ^{- 1} } g _ { k } , \ N _ { k } ^{- \frac { d - 2} { 2 } } \, P _ { N _ { k } } \delta _ { 0 } \rangle _ { L ^{2} \times L ^{2} ( \mathbb { H } ^{d} ) } | .
	\end{align}
	Next, we assert that there exists a profile $\tilde{f}_{\mathcal{O}'_{k}}$ that is associated with the frame $\mathcal{O}'$ satisfying
	\begin{align}\label{eq1.4v27}
		\operatorname* { l i m s u p } _ { k \to \infty } \left\| \nabla _ { g } \tilde { f } _ { \mathcal { O } _ { k } ^{\prime} } \right\| _ { L ^{2} } \lesssim 1,
	\end{align}
	moreover
\begin{align}\label{eq1.4v28}
\Pi _ { - t _ { k } , h _ { k } ^{- 1} } \tilde { f } _ { \mathcal { O } _ { k } ^{\prime} } - N _ { k } ^{- \frac { d + 2} { 2 } } e ^{N _ { k} ^{- 2} \Delta _ { g } } \delta _ { \mathbf { 0 } } \rightarrow 0 \mathrm { ~ i n ~ } H ^{1} (\mathbb { H } ^{d}) , \mathrm { ~ a s ~ } k \rightarrow \infty .
	\end{align}
	In fact, when $\mathcal{O}'$ is a hyperbolic frame, we choose
	\begin{align}\label{eq1.4v29}
		f ( x ) = N ^{- \frac { d + 2} { 2 } } e ^{N ^ { - 2} \Delta _ { g } } \delta _ { \mathbf { 0 } } ( x ) ,
	\end{align}
	when $\mathcal{O}'$ is a Euclidean frame, we choose
	\begin{align}\label{eq1.4v30}
		f ( x ) = ( 4 \pi ) ^{- \frac { d} { 2 } } e ^{- \frac { | x | ^ { 2} } { 4 } } = e ^{\Delta} \delta _ { 0 } ( x ) .
	\end{align}
	If $\mathcal{O}'$ is a hyperbolic frame, because $	| c ( \lambda ) | ^{- 2} \lesssim | \lambda | ^{2} ( 1 + | \lambda | ) ^{d - 3}$, it follows from the Plancherel identity that we obtain
	\begin{align}\label{eq1.4v31}
		\left\| \nabla _ { g } \tilde { f }_{\mathcal { O } _ { k }} \right\| _ { L ^{2} \left( \mathbb { H } ^{d} \right) } =& N ^{- \frac { d + 2} { 2 } } \left\| \nabla _ { g } e ^{N ^ { - 2} \Delta _ { g } } \delta _ { \mathbf { 0 } } ( x ) \right\| _ { L ^{2} \left( \mathbb { H } ^{d} \right) } \nonumber\\
\sim& N ^{- \frac { d + 2} { 2 } } \left( \int _ { \mathbb { R } } \left( \lambda ^{2} + \left( \frac { d - 1 } { 2 } \right) ^{2} \right) e ^{- 2 N ^ { - 2} \left( \lambda ^{2} + \left( \frac { d - 1 } { 2 } \right) ^{2} \right) } | c ( \lambda ) | ^{- 2} \, \mathrm { d } \lambda \right) ^{\frac { 1} { 2 } } \nonumber\\
	 \lesssim& N ^{- \frac { d + 2} { 2 } } \left( \int _ { \mathbb { R } } \left( \lambda ^{2} + \left( \frac { d - 1 } { 2 } \right) ^{2} \right) e ^{- 2 N ^ { - 2} \left( \lambda ^{2} + \left( \frac { d - 1 } { 2 } \right) ^{2} \right) } | \lambda | ^{2} ( 1 + | \lambda | ) ^{d - 3} \mathrm { d } \lambda \right) ^{\frac { 1} { 2 } } \nonumber\\
		\lesssim& N ^{- \frac { d + 2} { 2 } } \left( \int _ { 0 } ^{\infty} \left( \lambda ^{2} + 1 \right) ^{\frac { d - 1} { 2 } } \lambda ^{2} e ^{- 2 N ^ { - 2} \left( \lambda ^{2} + 1 \right) } \mathrm { d } \lambda \right) ^{\frac { 1} { 2 } } \nonumber\\
		\lesssim& \frac { 1 } { N ^{\frac { d - 1} { 2 } } e ^{N^ {- 2}} } \left( \int _ { 0 } ^{N ^ { - 1} } \lambda ^{2} e ^{- 2 \lambda ^ { 2} } \mathrm { d } \lambda \right) ^{\frac { 1} { 2 } } + \frac { 1 } { N ^{\frac { d - 1} { 2 } } e ^{N ^{- 2}} } \left( \int _ { N ^{- 1} } ^{\infty} N ^{d - 1} \lambda ^{d - 1} \lambda ^{2} e ^{- 2 \lambda ^ { 2} } \mathrm { d } \lambda \right) ^{\frac { 1} { 2 } }\nonumber\\
		 \lesssim& 1,
	\end{align}
	so that
	\begin{align*}
	 \limsup _{k \rightarrow \infty} \left\| \nabla _ { g } \tilde { f } _ { \mathcal { O } _ { k } ^{\prime} } \right\| _ { L ^{2} ( \mathbb { H } ^{d} ) } \lesssim 1 ,
	\end{align*}
	and
	\begin{align*}
		& \left\| \Pi _ { - t _ { k } , h _ { k } ^{- 1} } \tilde { f } _ { \mathcal { O } _ { k } ^{\prime} } - N _ { k } ^{- \frac { d + 2} { 2 } }e^{N_k^{-2}\Delta_g} \delta _ { \mathbf { 0 } } \right\| _ { H ^{1} \left( \mathbb { H } ^{d} \right) } \\
		= & \left\| N ^{- \frac { d + 2} { 2 } } e ^{N ^ { - 2} \Delta _ { g } } \delta _ { \mathbf { 0 } } ( x ) - N _ { k } ^{- \frac { d + 2} { 2 } }e^{N_k^{-2}\Delta_g}\delta _ { \mathbf { 0 } } ( x ) \right\| _ { H ^{1} \left( \mathbb { H } ^{d} \right) } \\
		\sim & \left( \int _ { \mathbb { R } } \left( \lambda ^{2} + \left( \frac { d - 1 } { 2 } \right) ^{2} \right) \left( N ^{- \frac { d + 2} { 2 } } e ^{- N ^ { - 2} \left( \lambda ^{2} + \left( \frac { d - 1 } { 2 } \right) ^{2} \right) } - N _ { k } ^{- \frac { d + 2} { 2 } } e ^{- N _ { k} ^{- 2} \left( \lambda ^{2} + \left( \frac { d - 1 } { 2 } \right) ^{2} \right) } \right) ^{2} | c ( \lambda ) | ^{- 2} \, \mathrm { d } \lambda \right) ^{\frac { 1} { 2 } }.
	\end{align*}
	Thus, applying the dominated convergence theorem yields
	\begin{align*}
		\left\| \Pi _ { - t _ { k } , h _ { k } ^{- 1} } \tilde { f } _ { \mathcal { O } _ { k } ^{\prime} } - N _ { k } ^{- \frac { d + 2} { 2 } } e ^{N _ { k} ^{- 2} \Delta _ { g } } \delta _ { \mathbf { 0 } } \right\| _ { H ^{1} \left( \mathbb { H } ^{d} \right) } \to 0 , \ \mathrm { a s } \ k \to \infty .
	\end{align*}
When $\mathcal{O}'$ is a Euclidean frame, we observe that
	\begin{align*}
		\left\| \nabla _ { g } \tilde { f } _ { \mathcal { O } _ { k } } \right\| _ { L ^{2} } \lesssim \left\| \nabla f \right\| _ { L ^{2} ( \mathbb { R } ^{d} ) } \sim \left\| \nabla e ^{\Delta \delta _ { 0} } \right\| _ { L ^{2} ( \mathbb { R } ^{d} ) } \sim \left( \int _ { \mathbb { R } ^{d} } | \xi | ^{2} e ^{- 2 | \xi | ^ { 2} } \, \mathrm { d } \xi \right) ^{\frac { 1} { 2 } } \lesssim 1 ,
	\end{align*}
	\begin{align*}
		 \limsup _{k \rightarrow \infty}\left\| \nabla _ { g } \tilde { f } _ { \mathcal { O } _ { k } ^{\prime} } \right\| _ { L ^{2} ( \mathbb { H } ^{d} ) } \lesssim 1 .
	\end{align*}
We now prove \eqref{eq1.4v28}, which is equivalent to
	\begin{align}\label{eq1.4v32}
		\left\| N _ { k } ^{- \frac { d + 2} { 2 } } e ^{N _ { k} ^{- 2} \Delta _ { g } } \delta _ { 0 } - T _ { N _ { k } } f \right\| _ { H ^{1} ( \mathbb { H } ^{d} ) } \to 0 , \mathrm { ~ a s ~ } k \to \infty .
	\end{align}
We now prove \eqref{eq1.4v32}. First, we have the following result
\begin{align}\label{eq1.4v33}
\left\| T _ { N _ { k } } f - N _ { k } ^{\frac { d - 2} { 2 } } f \left( N _ { k } \Psi _ { I } ^{- 1} \cdot \right) \right\| _ { H ^{1} \left( \mathbb { H } ^{d} \right) } \to 0 , \ \mathrm { a s } \ k \to \infty .
	\end{align}
	That is,
	\begin{align*}
		& \left\| T_{N_{k}}f - N_{k}^{\frac{d-2}{2}} f\left(N_{k}\Psi_{I}^{-1}\cdot\right)\right\|_{H^{1}(\mathbb{H}^{d}) }\\
		\leq & N_{k}^{\frac{d-2}{2}}\left\|\left(\eta\left(N_{k}^{\frac{1}{2}}\Psi_{I}^{-1}(x)\right)-1\right)\left(e^{\Delta/N_{k}}f\right)\left(N_{k}\Psi_{I}^{-1}(x)\right)\right\|_{H^{1}\left(\mathbb{H}^{d}\right)} \\
		& + N_{k}^{\frac{d-2}{2}}\left\|\left(e^{\Delta/N_{k}}f\right)\left(N_{k}\Psi_{I}^{-1}(x)\right)-f\left(N_{k}\Psi_{I}^{-1}(x)\right)\right\|_{H^{1}\left(\mathbb{H}^{d}\right)}.
	\end{align*}
We observe that
	\begin{align*}
		& N_{k}^{\frac{d-2}{2}}\left\|\left(\eta\left(N_{k}^{\frac{1}{2}}\Psi_{I}^{-1}(x)\right)-1\right)\left(e^{\Delta/N_{k}}f\right)\left(N_{k}\Psi_{I}^{-1}(x)\right)\right\|_{H^{1}\left(\mathbb{H}^{d}\right)} \\
		\leq & N_{k}^{\frac{d-2}{2}}\left\|\nabla_{g}\left(\eta\left(N_{k}^{\frac{1}{2}}\Psi_{I}^{-1}(x)\right)\right)\left(e^{\Delta/N_{k}}f\right)\left(N_{k}\Psi_{I}^{-1}(x)\right)\right\|_{L^{2}\left(\mathbb{H}^{d}\right)} \\
		& + N_{k}^{\frac{d-2}{2}}\left\|\left(\eta\left(N_{k}^{\frac{1}{2}}\Psi_{I}^{-1}(x)\right)-1\right)\nabla_{g}\left(\left(e^{\Delta/N_{k}}f\right)\left(N_{k}\Psi_{I}^{-1}(x)\right)\right)\right\|_{L^{2}\left(\mathbb{H}^{d}\right)} \\
		= & I_{1}+I_{2}.
	\end{align*}
One can check that
	\begin{align*}
		I_{1} = &N_{k}^{\frac{d-2}{2}}\left(\int_{\mathbb{R}^{d}} \sum_{i,j=1}^{d}(\delta_{ij}+v_{i}v_{j})\partial_{i}\bigl(\eta(N_{k}^{1/2}v)\bigr)\partial_{j}\bigl(\eta(N_{k}^{1/2}v)\bigr)\bigl|\bigl(e^{\Delta/N_{k}}f\bigr)(N_{k}v)\bigr|^{2}(1+|v|^{2})^{-1/2}\mathrm{d}v\right)^{1/2} \\
		\lesssim& N_{k}^{-1/2}\left(\int_{\mathbb{R}^{d}}\bigl|\bigl(e^{\Delta/N_{k}}f\bigr)(v)\bigr|^{2}\mathrm{d}v\right)^{1/2}\lesssim N_{k}^{-1/2}\|f\|_{L^{2}(\mathbb{R}^{d})} \rightarrow 0, \quad k\rightarrow\infty,
	\end{align*}
and
	\begin{align*}
		I_{2} = &N_{k}^{\frac{d-2}{2}}\left(\int_{\mathbb{R}^{d}} (1-\eta(N_{k}^{1/2}v))^{2}\sum_{i,j=1}^{d}(\delta_{ij}+v_{i}v_{j})\partial_{i}\bigl((e^{\Delta/N_{k}}f)(N_{k}v)\bigr)\partial_{j}\bigl((e^{\Delta/N_{k}}f)(N_{k}v)\bigr)(1+|v|^{2})^{-1/2}\mathrm{d}v\right)^{1/2} \\
		\lesssim &\left(\int_{|v|\geq N_{k}^{1/2}} (1+|v|^{2})^{1/2}\bigl|\nabla(e^{\Delta/N_{k}}f)(v)\bigr|^{2}\mathrm{d}v\right)^{1/2} \rightarrow 0, \quad k\rightarrow\infty,
	\end{align*}
	as
	\begin{align*}
		\left|\nabla(e^{\Delta/N_{k}}f)(v)\right| = \left|\nabla e^{(1+1/N_{k})\Delta}\delta_{0}(v)\right| = \left|\nabla\left(\frac{1}{(4\pi(1+1/N_{k}))^{d/2}}e^{-|v|^{2}/(4(1+1/N_{k}))}\right)\right| \sim |v|e^{-|v|^{2}/4},
	\end{align*}
	for large enough $k$, which leads to
	\begin{align*}
		(1+|v|^{2})^{1/2}\left|\nabla(e^{\Delta/N_{k}}f)(v)\right|^{2} \in L^{1}(\mathbb{R}^{d}).
	\end{align*}
	Therefore
	\begin{align*}
		N_{k}^{\frac{d-2}{2}}\left\|\left(\eta(N_{k}^{1/2}\Psi_{I}^{-1}(x))-1\right)\left(e^{\Delta/N_{k}}f\right)\left(N_{k}\Psi_{I}^{-1}(x)\right)\right\|_{H^{1}(\mathbb{H}^{d})} \rightarrow 0, \quad k\rightarrow\infty.
	\end{align*}
	Moreover
	\begin{align*}
		& N_{k}^{\frac{d-2}{2}}\left\|\left(e^{\Delta/N_{k}}f\right)\left(N_{k}\Psi_{I}^{-1}(x)\right)-f\left(N_{k}\Psi_{I}^{-1}(x)\right)\right\|_{H^{1}(\mathbb{H}^{d})} \\
		= & N_{k}^{\frac{d-2}{2}}\left(\int_{\mathbb{R}^{d}} \sum_{i,j=1}^{d}(\delta_{ij}+v_{i}v_{j})\partial_{i}\bigl((e^{\Delta/N_{k}}f)(N_{k}v)-f(N_{k}v)\bigr)\partial_{j}\bigl((e^{\Delta/N_{k}}f)(N_{k}v)-f(N_{k}v)\bigr)(1+|v|^{2})^{-1/2}\mathrm{d}v\right)^{1/2} \\
		\leq & \left(\int_{\mathbb{R}^{d}} (1+|v|^{2})^{1/2}\bigl|\nabla(e^{\Delta/N_{k}}f-f)(v)\bigr|^{2}\mathrm{d}v\right)^{1/2}.
	\end{align*}
	From \eqref{eq1.4v30}, it follows that
	\begin{align*}
		e^{\Delta/N_{k}}f(v) = e^{(1+1/N_{k})\Delta}\delta_{0}(v) = \frac{1}{(4\pi(1+1/N_{k}))^{d/2}} e^{-|v|^{2}/(4(1+1/N_{k}))},
	\end{align*}
and
	\begin{align*}
		\nabla(e^{\Delta/N_{k}}f-f)(v) = -\frac{1}{(4\pi(1+1/N_{k}))^{d/2}}\frac{v}{2(1+1/N_{k})}e^{-|v|^{2}/(4(1+1/N_{k}))} + \frac{v}{2(4\pi)^{d/2}}e^{-|v|^{2}/4}.
	\end{align*}
By the dominated convergence theorem, we have
	\begin{align*}
		\int_{\mathbb{R}^{d}} (1+|v|^{2})^{1/2}\bigl|\nabla(e^{\Delta/N_{k}}f-f)(v)\bigr|^{2}\mathrm{d}v \rightarrow 0, \quad k\rightarrow\infty,
	\end{align*}
	which implies that
	\begin{align*}
		N_{k}^{\frac{d-2}{2}}\left\|\left(e^{\Delta/N_{k}}f\right)\left(N_{k}\Psi_{I}^{-1}(x)\right)-f\left(N_{k}\Psi_{I}^{-1}(x)\right)\right\|_{H^{1}(\mathbb{H}^{d})} \rightarrow 0, \quad k\rightarrow\infty.
	\end{align*}
	Thus, we obtain \eqref{eq1.4v33}. By applying \eqref{eq1.4v33}, we see that \eqref{eq1.4v32} is simplified to
	\begin{align}\label{eq1.4v34}
		\left\| N_{k}^{-\frac{d+2}{2}} e^{N_{k}^{-2}\Delta_{g}}\delta_{\mathbf{0}} - N_{k}^{\frac{d-2}{2}} f\left(N_{k}\Psi_{I}^{-1}\cdot\right)\right\|_{H^{1}(\mathbb{H}^{d})} \rightarrow 0, \quad k\rightarrow\infty.
	\end{align}
	This implies that \eqref{eq1.4v30} leads to
	\begin{align*}
		N_{k}^{\frac{d-2}{2}} f(N_{k}v) = N_{k}^{-\frac{d+2}{2}} e^{N_{k}^{-2}\Delta_{g}}\delta_{0}(v),
	\end{align*}
	and
	\begin{align*}
		\left\| N_{k}^{-\frac{d+2}{2}} e^{N_{k}^{-2}\Delta_{g}}\delta_{\mathbf{0}} - N_{k}^{\frac{d-2}{2}} f\left(N_{k}\Psi_{I}^{-1}\cdot\right)\right\|_{H^{1}(\mathbb{H}^{d})} = N_{k}^{-\frac{d+2}{2}} \left\| \left(e^{N_{k}^{-2}\Delta_{g}}\delta_{\mathbf{0}}\right)(\Psi_{I}^{-1}(x)) - e^{N_{k}^{-2}\Delta_{g}}\delta_{\mathbf{0}}(x) \right\|_{H^{1}(\mathbb{H}^{d})}.
	\end{align*}
Therefore, \eqref{eq1.4v34} becomes a consequence of the heat kernel estimate, and we will demonstrate that
	\begin{align}\label{eq1.4v35}
		N_{k}^{-\frac{d+2}{2}} \left\| \left(e^{N_{k}^{-2}\Delta}\delta_{0}\right)(\Psi_{I}^{-1}(x)) - e^{N_{k}^{-2}\Delta_{g}}\delta_{0}(x) \right\|_{H^{1}(\mathbb{H}^{d})} \rightarrow 0, \quad k\rightarrow\infty.
	\end{align}
	Since $N_k \to \infty$ as $k \to \infty$, a classical result (for instance, see \cite{Li2012}, Theorem 11.1) implies that
	\begin{align}\label{eq1.4v36}
		e^{N_{k}^{-2}\Delta_{g}}\delta_{\mathbf{0}}(x) \rightarrow \left(e^{N_{k}^{-2}\Delta}\delta_{0}\right)(\Psi_{I}^{-1}(x)), \quad k\rightarrow\infty.
	\end{align}
	From the equivalence of $H^{1}(\mathbb{H}^{d})$ and $\dot{H}^{1}(\mathbb{H}^{d})$, we apply integration by parts to deduce
	\begin{align*}
		& N_{k}^{-(d+2)} \left\| \left(e^{N_{k}^{-2}\Delta}\delta_{0}\right)(\Psi_{I}^{-1}(x)) - e^{N_{k}^{-2}\Delta_{g}}\delta_{\mathbf{0}}(x) \right\|_{H^{1}(\mathbb{H}^{d})}^{2} \\
		= & N_{k}^{-(d+2)} \left\langle \left(e^{N_{k}^{-2}\Delta}\delta_{0}\right)(\Psi_{I}^{-1}(x)) - e^{N_{k}^{-2}\Delta_{g}}\delta_{\mathbf{0}}(x), \Delta_{g}\left(\left(e^{N_{k}^{-2}\Delta}\delta_{0}\right)(\Psi_{I}^{-1}(x)) - e^{N_{k}^{-2}\Delta_{g}}\delta_{\mathbf{0}}(x)\right)\right\rangle_{L^{2}(\mathbb{H}^{d})\times L^{2}(\mathbb{H}^{d})}.
	\end{align*}
	Next, we introduce polar coordinates. Let $r = d_{g}(\mathbf{0},x)$ for $x = (\sqrt{1+|v|^{2}},v) \in \mathbb{H}^{d}$, then
\begin{align*}
\cosh r = \cosh d_{g}(\mathbf{0},x) = [\mathbf{0},x] = \sqrt{1+|v|^{2}} \Rightarrow |v|^{2} = \sinh^{2} r,
\end{align*}
and so $|\Psi_{I}^{-1}(x)|^{2}=|v|^{2}=\sinh^{2}r$. Hence,
	\begin{align*}
		\left(e^{N_{k}^{-2}\Delta}\delta_{0}\right)(\Psi_{I}^{-1}(x)) = \frac{1}{(4\pi N_{k}^{-2})^{d/2}} e^{-\frac{|\Psi_{I}^{-1}(x)|^{2}}{4N_{k}^{-2}}} = \frac{N_{k}^{d}}{(4\pi)^{d/2}} e^{-\frac{N_{k}^{2}\sinh^{2}r}{4}}.
	\end{align*}
	 If $d=2m+1$,
	\begin{align*}
		e^{N_{k}^{-2}\Delta_{g}}\delta_{\mathbf{0}}(x) = \frac{(-1)^{\frac{d-1}{2}}}{(2\pi)^{\frac{d-1}{2}}} \frac{e^{-\frac{(d-1)^{2}N_{k}^{-2}}{4}}}{(4\pi N_{k}^{-2})^{1/2}} \left(\frac{1}{\sinh r}\frac{\partial}{\partial r}\right)^{\frac{d-1}{2}} e^{-\frac{N_{k}^{2}r^{2}}{4}},
	\end{align*}
	if $d=2m+2$,
	\begin{align*}
		e^{N_{k}^{-2}\Delta_{g}}\delta_{\mathbf{0}}(x) = \frac{(-1)^{\frac{d}{2}-1}}{2(2\pi)^{\frac{d}{2}-1}} N_{k}^{3} e^{-\frac{(d-1)^{2}N_{k}^{-2}}{4}} \left(\frac{1}{\sinh r}\frac{\partial}{\partial r}\right)^{\frac{d}{2}-1} \int_{r}^{\infty} \frac{s e^{-s^{2}N_{k}^{2}/4}}{(\cosh s - \cosh r)^{1/2}} \mathrm{d}s.
	\end{align*}
	For the radial case, $\Delta_{g}=\partial_{r}^{2}+\frac{(d-1)\cosh r}{\sinh r}\partial_{r}$. We apply $\Delta_{g}$ to the two quantities $e^{N_{k}^{-2}\Delta_{g}}\delta_{0}(x)$ and $\left(e^{N_{k}^{-2}\Delta}\delta_{0}\right)(\Psi_{I}^{-1}(x)$ sequentially, and we obtain the principal parts as follows:
	\begin{align*}
		\Delta_{g}\left(\left(e^{N_{k}^{-2}\Delta}\delta_{0}\right)(\Psi_{I}^{-1}(x))\right) = &\left(\frac{N_{k}^{4}}{4}(\sinh r\cosh r)^{2} - \frac{N_{k}^{2}}{2}(\cosh^{2}r+\sinh^{2}r)\right)\left(e^{N_{k}^{-2}\Delta}\delta_{0}\right)(\Psi_{I}^{-1}(x)) \\
		&- \frac{N_{k}^{2}}{2}(d-1)\cosh^{2}r\left(e^{N_{k}^{-2}\Delta}\delta_{0}\right)(\Psi_{I}^{-1}(x)).
	\end{align*}
	
 We now give a comprehensive analysis for the Laplacian of heat kernels. For the odd dimensional case, the heat kernel has a simple formula, while is a non-local term for the case of even dimension. For $d=2m+1$, the differentiation with respect to $r$ gives
	\[
	\begin{aligned}
		\pa_r  e^{N_{k}^{-2}\Delta_{g}}\delta_{\mathbf{0}}(x) &=  \frac{(-1)^{\frac{d-1}{2}}}{(2\pi)^{\frac{d-1}{2}}} \frac{e^{-\frac{(d-1)^{2}N_{k}^{-2}}{4}}}{(4\pi N_{k}^{-2})^{1/2}} \sum_{j=1}^m\left(\frac{1}{\sinh r}\frac{\partial}{\partial r}\right)^{j-1} \left( - \frac{\cosh r}{\sinh r}\left(\frac{1}{\sinh r}\frac{\partial}{\partial r}\right)^{m-j+1}\right)e^{-\frac{N_{k}^{2}r^{2}}{4}}\\
		& \quad + \frac{(-1)^{\frac{d-1}{2}}}{(2\pi)^{\frac{d-1}{2}}} \frac{e^{-\frac{(d-1)^{2}N_{k}^{-2}}{4}}}{(4\pi N_{k}^{-2})^{1/2}}\left(\frac{1}{\sinh r}\frac{\partial}{\partial r}\right)^{\frac{d-1}{2}} \left(-\frac{N_k r}{2} e^{-\frac{N_{k}^{2}r^{2}}{4}} \right).
	\end{aligned}	
	\]
	We note that the differentiation for $\frac{1}{\sinh r}$ has no effect for $N_k$. In the following, we mainly handle the limit as $k \to \infty$. So we can split
	\[
	\pa_r  e^{N_{k}^{-2}\Delta_{g}}\delta_{\mathbf{0}}(x) = -\frac{N_k r}{2}e^{N_{k}^{-2}\Delta_{g}}\delta_{\mathbf{0}}(x) + \mathcal{O}(N_{k}^{2})
	\]
	Here, $\mathcal{O}(N_{k}^{2})$ denotes the lower-order terms of $N_{k}$ satisfying $\|\mathcal{O}(N_{k}^{2})\|_{L^{2}(\mathbb{H}^{d})}\lesssim_{d}N_{k}^{2}$.
	Due to the formula of $\Delta_g$, we obtain 
	\begin{align*}
		\Delta_{g} e^{N_{k}^{-2}\Delta_{g}}\delta_{\mathbf{0}}(x) = \left(\frac{N_{k}^{4}}{4}r^{2} - \frac{N_{k}^{2}}{2}\right)e^{N_{k}^{-2}\Delta_{g}}\delta_{\mathbf{0}}(x) - \frac{N_{k}^{2}}{2}(d-1)r\frac{\cosh r}{\sinh r}e^{N_{k}^{-2}\Delta_{g}}\delta_{\mathbf{0}}(x) +  \mathcal{O}(N_{k}^{2}),
	\end{align*}
	
	For $d=2m+2$, then an explicit computation gives that
	\[
	\begin{aligned}
		\pa_r  e^{N_{k}^{-2}\Delta_{g}}\delta_{\mathbf{0}}(x)& = \frac{(-1)^{\frac{d}{2}-1}}{2(2\pi)^{\frac{d}{2}-1}} N_{k}^{3} e^{-\frac{(d-1)^{2}N_{k}^{-2}}{4}} \sum_{j=1}^m\left(\frac{1}{\sinh r}\frac{\partial}{\partial r}\right)^{j-1} \left( - \frac{\cosh r}{\sinh r}\left(\frac{1}{\sinh r}\frac{\partial}{\partial r}\right)^{m-j+1}\right) \\
		&\quad \times \int_{r}^{\infty} \frac{s e^{-s^{2}N_{k}^{2}/4}}{(\cosh s - \cosh r)^{1/2}} \mathrm{d}s+ \frac{(-1)^{\frac{d}{2}-1}}{2(2\pi)^{\frac{d}{2}-1}} N_{k}^{3} e^{-\frac{(d-1)^{2}N_{k}^{-2}}{4}} \left(\frac{1}{\sinh r}\frac{\partial}{\partial r}\right)^{\frac{d}{2}-1}\\
		&\quad \times \left( - \lim_{s\to r^+} \frac{s e^{-s^2N_k^2/4}}{(\cosh s - \cosh r)^{1/2}} + \int_{r}^{\infty} \frac{s e^{-s^2N_k^2/4} \sinh r}{2 (\cosh s - \cosh r)^{3/2}} \, \mathrm{d}s\right)
	\end{aligned}
	\]
	By a careful asymptotic analysis near $r$, the singularity between the above two terms on the right-hand side is canceled. Then by the theory of Bessel function, we can obtain
	\[
	\begin{aligned}
		\pa_r  e^{N_{k}^{-2}\Delta_{g}}\delta_{\mathbf{0}}(x)& = \frac{(-1)^{\frac{d}{2}-1}}{2(2\pi)^{\frac{d}{2}-1}} N_{k}^{3} e^{-\frac{(d-1)^{2}N_{k}^{-2}}{4}} \sum_{j=1}^m\left(\frac{1}{\sinh r}\frac{\partial}{\partial r}\right)^{j-1} \left( - \frac{\cosh r}{\sinh r}\left(\frac{1}{\sinh r}\frac{\partial}{\partial r}\right)^{m-j+1}\right) \\
		&\quad \times \int_{r}^{\infty} \frac{s e^{-s^{2}N_{k}^{2}/4}}{(\cosh s - \cosh r)^{1/2}} \mathrm{d}s+ \frac{(-1)^{\frac{d}{2}-1}}{2(2\pi)^{\frac{d}{2}-1}} N_{k}^{3} e^{-\frac{(d-1)^{2}N_{k}^{-2}}{4}} \left(\frac{1}{\sinh r}\frac{\partial}{\partial r}\right)^{\frac{d}{2}-1}\\
		&\quad \times  \frac{-r  e^{-r^2N_k^2/4}}{\sqrt{\sinh r}},
	\end{aligned} 
	\]
	and then by the definition of $\Delta_g$,
	\begin{align*}
		\Delta_{g} e^{N_{k}^{-2}\Delta_{g}}\delta_{\mathbf{0}}(x) &= \frac{(-1)^{\frac{d}{2}-1}}{2(2\pi)^{\frac{d}{2}-1}} N_{k}^{3} e^{-\frac{(d-1)^{2}N_{k}^{-2}}{4}} \bigg[ -\frac{N_k^2}{2} \left(\frac{1}{\sinh r}\frac{\partial}{\partial r}\right)^{\frac{d}{2}-1}  \int_{r}^{\infty} \frac{s e^{-s^{2}N_{k}^{2}/4}}{(\cosh s - \cosh r)^{1/2}} \mathrm{d}s \\
		&\quad- \frac{2d-5}{2} \left(\frac{1}{\sinh r}\frac{\partial}{\partial r}\right)^{\frac{d}{2}-1}  \left( \frac{r \cosh r \, e^{-r^2 N_k^2/4}}{(\sinh r)^{3/2}} \right) \bigg]\\
		& = -\frac{N_k^2}{2} e^{N_{k}^{-2}\Delta_{g}}\delta_{\mathbf{0}}(x) -\frac{2d-5}{2} \left(\frac{1}{\sinh r}\frac{\partial}{\partial r}\right)^{\frac{d}{2}-1}  \left( \frac{r \cosh r \, e^{-r^2 N_k^2/4}}{(\sinh r)^{3/2}} \right).
	\end{align*}
 To estimate \eqref{eq1.4v35}, it suffices to prove that the highest-order part in $N_{k}$ of
\eqref{eq1.4v35} tends to vanish as $N_{k}\rightarrow\infty$, since the lower-order terms are bounded by $o(1/N_k)$. By changing the variable $r$ to $r/N_k$, we obtain the following result for the case of $d=2m+1$:
	\begin{align*}
		& N_{k}^{-(d+2)}\left\|\left(\frac{N_{k}^{4}}{4}(\sinh r\cosh r)^{2}\right)\left(e^{N_{k}^{-2}\Delta}\delta_{0}\right)(\Psi_{I}^{-1}(\cdot))-\left(\frac{N_{k}^{4}}{4}r^{2}\right)e^{N_{k}^{-2}\Delta_{g}}\delta_{0}\right\|_{L^{2}(\mathbb{H}^{d})}^{2} \\
		= & C_{d}N_{k}^{-(d+2)}\frac{1}{16}\int_{0}^{\infty}\Bigg[N_{k}^{2}\left(N_{k}^{2}\sinh^{2}\left(\frac{r}{N_{k}}\right)\cosh^{2}\left(\frac{r}{N_{k}}\right)\frac{N_{k}^{d}}{(4\pi)^{d/2}}e^{-\frac{N_{k}^{2}\sinh^{2}(r/N_{k})}{4}}\right) \\
		& - N_{k}^{2}r^{2}\frac{(-1)^{\frac{d-1}{2}}}{(2\pi)^{\frac{d-1}{2}}}\frac{e^{-\frac{(d-1)^{2}N_{k}^{2}}{4}}}{(4\pi N_{k}^{-2})^{1/2}}\left(\frac{N_{k}}{\sinh(r/N_{k})}\frac{\partial}{\partial r}\right)^{\frac{d-1}{2}}e^{-\frac{r^{2}}{4}}\Bigg]^{2}\frac{\sinh^{d-1}(r/N_{k})}{N_{k}}\mathrm{d}r,
	\end{align*}
	and for the case of $d=2m+2$:
	\begin{align*}
		& N_{k}^{-(d+2)}\left\|\left(\frac{N_{k}^{4}}{4}(\sinh r\cosh r)^{2}\right)\left(e^{N_{k}^{-2}\Delta}\delta_{0}\right)(\Psi_{I}^{-1}(x))-\left(\frac{N_{k}^{2}}{2}\right)e^{N_{k}^{-2}\Delta_{g}}\delta_{\mathbf{0}}\right\|_{L^{2}(\mathbb{H}^{d})}^{2} \\
		= & C_{d}N_{k}^{-(d+2)}\frac{1}{16}\int_{0}^{\infty}\Bigg[N_{k}^{2}\left(N_{k}^{2}\sinh^{2}\left(\frac{r}{N_{k}}\right)\cosh^{2}\left(\frac{r}{N_{k}}\right)\frac{N_{k}^{d}}{(4\pi)^{d/2}}e^{-\frac{N_{k}^{2}\sinh^{2}(r/N_{k})}{4}}\right) \\
		& - \frac{(-1)^{\frac{d}{2}-1}}{(2\pi)^{\frac{d}{2}-1}}\frac{N_{k}^{3}}{e^{\frac{(d-1)^{2}N_{k}^{-2}}{4}}}\left(\frac{N_{k}}{\sinh\left(\frac{r}{N_{k}}\right)}\frac{\partial}{\partial r}\right)^{\frac{d}{2}-1}\int_{\frac{r}{N_{k}}}^{\infty}\frac{se^{-s^{2}N_{k}^{2}/4}}{\sqrt{\cosh s-\cosh\left(\frac{r}{N_{k}}\right)}}\mathrm{d}s\Bigg]^{2}\frac{\sinh^{d-1}\left(\frac{r}{N_{k}}\right)}{N_{k}}\mathrm{d}r. 
	\end{align*} 
	Now, observe that
	\begin{align*}
		N_{k}\sinh\frac{r}{N_{k}} \rightarrow r, \quad k\rightarrow\infty,
	\end{align*}
	then by applying \eqref{eq1.4v36}, we can conclude that the integrand in
	\begin{multline}\label{eq1.4v37}
		N_{k}^{-(d+2)}\left\langle \left(e^{N_{k}^{-2}\Delta}\delta_{0}\right)(\Psi_{I}^{-1}(\cdot)) - e^{N_{k}^{-2}\Delta_{g}}\delta_{\mathbf{0}}(\cdot),\right.\\
		\left. N_{k}^{4}\left((\sinh r\cosh r)^{2}\left(e^{N_{k}^{-2}\Delta}\delta_{0}\right)(\Psi_{I}^{-1}(\cdot)) - r^{2}e^{N_{k}^{-2}\Delta_{g}}\delta_{\mathbf{0}}(\cdot)\right)\right\rangle_{L^{2}(\mathbb{H}^{d})\times L^{2}(\mathbb{H}^{d})}
	\end{multline}
vanishes as $k\to\infty$. In terms of the factors $e^{-r^{2}/4}$ and $e^{-\sinh^{2}r/4}$, we find that the above integrands can be controlled by some Gaussian functions independent of $N_{k}$ for sufficiently large $k$. Then \eqref{eq1.4v37} tends to zero as $k\to\infty$, which is the desired result. Regarding the remaining terms of $N_{k}^{2}$ in $\left(e^{N_{k}^{-2}\Delta}\delta_{0}\right)(\Psi_{I}^{-1}(x))$ and $e^{N_{k}^{-2}\Delta_{g}}\delta_{0}(x)$, they are of lower order in $N_{k}$, so that their $L^{2}$-norms tend to zero as $k\to\infty$. Thus, we have proved \eqref{eq1.4v35}.
	
	Since $g_{k}$ is absent from the frames $\mathcal{O}^{\alpha}$ for $\alpha \leq A$, it follows from \eqref{eq1.4v28} that
\begin{align*}
\frac{\delta}{2} \leq &\lim_{k\to\infty} \left|\left\langle \Pi_{-t_{k},h_{k}^{-1}}g_{k}, N_{k}^{-\frac{d-2}{2}} e^{N_{k}^{-2}\Delta_{g}}\delta_{\mathbf{0}} \right\rangle_{L^{2}\times L^{2}(\mathbb{H}^{d})}\right|\\
\leq& \lim_{k\to\infty} \left|\left\langle \nabla_{g}\Pi_{-t_{k},h_{k}^{-1}}g_{k}, \nabla_{g}\pi_{h_{k}^{-1}}\tilde{f}_{\mathcal{O}_{k}'}\right\rangle_{L^{2}\times L^{2}(\mathbb{H}^{d})}\right|\\
&+ \lim_{k\to\infty} \sup \left|\left\langle \nabla_{g}\Pi_{-t_{k},h_{k}^{-1}}g_{k}, \nabla_{g}\left(\Pi_{-t_{k},h_{k}^{-1}}\tilde{f}_{\mathcal{O}_{k}'} - N_{k}^{-\frac{d+2}{2}}e^{N_{k}^{-2}\Delta_{g}}\delta_{\mathbf{0}}\right)\right\rangle_{L^{2}\times L^{2}(\mathbb{H}^{d})}\right| \\
=& \lim_{k\to\infty} \sup \left|\left\langle \nabla_{g}g_{k}, \nabla_{g}\tilde{f}_{\mathcal{O}_{k}'}\right\rangle_{L^{2}\times L^{2}(\mathbb{H}^{d})}\right|.
\end{align*}
Thus, we have $\mathcal{O}'\perp\mathcal{O}^{\alpha}$, for $\alpha\leq A$.
	
	Now, in the case where $\mathcal{O}'$ is a hyperbolic frame, we take $\psi \in H^{1}(\mathbb{H}^{d})$ to be the weak limit of $\Pi_{-t_{k},h_{k}^{-1}}g_{k}$. Then, passing to a subsequence, we may assume that $\forall \varphi \in H^{1}(\mathbb{H}^{d})$,
	\begin{align*}
		\left\langle \nabla_{g}\left(\Pi_{-t_{k},h_{k}^{-1}}g_{k} - \psi\right), \nabla_{g}\varphi\right\rangle_{L^{2}\times L^{2}} = \left\langle \nabla_{g}\left(g_{k} - \Pi_{-t_{k}^{-1},h_{k}}\psi\right), \nabla_{g}\Pi_{-t_{k}^{-1},h_{k}}\varphi\right\rangle_{L^{2}\times L^{2}} \to 0, \quad k\to\infty.
	\end{align*}
And	so $g_{k}^{\prime}=g_{k}-\Pi_{-t_{k}^{-1},h_{k}}\psi$ is absent from $\mathcal{O}^{\prime}$.
	
	Specifically, it follows \eqref{eq1.4v26} that we obtain
	\begin{align*}
		\frac{\delta}{2} \leq& \lim_{k\to\infty} \left|\left\langle \Pi_{-t_{k},h_{k}^{-1}}g_{k}, \Delta_{g}N_{k}^{-\frac{d+2}{2}}\left(e^{N_{k}^{-2}\Delta_{g}}\delta_{\mathbf{0}}\right)\right\rangle_{L^{2}\times L^{2}}\right|\\
		\leq &\left|\left\langle \psi, \Delta_{g}\left(N^{-\frac{d+2}{2}}e^{N_{k}^{-2}\Delta_{g}}\delta_{\mathbf{0}}\right)\right\rangle_{L^{2}\times L^{2}}\right| \\
		\lesssim& \|\nabla_{g}\psi\|_{L^{2}(\mathbb{H}^{d})},
	\end{align*}
	which yields \eqref{le.4v71}.
	
	Finally, in order to prove that $(g'_k)_k$ is orthogonal to all frames $\mathcal{O}^\alpha$, $1 \leq \alpha \leq A$, it is enough to demonstrate that  $\widetilde{\psi}_{\mathcal{O}_k'}$ is orthogonal to each $\mathcal{O}^\alpha$, $1 \leq \alpha \leq A$. However, because $\mathcal{O}^\alpha \perp \mathcal{O}^k$, $1 \leq \alpha \leq A$. Lemma \ref{le.2v3} (ii) implies that $\widetilde{\psi}_{\mathcal{O}_k'}$ is indeed absent from each $\mathcal{O}^\alpha$, $1 \leq \alpha \leq A$.
	
	If $\mathcal{O}'$ is a Euclidean frame, we begin by taking $R > 0$ and defining
	\begin{align*}
		\phi_{k}^{R}(v) = N_{k}^{-\frac{d-2}{2}} \eta\left(\frac{v}{R}\right) \left(\Pi_{-t_{k},h_{k}^{-1}}g_{k}\right)\left(\Psi_{I}\left(\frac{v}{N_{k}}\right)\right),
	\end{align*}
where $\eta$ is the smooth cut-off function from \eqref{eq1.3v1}. Differentiating, we obtain
	\begin{align*}
		\nabla\phi_{k}^{R}(v)=&N_{k}^{-\frac{d-2}{2}}R^{-1}(\nabla\eta)\left(\frac{v}{R}\right)\left(\Pi_{-t_{k},h_{k}^{-1}}g_{k}\right)\left(\Psi_{I}\left(\frac{v}{N_{k}}\right)\right)\\
		&+N_{k}^{-\frac{d-2}{2}}N_{k}^{-1}\eta\left(\frac{v}{R}\right)\left(\nabla\left(\Pi_{-t_{k},h_{k}^{-1}}g_{k}\right)(\Psi_{I}\cdot)\right)\left(\frac{v}{N_{k}}\right),
	\end{align*}
We estimate the two terms separately. For the first term, H\"older's inequality gives
	\begin{align*}
		& \left\| N_{k}^{-\frac{d-2}{2}}R^{-1}(\nabla\eta)\left(\frac{v}{R}\right)\left(\Pi_{-t_{k},h_{k}^{-1}}g_{k}\right)\left(\Psi_{I}\left(\frac{v}{N_{k}}\right)\right)\right\|_{L^{2}(\mathbb{R}^{d})} \\
		\leq & N_{k}^{-\frac{d-2}{2}}R^{-1}\left\|(\nabla\eta)\left(\frac{v}{R}\right)\right\|_{L^{d}(\mathbb{R}^{d})} \left\|\left(\Pi_{-t_{k},h_{k}^{-1}}g_{k}\right)\left(\Psi_{I}\left(\frac{v}{N_{k}}\right)\right)\right\|_{L^{2^{*}}(|v|\leq 2R)} \\
		\lesssim & \left(1+\frac{R^{2}}{N_{k}^{2}}\right)^{\frac{d-2}{4d}} \left(\int_{|v|\leq\frac{2R}{N_{k}}} \left|\left(\Pi_{-t_{k},h_{k}^{-1}}g_{k}\right)(\Psi_{I}(v))\right|^{\frac{2d}{d-2}} (1+|v|^{2})^{-\frac{1}{2}} \mathrm{d}v\right)^{\frac{d-2}{2d}} \\
		\lesssim & \left(1+\frac{R^{2}}{N_{k}^{2}}\right)^{\frac{d-2}{4d}} \left\|\Pi_{-t_{k},h_{k}^{-1}}g_{k}\right\|_{L^{2^{*}}(\mathbb{H}^{d})} \lesssim \|\nabla_{g}g_{k}\|_{L^{2}(\mathbb{H}^{d})}, 
	\end{align*}
for sufficiently large $k$. For the second term, we have
	\begin{align*}
		\left\| N_{k}^{-\frac{d-2}{2}}N_{k}^{-1}\eta\left(\frac{v}{R}\right)\left(\nabla\left(\Pi_{-t_{k},h_{k}^{-1}}g_{k}\right)(\Psi_{I}\cdot)\right)\left(\frac{v}{N_{k}}\right)\right\|_{L^{2}(\mathbb{R}^{d})} \leq \left\| \nabla\left(\left(\Pi_{-t_{k},h_{k}^{-1}}g_{k}\right)(\Psi_{I}(v))\right) \right\|_{L^{2}(|v|\leq \frac{2R}{N_{k}})},
	\end{align*}
	For $x=\Psi_{I}(v)$ with $|v|\leq 2R/N_{k}$, we note that
	\begin{align*}
		|\nabla_{g}\left(\left(\Pi_{-t_{k},h_{k}^{-1}}g_{k}\right)(x)\right)|^{2} =&\sum_{i,j=1}^{d} (\delta_{ij}+v_{i}v_{j})\partial_{i}\left(\left(\Pi_{-t_{k},h_{k}^{-1}}g_{k}\right)(\Psi_{I}(v))\right)\partial_{j}\left(\left(\Pi_{-t_{k},h_{k}^{-1}}g_{k}\right)(\Psi_{I}(v))\right) \\
		\geq& \left|\nabla\left(\left(\Pi_{-t_{k},h_{k}^{-1}}g_{k}\right)(\Psi_{I}(v))\right)\right|^{2} - |v|^{2}\left|\nabla\left(\left(\Pi_{-t_{k},h_{k}^{-1}}g_{k}\right)(\Psi_{I}(v))\right)\right|^{2} \\
		\geq& \left(1-\frac{4R^{2}}{N_{k}^{2}}\right)\left|\nabla\left(\left(\Pi_{-t_{k},h_{k}^{-1}}g_{k}\right)(\Psi_{I}(v))\right)\right|^{2} \geq \frac{1}{2}\left|\nabla\left(\left(\Pi_{-t_{k},h_{k}^{-1}}g_{k}\right)(\Psi_{I}(v))\right)\right|^{2},
	\end{align*}
for sufficiently large $k$. Consequently,
	\begin{align*}
		\left\| \nabla\left(\left(\Pi_{-t_{k},h_{k}^{-1}}g_{k}\right)(\Psi_{I}(v))\right) \right\|_{L^{2}(|v|\leq \frac{2R}{N_{k}})} \lesssim& \left(1+\frac{R^{2}}{N_{k}^{2}}\right)^{\frac{1}{4}} \left(\int_{\mathbb{R}^{d}} \left|\left(\nabla_{g}\left(\Pi_{-t_{k},h_{k}^{-1}}g_{k}\right)\right)(\Psi_{I}(v))\right|^{2} (1+|v|^{2})^{-\frac{1}{2}} \mathrm{d}v\right)^{\frac{1}{2}} \\
\lesssim& \|\nabla_{g}g_{k}\|_{L^{2}(\mathbb{H}^{d})}, 
\end{align*}
$\text{when } k \text{ large enough}$. This yields
	\begin{align*}
		\limsup _{k \rightarrow \infty}\|\nabla\phi_{k}^{R}\|_{L^{2}(\mathbb{R}^{d})} \lesssim  \limsup _{k \rightarrow \infty}\|\nabla_{g}g_{k}\|_{L^{2}(\mathbb{H}^{d})},
	\end{align*}
	and so $\{\phi_{k}^{R}\}$ contains a subsequence that is bounded in $\dot{H}^{1}(\mathbb{R}^{d})$ uniformly in the parameter $R>0$. By passing again to a subsequence, we can find a weak limit $\phi^{R}\in\dot{H}^{1}(\mathbb{R}^{d})$. The uniformity of the bound in $R>0$ allows us to let $R\to\infty$ and extract a weak limit $\phi\in\dot{H}^{1}(\mathbb{R}^{d})$ satisfying
	\begin{align*}
		\phi^{R} \to \phi \quad \text{in } \dot{H}^{1}(\mathbb{R}^{d}), \quad \text{as } R\to\infty.
	\end{align*}
	Next, we will prove $g'_k = g_k - \widetilde{\phi}_{\mathcal{O}'_k}$ is absent from $\mathcal{O}'$, it follows the boundedness of $T_{N'_k}$ that we will prove $\forall \psi \in C^{\infty}_{0}(\mathbb{R}^{d})$,
	\begin{align}\label{eq1.4v38}
		\left\langle \nabla_{g}\left(g_{k} - \tilde{\phi}_{\mathcal{O}_{k}'}\right), \nabla_{g}\tilde{\psi}_{\mathcal{O}_{k}'}\right\rangle_{L^{2}\times L^{2}(\mathbb{H}^{d})} \to 0, \quad k\to\infty.
	\end{align}
Since $\psi \in C_{0}^{\infty}(\mathbb{R}^{d})$, we have 
	\begin{align}\label{eq1.4v39}
		\left\| T_{N_{k}}\psi - N_{k}^{\frac{d-2}{2}}\psi\left(N_{k}\Psi_{I}^{-1}\cdot\right)\right\|_{H^{1}(\mathbb{H}^{d})} \to 0, \quad k\to\infty,
	\end{align}
Then, using \eqref{eq1.4v39} and Lemma \ref{le.4v4} (iii) , we obtain
	\begin{align*}
		& \left\langle \nabla_{g}\left(g_{k} - \bar{\phi}_{\mathcal{O}_{k}'}\right), \nabla_{g}\bar{\psi}_{\mathcal{O}_{k}'}\right\rangle_{L^{2}\times L^{2}(\mathbb{H}^{d})} \\
		= & \left\langle \nabla_{g}\Pi_{-t_{k},h_{k}^{-1}}g_{k}, \nabla_{g}\left(N_{k}^{\frac{d-2}{2}}\psi\left(N_{k}\Psi_{I}^{-1}\cdot\right)\right)\right\rangle_{L^{2}\times L^{2}(\mathbb{H}^{d})} - \left\langle \nabla\phi, \nabla\psi\right\rangle_{L^{2}\times L^{2}(\mathbb{H}^{d})} + o_{k}(1), \quad k\to\infty.
	\end{align*}
	Therefore, \eqref{eq1.4v38} is the same as
	\begin{align*}
		\left\langle \nabla_{g}\Pi_{-t_{k},h_{k}^{-1}}g_{k}, \nabla_{g}\left(N_{k}^{\frac{d-2}{2}}\psi\left(N_{k}\Psi_{I}^{-1}\cdot\right)\right)\right\rangle_{L^{2}\times L^{2}(\mathbb{H}^{d})} \to \left\langle \nabla\phi, \nabla\psi\right\rangle_{L^{2}\times L^{2}(\mathbb{R}^{d})}, \quad k\to\infty.
	\end{align*}
	Then
	\begin{align*}
		& \left\langle \nabla_{g}\Pi_{-t_{k},h_{k}^{-1}}g_{k}, \nabla_{g}\left(N_{k}^{\frac{d-2}{2}}\psi\left(N_{k}\Psi_{I}^{-1}\cdot\right)\right)\right\rangle_{L^{2}\times L^{2}(\mathbb{H}^{d})} \\
		= & \int_{\mathbb{R}^{d}} \sum_{i,j=1}^{d}(\delta_{ij}+v_{i}v_{j})\partial_{i}\left(\left(\Pi_{-t_{k},h_{k}^{-1}}g_{k}\right)(\Psi_{I}(v))\right)\partial_{j}\left(N_{k}^{\frac{d-2}{2}}\psi(N_{k}v)\right)(1+|v|^{2})^{-\frac{1}{2}}\mathrm{d}v,\\
		= & \int_{\mathbb{R}^{d}} \sum_{i,j=1}^{d}\left(\delta_{ij}+\frac{v_{i}v_{j}}{N_{k}^{2}}\right)\partial_{i}\phi_{k}^{R}(v)\partial_{j}\psi(v)\left(1+\left|\frac{v}{N_{k}}\right|^{2}\right)^{-\frac{1}{2}}\mathrm{d}v \\
		& - \int_{\mathbb{R}^{d}}\sum_{i,j=1}^{d} \left(\delta_{ij}+\frac{v_{i}v_{j}}{N_{k}^{2}}\right)\partial_{i}\left(\phi_{k}^{R}(v)-N_{k}^{-\frac{d-2}{2}}\left(\Pi_{-t_{k},h_{k}^{-1}}g_{k}\right)\left(\Psi_{I}\left(\frac{v}{N_{k}}\right)\right)\right)\partial_{j}\psi(v)\left(1+\left|\frac{v}{N_{k}}\right|^{2}\right)^{-\frac{1}{2}}\mathrm{d}v \\
		= & I_{1}+I_{2},
	\end{align*}
	it is worth noting that
	\begin{align*}
		\int_{\mathbb{R}^{d}} \sum_{i,j=1}^{d}\left(\delta_{ij}+\frac{v_{i}v_{j}}{N_{k}^{2}}\right)\partial_{i}\phi_{k}^{R}(v)\partial_{j}\psi(v)\left(1+\left|\frac{v}{N_{k}}\right|^{2}\right)^{-\frac{1}{2}}\mathrm{d}v \\
		= \int_{\mathbb{R}^{d}} \sum_{i=1}^{d}\partial_{i}\phi_{k}^{R}(v)\left(\partial_{i}\psi(v)+\frac{v_{i}(v\cdot\nabla\psi)}{N_{k}^{2}}\right)\left(1+\left|\frac{v}{N_{k}}\right|^{2}\right)^{-\frac{1}{2}}\mathrm{d}v.
	\end{align*}
Consequently, the dominated convergence theorem gives, for each $i=1,2,\cdots,d$,
	\begin{align*}
		\sum_{i=1}^{d}\left(\partial_{i}\psi(v)+\frac{v_{i}(v\cdot\nabla\psi)}{N_{k}^{2}}\right)\left(1+\left|\frac{v}{N_{k}}\right|^{2}\right)^{-\frac{1}{2}} - \sum_{i=1}^{d}\partial_{i}\psi(v) \rightarrow 0 \quad \text{in } L^{2}(\mathbb{R}^{d}), \quad k\rightarrow\infty.
	\end{align*}
	Now, for $i=1,2,\cdots,d$, we also have
	\begin{align*}
		\partial_{i}\phi_{k}^{R}(v) \rightharpoonup \partial_{i}\phi^{R}(v) \quad \text{in } L^{2}(\mathbb{R}^{d}), \quad k\to\infty, \quad \forall R>0.
	\end{align*}
	This implies that for any $R>0$, 
	\begin{align}\label{eq1.4v40}
		\int_{\mathbb{R}^{d}}\sum_{i,j=1}^{d} \left(\delta_{ij}+\frac{v_{i}v_{j}}{N_{k}^{2}}\right)\partial_{i}\phi_{k}^{R}(v)\partial_{j}\psi(v)\left(1+\left|\frac{v}{N_{k}}\right|^{2}\right)^{-\frac{1}{2}}\mathrm{d}v \rightarrow \int_{\mathbb{R}^{d}} \nabla\phi^{R}(v)\cdot\nabla\psi(v)\mathrm{d}v, \quad k\rightarrow\infty.
	\end{align}
	Also, observe that
	\begin{align*}
		& \int_{\mathbb{R}^{d}}\sum_{i,j=1}^{d} \left(\delta_{ij}+\frac{v_{i}v_{j}}{N_{k}^{2}}\right)\partial_{i}\left(\phi_{k}^{R}(v)-N_{k}^{-\frac{d-2}{2}}\left(\Pi_{-t_{k},h_{k}^{-1}}g_{k}\right)\left(\Psi_{I}\left(\frac{v}{N_{k}}\right)\right)\right)\partial_{j}\psi(v)\left(1+\left|\frac{v}{N_{k}}\right|^{2}\right)^{-\frac{1}{2}}\mathrm{d}v \\
		= & \int_{\mathbb{R}^{d}}\sum_{i,j=1}^{d} \left(\delta_{ij}+\frac{v_{i}v_{j}}{N_{k}^{2}}\right)\frac{1}{R}(\partial_{i}\eta)\left(\frac{v}{R}\right)N_{k}^{-\frac{d-2}{2}}\left(\Pi_{-t_{k},h_{k}^{-1}}g_{k}\right)\left(\Psi_{I}\left(\frac{v}{N_{k}}\right)\right)\partial_{j}\psi(v)\left(1+\left|\frac{v}{N_{k}}\right|^{2}\right)^{-\frac{1}{2}}\mathrm{d}v \\
		& + \int_{\mathbb{R}^{d}}\sum_{i,j=1}^{d} \left(\delta_{ij}+\frac{v_{i}v_{j}}{N_{k}^{2}}\right)\left(\eta\left(\frac{v}{R}\right)-1\right)\partial_{i}\left(N_{k}^{-\frac{d-2}{2}}\left(\Pi_{-t_{k},h_{k}^{-1}}g_{k}\right)\left(\Psi_{I}\left(\frac{v}{N_{k}}\right)\right)\right)\partial_{j}\psi(v)\left(1+\left|\frac{v}{N_{k}}\right|^{2}\right)^{-\frac{1}{2}}\mathrm{d}v.
	\end{align*}
Since $\psi \in C_{0}^{\infty}(\mathbb{R}^{d})$, we have
	\begin{align*}
		& \left|\int_{\mathbb{R}^{d}}\sum_{i,j=1}^{d} \left(\delta_{ij}+\frac{v_{i}v_{j}}{N_{k}^{2}}\right)\frac{1}{R}(\partial_{i}\eta)\left(\frac{v}{R}\right)N_{k}^{-\frac{d-2}{2}}\left(\Pi_{-t_{k},h_{k}^{-1}}g_{k}\right)\left(\Psi_{I}\left(\frac{v}{N_{k}}\right)\right)\partial_{j}\psi(v)\left(1+\left|\frac{v}{N_{k}}\right|^{2}\right)^{-\frac{1}{2}}\mathrm{d}v\right| \\
		\lesssim & \left(1+\frac{R^{2}}{N_{k}^{2}}\right)^{\frac{1}{2}} \int_{\mathbb{R}^{d}} \left|\frac{1}{R}(\nabla\eta)\left(\frac{v}{R}\right)N_{k}^{-\frac{d-2}{2}}\left(\Pi_{-t_{k},h_{k}^{-1}}g_{k}\right)\left(\Psi_{I}\left(\frac{v}{N_{k}}\right)\right)\nabla\psi(v)\right|\mathrm{d}v \\
		\lesssim & \left(1+\frac{R^{2}}{N_{k}^{2}}\right)^{\frac{1}{2}} \left\|\left(\Pi_{-t_{k},h_{k}^{-1}}g_{k}\right)(\Psi_{I}(v))\right\|_{L^{2^{*}}(|v|\leq 2R)} \|\nabla\psi\|_{L^{2}(|v|\geq R)}\\
		 \lesssim& \left(1+\frac{R^{2}}{N_{k}^{2}}\right)\|\nabla_{g}g_{k}\|_{L^{2}(\mathbb{H}^{d})}\|\nabla\psi\|_{L^{2}(|v|\geq R)},
	\end{align*}
Consequently, since $\psi$ has compact support, we obtain
	\begin{align}\label{eq1.4v41}
		\underset{k \to \infty} \limsup \int_{\mathbb{R}^{d}}\sum_{i,j=1}^{d} \left(\delta_{ij}+\frac{v_{i}v_{j}}{N_{k}^{2}}\right)\frac{1}{R}(\partial_{i}\eta)\left(\frac{v}{R}\right)N_{k}^{-\frac{d-2}{2}}\left(\Pi_{-t_{k},h_{k}^{-1}}g_{k}\right)\left(\Psi_{I}\left(\frac{v}{N_{k}}\right)\right)\partial_{j}\psi(v)\left(1+\left|\frac{v}{N_{k}}\right|^{2}\right)^{-\frac{1}{2}}\mathrm{d}v \to 0,
	\end{align}
	as $R \rightarrow \infty$. Since $\psi \in C_{0}^{\infty}(\mathbb{R}^{d})$, Let $\operatorname{supp}\psi \subset B(0,R_{0})$, $R_{0} > 0$. Then
	\begin{align*}
		& \left|\int_{\mathbb{R}^{d}}\sum_{i,j=1}^{d}\left(\delta_{ij}+\frac{v_{i}v_{j}}{N_{k}^{2}}\right)\left(\eta\left(\frac{v}{R}\right)-1\right)\partial_{i}\left(N_{k}^{-\frac{d-2}{2}}\left(\Pi_{-t_{k},h_{k}^{-1}}g_{k}\right)\left(\Psi_{I}\left(\frac{v}{N_{k}}\right)\right)\right)\partial_{j}\psi(v)\left(1+\left|\frac{v}{N_{k}}\right|^{2}\right)^{-\frac{1}{2}}\mathrm{d}v\right| \\
		\leq & N_{k}^{\frac{d-2}{2}}\int_{\mathbb{R}^{d}}\left(1+|v|^{2}\right)^{\frac{1}{2}}\left|\left(\eta\left(\frac{v}{R}\right)-1\right)\left(\nabla\left(\left(\Pi_{-t_{k},h_{k}^{-1}g_{k}}\right)(\Psi_{I})\right)\right)(v)\nabla(\psi(N_{k}v))\right|\mathrm{d}v \\
		\leq & N_{k}^{\frac{d-2}{2}}\left(1+\frac{R_{0}^{2}}{N_{k}^{2}}\right)^{\frac{1}{2}}\left|\left(\eta\left(\frac{N_{k}v}{R}\right)-1\right)\nabla(\psi(N_{k}v))\right|_{L^{2}(v|\leq\frac{R_{0}}{N_{k}})}\left\|\nabla\left(\left(\Pi_{-t_{k},h_{k}^{-1}}g_{k}\right)(\Psi_{I}\cdot)\right)(v)\right\|_{L^{2}(|v|\leq\frac{R_{0}}{N_{k}})} \\
		\leq & \left(1+\frac{R_{0}^{2}}{N_{k}^{2}}\right)^{\frac{1}{2}}\|\nabla\psi\|_{L^{2}(v)\geq R}\|\nabla_{g}g_{k}\|_{L^{2}(\mathbb{R}^{n})},
	\end{align*}
	 for large enough $k$, we can obtain
	\begin{align}\label{eq1.4v42}
		\lim_{R\to\infty}\lim_{k\to\infty} \int_{\mathbb{R}^{d}}\sum_{i,j=1}^{d} \left(\delta_{ij}+\frac{v_{i}v_{j}}{N_{k}^{2}}\right)\left(\eta\left(\frac{v}{R}\right)-1\right)\partial_{i}\left(N_{k}^{-\frac{d-2}{2}}\left(\Pi_{-t_{k},h_{k}^{-1}}g_{k}\right)\left(\Psi_{I}\left(\frac{v}{N_{k}}\right)\right)\right)\nonumber\\ \times\partial_{j}\psi(v)\left(1+\left|\frac{v}{N_{k}}\right|^{2}\right)^{-\frac{1}{2}}\mathrm{d}v = 0.
	\end{align}
	Thus, by \eqref{pr.4v51}, \eqref{eq1.4v41}, and \eqref{eq1.4v42}, we have
	\begin{align*}
		& \lim_{k\to\infty} \left\langle \nabla_{g}\left(\Pi_{-t_{k},h_{k}^{-1}}g_{k}\right), \nabla_{g}\left(N_{k}^{\frac{d-2}{2}}\psi\left(N_{k}\Psi_{I}^{-1}\cdot\right)\right)\right\rangle_{L^{2}\times L^{2}(\mathbb{H}^{d})} \\
		= & \lim_{k\to\infty} \int_{\mathbb{R}^{d}}\sum_{i,j=1}^{d} \left(\delta_{ij}+\frac{v_{i}v_{j}}{N_{k}^{2}}\right)\partial_{i}\phi_{k}^{R}(v)\partial_{j}\psi(v)\left(1+\left|\frac{v}{N_{k}}\right|^{2}\right)^{-\frac{1}{2}}\mathrm{d}v \\
		& + \lim_{k\to\infty} \int_{\mathbb{R}^{d}}\sum_{i,j=1}^{d} \left(\delta_{ij}+\frac{v_{i}v_{j}}{N_{k}^{2}}\right)\partial_{i}\left(N_{k}^{-\frac{d-2}{2}}\left(\Pi_{-t_{k},h_{k}^{-1}}g_{k}\right)\left(\Psi_{I}\left(\frac{v}{N_{k}}\right)\right)-\phi_{k}^{R}(v)\right)\partial_{j}\psi(v)\left(1+\left|\frac{v}{N_{k}}\right|^{2}\right)^{-\frac{1}{2}}\mathrm{d}v \\
		= & \lim_{R\to\infty}\lim_{k\to\infty} \int_{\mathbb{R}^{d}}\sum_{i,j=1}^{d} \left(\delta_{ij}+\frac{v_{i}v_{j}}{N_{k}^{2}}\right)\partial_{i}\phi_{k}^{R}(v)\partial_{j}\psi(v)\left(1+\left|\frac{v}{N_{k}}\right|^{2}\right)^{-\frac{1}{2}}\mathrm{d}v \\
		& + \lim_{R\to\infty}\lim_{k\to\infty} \int_{\mathbb{R}^{d}}\sum_{i,j=1}^{d} \left(\delta_{ij}+\frac{v_{i}v_{j}}{N_{k}^{2}}\right)\partial_{i}\left(N_{k}^{-\frac{d-2}{2}}\left(\Pi_{-t_{k},h_{k}^{-1}}g_{k}\right)\left(\Psi_{I}\left(\frac{v}{N_{k}}\right)\right)-\phi_{k}^{R}(v)\right)\partial_{j}\psi(v)\left(1+\left|\frac{v}{N_{k}}\right|^{2}\right)^{-\frac{1}{2}} \\
		= & \lim_{R\to\infty} \int_{\mathbb{R}^{d}} \nabla\phi^{R}(v)\cdot\nabla\psi(v)\mathrm{d}v = \int_{\mathbb{R}^{d}} \nabla\phi(v)\cdot\nabla\psi(v)\mathrm{d}v.
	\end{align*}
Therefore, $(g'_k)_k$ lies outside $\mathcal{O}'$. Now, we prove \eqref{le.4v71}. Applying \eqref{eq1.4v26} yields
	\begin{align*}
		\frac{\delta}{2} \leq& \limsup_{k\to\infty} \left|\left\langle \Pi_{-t_{k},h_{k}^{-1}}g_{k}, N_{k}^{-\frac{d-2}{2}}P_{N_{k}}\delta_{\mathbf{0}}\right\rangle_{L^{2}\times L^{2}(\mathbb{H}^{d})}\right|\\ = &\limsup_{k\to\infty} \left|\left\langle \nabla_{g}\Pi_{-t_{k},h_{k}^{-1}}g_{k}, \nabla_{g}\left(N_{k}^{-\frac{d+2}{2}}e^{N_{k}^{-2}\Delta_{g}}\delta_{\mathbf{0}}\right)\right\rangle_{L^{2}\times L^{2}(\mathbb{H}^{d})}\right|.
	\end{align*}
	It follows from \eqref{eq1.4v27} and \eqref{eq1.4v28} that
	\begin{align*}
		\frac{\delta}{2} \leq &\limsup_{k\to\infty} \left|\left\langle \nabla_{g}\Pi_{-t_{k},h_{k}^{-1}}g_{k}, \nabla_{g}\left(N_{k}^{-\frac{d+2}{2}}e^{N_{k}^{-2}}\delta_{\mathbf{0}}\right)\right\rangle_{L^{2}\times L^{2}(\mathbb{H}^{d})}\right|\\
		= &\limsup_{k\to\infty} \left|\left\langle \nabla_{g}\tilde{\phi}_{\mathcal{O}_{k}'}, \nabla_{g}\tilde{f}_{\mathcal{O}_{k}'}\right\rangle_{L^{2}\times L^{2}(\mathbb{H}^{d})}\right| \lesssim \limsup_{k\to\infty} \left\| \nabla_{g}\tilde{\phi}_{\mathcal{O}_{k}'}\right\|_{L^{2}(\mathbb{H}^{d})},
	\end{align*}
	then
	\begin{align*}
		\limsup_{k\to\infty} \left\| \nabla_{g}\tilde{\phi}_{\mathcal{O}_{k}'}\right\|_{L^{2}(\mathbb{H}^{d})} \gtrsim \delta.
	\end{align*}
	Finally, given that $\mathcal{O}' \perp \mathcal{O}^\alpha$, $\alpha \leq A$, the sequence $\left(\tilde{\phi}_{\mathcal{O}'_k}\right)_k$ lies outside each $\mathcal{O}^\alpha$ for $\alpha \leq A$. Because $(g_k)_k$ is also absent from each $\mathcal{O}^\alpha$ for $\alpha \leq A$, we conclude that $(g_k')_k$ is absent from each $\mathcal{O}^\alpha$ for $\alpha \leq A$. This completes the proof of the lemma.
\end{proof}
	\begin{lemma}\label{le.4v6}
	Suppose that $(f_k)_{k\geq1}$ be a bounded sequence in $H^1(\mathbb{H}^d)$ and let $\delta\in(0,\delta_0]$ be sufficiently small. After passing to a subsequence, the sequence $(f_k)_{k\geq1}$ admits a decomposition into $2J+1=O(\delta^{-2})$ terms:
	\begin{align*}
		f _ { k } = \sum _ { 1 \leq \mu \leq J } \tilde { \phi } _ { \mathcal { O } ^\mu_ { k } } ^{\mu} + \sum _ { 1 \leq \nu \leq J } \tilde { \psi } _ { \tilde{ \mathcal { O } }^\nu_ { k } } ^{\nu} + r _ { k } ,
	\end{align*}
	where $\tilde{\phi}_{\mathcal{O}_k}^\mu$ and $\tilde{\psi}_{\mathcal{O}_k}^\nu$ denote the Euclidean and hyperbolic profiles associated with the sequences $(\phi^\mu, \mathcal{O}^\mu) \in H^1(\mathbb{R}^d) \times \mathcal{F}_e$ and $(\psi^\nu, \tilde{\mathcal{O}}^\nu) \in H^1(\mathbb{H}^d) \times \mathcal{F}_h$ as in Definition \ref{def.4v1}.
	Furthermore, the remainder $r_k$ is orthogonal to every frame $\mathcal{O}^\mu, \tilde{\mathcal{O}}^\nu$ for $1 \leq \mu, \nu \leq J$, and obeys
	\begin{align*}
		\limsup _ { k \to \infty } \sup_{\substack{N\ge1\\ t\in\mathbb{R}\\ x\in\mathbb{H}^d}} N ^{- \frac { d - 2} { 2 } } \big | e ^{i t \Delta _ { g } } P _ { N } r _ { k } \big | ( x ) \leq \delta .
	\end{align*}
	In addition, the families $\mathcal{O}^{\mu}$ and $\tilde{\mathcal{O}}^{\nu}$ are mutually orthogonal, and the decomposition satisfies the following asymptotic orthogonality:
	\begin{align}\label{le.4v61}
		\| \nabla _ { g } f _ { k } \| _ { L ^{2} } ^{2} = \sum _ { 1 \leq \mu \leq J } \| \nabla _ { g } \tilde { \phi } _ { \mathcal { O }^{\mu} _ { k } } ^{\mu} \| _ { L ^{2} } ^{2} + \sum _ { 1 \leq \nu \leq J } \| \nabla _ { g } \tilde { \psi } _ {\tilde{ \mathcal { O } }^{\nu} _ { k } } ^{\nu} \| _ { L ^{2} } ^{2} + \| \nabla _ { g } r _ { k } \| _ { L ^{2} } ^{2} + o _ { k } ( 1 ) ,
	\end{align}
	where $o_{k}(1) \rightarrow 0$ as $k \rightarrow \infty$.
\end{lemma}
	\begin{proof}[Proof of proposition \ref{pr.4v5}]
			The proof of the profile decomposition in Proposition \ref{pr.4v5} relies on the finite decomposition in Lemma \ref{le.4v6} in combination with the inequality $\|e^{it\Delta_{g}}f\|_{L_{t}^{\infty}L_{x}^{\frac{2d}{d-2}}} \leq \|f\|_{H^1}$ (as in \eqref{eq1.2v1}) and the interpolation inequality, which is used to obtain the necessary bounds.
		\end{proof}

	
\section{Proof of the main Theorem}\label{sec:ProofMainThm}
	
In this section, we first prove Proposition \ref{pr.2v8}. And then, we utilize this proposition to show
Theorem \ref{th1v1}.

	\begin{proof}[Proof of Proposition \ref{pr.2v8}]
By time-translation invariance, after replacing $u_k(t)$ with $u_k(t+t_k)$, we may assume $t_k=0$ for all $k\geq 1$. Let us apply Proposition \ref{pr.4v5} to the sequence $\{u_k(0)\}_k$, which is bounded in $H^1(\mathbb{H}^d)$. As a result, there exist sequences of pairs $(\phi^\mu,\mathcal{O}^\mu)\in\dot{H}^1(\mathbb{R}^d)\times\mathcal{F}_e$ and $(\psi^\nu,\tilde{\mathcal{O}}^\nu)\in H^1(\mathbb{H}^d)\times\mathcal{F}_h$, $\mu,\nu=1,2,\cdots$, for which the conclusion of Proposition \ref{pr.4v5} holds. Following Proposition 3.4 of \cite{Ionescu2012}, we give a sketch proof that
		\begin{align*}
			\lim_{J\to\infty} \lim_{k\to\infty} \left\| \nabla_g \left( F(U_k^{\mathrm{app}}) - \sum_{1\leq\gamma\leq 2J} F(U_k^{\gamma}) \right) \right\|_{N^0(\mathbb{R})} = 0,
		\end{align*}
where $U_k^{\mathrm{app}}=\sum_{1\leq \mu\leq J} U_{e,k}^{\mu}+
\sum_{1\leq \nu\leq J} U_{h,k}^{\nu}+e^{it\Delta_g}r_k^J$ is a global approximate solution with bounded $Z$ norm for all $k$ sufficiently large, and the associated nonlinear profile $U_{e,k}^{\mu}$ (resp. $U_{h,k}^{\nu}$) is defined as the maximal solution of \eqref{eq1.1v1} with initial data $U_{e,k}^{\mu}(0)=\tilde{\phi}_{\mathcal{O}_k}^{\mu}$ (resp. $U_{h,k}^{\nu}(0)=\tilde{\psi}_{\tilde{\mathcal{O}}_k}^{\nu}$), associated with each linear profile $\tilde{\phi}_{\mathcal{O}_k}^{\mu}$ (resp. $\tilde{\psi}_{\tilde{\mathcal{O}}_k}^{\nu}$). To begin, choose an arbitrarily small $\epsilon_0 > 0$. For a given $J$, we will define
		\begin{align*}
			U_{\mathrm{prof},k}^{J} = \sum_{1\leq\mu\leq J} U_{e,k}^{\mu} + \sum_{1\leq\nu\leq J} U_{h,k}^{\nu} = \sum_{1\leq\gamma\leq 2J} U_k^{\gamma}
		\end{align*}
		as the sum of the respective profiles, where $U_k^\gamma$ is the permutation of $U_{e,k}^{\mu}$ and $U_{h,k}^{\nu}$. When no distinction between Euclidean and hyperbolic profiles is needed, we simply denote either type of profile by $U_k^\gamma$. Next, we will separate
		\begin{align}\label{eq1.5v8}
			& \left\| \nabla_g \left( F(U_k^{\mathrm{app}}) - \sum_{1\leq\gamma\leq 2J} F(U_k^{\gamma}) \right) \right\|_{N^0(\mathbb{R})} \nonumber\\
			\leq& \left\| \nabla_g \left( F(U_{\mathrm{prof},k}^{J}) - \sum_{1\leq\gamma\leq 2J} F(U_k^{\gamma}) \right) \right\|_{N^0(\mathbb{R})} + \left\| \nabla_g \left( F(U_k^{\mathrm{app}}) - F(U_{\mathrm{prof},k}^{J}) \right) \right\|_{N^0(\mathbb{R})}.
		\end{align}
	Now, let us consider the first term. If $d \geq 6$, then we obtain
		\begin{align*}
			\left| \nabla_g \left( \sum_{1\leq\gamma\leq 2J} F(U_k^{\gamma}) - F(U_{\mathrm{prof},k}^{J}) \right) \right| \lesssim_{J} \sum_{\gamma\neq\tilde{\gamma}} |\nabla_g U_k^{\gamma}| \, |U_k^{\tilde{\gamma}}|^{\frac{4}{d-2}}.
		\end{align*}
		Therefore, by Lemma \ref{le.5v3}, we find that
		\begin{align*}
			 \left\| \nabla_g \left( F(U_{\mathrm{prof},k}^{J}) - \sum_{1\leq\gamma\leq 2J} F(U_k^{\gamma}) \right) \right\|_{N^0(\mathbb{R})} & \lesssim_{J} \sum_{\gamma\neq\tilde{\gamma}} \left\| \nabla_g U_k^{\gamma} |U_k^{\tilde{\gamma}}|^{\frac{4}{d-2}} \right\|_{L_{t,x}^{\frac{2(d+2)}{d+4}}} \\
			&\lesssim_{J} \sum_{\gamma\neq\tilde{\gamma}} \left\| \nabla_g U_k^{\gamma} U_k^{\tilde{\gamma}} \right\|_{L_{t,x}^{\frac{d+2}{d+1}}} \left\| \nabla_g U_k^{\gamma} \right\|_{L_{t,x}^{\frac{2(d+2)}{d}}}^{\frac{d-6}{d-2}}\\
			& \lesssim_{J, E_{\max},\eta} o(1), \quad \text{as } k\to\infty.
		\end{align*}
		If $d=4,5$, we obtain
		\begin{align*}
			\left| \nabla_g \left( \sum_{1\leq\gamma\leq 2J} F(U_k^{\gamma}) - F(U_{\mathrm{prof},k}^{J}) \right) \right| \lesssim_{J} \sum_{\gamma\neq\tilde{\gamma}} |\nabla_g U_k^{\gamma}| |U_k^{\tilde{\gamma}}| \left( \left| \sum_{1\leq\beta\leq 2J} U_k^{\beta} \right|^{\frac{6-d}{d-2}} + |U_k^{\gamma}|^{\frac{6-d}{d-2}} \right),
		\end{align*}
and therefore, applying Lemma \ref{le.5v3} gives the conclusion
\begin{align*}
& \left\| \sum_{1\leq\gamma\leq 2J} F(U_k^{\gamma}) - F(U_{\mathrm{prof},k}^{J}) \right\|_{L_{t,x}^{\frac{2(d+2)}{d+4}}} \\
 \lesssim&_{J} \sum_{\gamma\neq\tilde{\gamma}} \left\| \nabla_g U_k^{\gamma} U_k^{\tilde{\gamma}} \right\|_{L_{t,x}^{\frac{d+2}{d-1}}} \sum_{1\leq\beta\leq 2J} \left\| U_k^{\beta} \right\|_{Z(\mathbb{R}_+)}^{\frac{6-d}{d-2}} \\
\lesssim&_{J, E_{\max},\eta} o(1), \quad \text{as } k\to\infty.
		\end{align*}
		Therefore, based on the discussion, the estimates above allow us to conclude that
		\begin{align*}
			\lim_{J\to\infty} \lim_{k\to\infty} \left\| \nabla_g \left( \sum_{1\leq\gamma\leq 2J} F(U_k^{\gamma}) - F(U_{\mathrm{prof},k}^{J}) \right) \right\|_{L_{t,x}^{\frac{2(d+2)}{d+4}}} = 0.
		\end{align*}
		Regarding the first term in \eqref{eq1.5v8}, and in the case $d \geq 4$, we will now prove the corresponding results
		\begin{align}\label{eq1.5v9}
			\lim_{J\to\infty} \lim_{k\to\infty} \left\| \nabla_g \left( F(U_{\mathrm{prof},k}^{J}) - F(U_k^{\mathrm{app}}) \right) \right\|_{L_{t,x}^{\frac{2(d+2)}{d+4}}} = 0.
		\end{align}
		It is worth noting that
\begin{align*}
& \left\| \nabla_g (F(U_{\mathrm{prof},k}^{J}) - F(U_k^{\mathrm{app}}) ) \right\|_{L_{t,x}^{\frac{2(d+2)}{d+4}}} \\
\lesssim& \left\| \nabla_g e^{it\Delta_g} r_k^{J} \right\|_{L_{t,x}^{\frac{2(d+2)}{d}}} \left\| e^{it\Delta_g} r_k^{J} \right\|_{Z(\mathbb{R})}^{\frac{4}{d-2}} + \left\| \nabla_g U_k^{\mathrm{app}} \right\|_{L_{t,x}^{\frac{2(d+2)}{d}}} \left\| e^{it\Delta_g} r_k^{J} \right\|_{Z(\mathbb{R})}^{\frac{4}{d-2}} \\
&+ \left\| |U_k^{\mathrm{app}}|^{\frac{4}{d-2}} \nabla_g e^{it\Delta_g} r_k^{J} \right\|_{L_{t,x}^{\frac{2(d+2)}{d+4}}} + \left\| \nabla_g U_k^{\mathrm{app}} \right\|_{L_{t,x}^{\frac{2(d+2)}{d}}} \left\| e^{it\Delta_g} r_k^{J} \right\|_{Z(\mathbb{R})} \left\| U_k^{\mathrm{app}} \right\|_{Z(\mathbb{R})}^{\frac{6-d}{d-2}}.
		\end{align*}
	If $d \geq 6$, the last term disappears, which simplifies the overall expression.
		
		Applying \eqref{pr.4v51},  together with the fact $r_k^{J}$ is bounded in $H^1(\mathbb{H}^d)$, we obtain that \eqref{eq1.5v9} follows as soon as we can prove
		\begin{align}\label{eq1.5v10}
			\lim_{J\to\infty} \lim_{k\to\infty} \left\| |U_k^{\mathrm{app}}|^{\frac{4}{d-2}} \nabla_g e^{it\Delta_g} r_k^{J} \right\|_{L_{t,x}^{\frac{2(d+2)}{d+4}}(\mathbb{R}\times\mathbb{H}^d)}= 0.
		\end{align}
		Combining Hölder's inequality and the Strichartz inequality gives
\begin{align*}
& \left\| |U_k^{\mathrm{app}}|^{\frac{4}{d-2}} \nabla_g e^{it\Delta_g} r_k^{J} \right\|_{L_{t,x}^{\frac{2(d+2)}{d+4}}} \\
\lesssim& \left\| U_k^{\mathrm{app}} \right\|_{Z(\mathbb{R})}^{\frac{3}{d-2}} \left\| \nabla_g e^{it\Delta_g} r_k^{J} \right\|_{L_{t,x}^{\frac{2(d+2)}{d}}}^{\frac{d-3}{d-2}} \left\| U_k^{\mathrm{app}} \nabla_g e^{it\Delta_g} r_k^{J} \right\|_{L_{t,x}^{\frac{d+2}{d-1}}}^{\frac{1}{d-2}} \\
 \lesssim& \left\| \sum_{1\leq\gamma\leq 2J} U_k^{\gamma} \nabla_g e^{it\Delta_g} r_k^{J} \right\|_{L_{t,x}^{\frac{d+2}{d-1}}}^{\frac{1}{d-2}} + \left\| e^{it\Delta_g} r_k^{J} \right\|_{Z(\mathbb{R})}^{\frac{1}{d-2}} \left\| \nabla_g e^{it\Delta_g} r_k^{J} \right\|_{L_{t,x}^{\frac{2(d+2)}{d}}}^{\frac{1}{d-2}} \\
 \lesssim& \left\| e^{it\Delta_g} r_k^{J} \right\|_{Z(\mathbb{R})}^{\frac{1}{d-2}} + \left\| \sum_{1\leq\gamma\leq 2J} U_k^{\gamma} \nabla_g e^{it\Delta_g} r_k^{J} \right\|_{L_{t,x}^{\frac{d+2}{d-1}}}^{\frac{1}{d-2}}.
		\end{align*}
		According to \eqref{pr.4v51}, the first term contributes adequately to \eqref{eq1.5v10}. Thus, the proof of \eqref{eq1.5v10} is reduced to showing that for each 	
		$1 \leq \gamma \leq 2J$, it holds that
		\begin{align}\label{eq1.5v11}
			\lim_{J\to\infty} \limsup_{k\to\infty}  \left\| U_k^{\gamma} \nabla_g e^{it\Delta_g} r_k^{J} \right\|_{L_{t,x}^{\frac{d+2}{d-1}}} = 0.
		\end{align}
		Let $\epsilon > 0$ be fixed. For each $U_k^{\gamma}$ defined below, we now define the following sets
		\begin{align*}
			S_{N_k^{\gamma},h_k^{\gamma}}^{R} = \{ (t,x) \in \mathbb{R} \times \mathbb{H}^{d} ,\; t \in [ -N_k^{-2}R^{2}, N_k^{-2}R^{2} ],\; d_g(h_k^{-1}x,\mathbf{0}) \leq N_k^{-1}R \}.
		\end{align*}
		When $R$ is chosen sufficiently large, it follows from \eqref{eq1.5v19} that
\begin{align*}
\left\| U_k^{\gamma} \nabla_g e^{it\Delta_g} r_k^{J} \right\|_{L_{t,x}^{\frac{d+2}{d-1}}\bigl( (\mathbb{R}\times\mathbb{H}^{d})\setminus S_{N_k^{\gamma},h_k^{\gamma}}^{R} \bigr)} 
 \leq \left\| U_k^{\gamma} \right\|_{L_{t,x}^{\frac{2(d+2)}{d-2}}\bigl( (\mathbb{R}\times\mathbb{H}^{d})\setminus S_{N_k^{\gamma},h_k^{\gamma}}^{R} \bigr)} \left\| \nabla_g e^{it\Delta_g} r_k^{J} \right\|_{L_{t,x}^{\frac{2(d+2)}{d}}}^{\frac{2(d+2)}{d}} \lesssim_{E_{\max},\eta} \epsilon.
		\end{align*}
		When dealing with a hyperbolic profile $U_{h,k}^{\nu}$, recall that $W^{\nu}$ satisfies $W^{\nu} \in Z(\mathbb{R})$. If we take $W^{\nu,1} \in C_{0}^{\infty}(\mathbb{R} \times \mathbb{H}^{d})$, it follows that	
		\begin{align*}
			\left\| W^{\nu} - W^{\nu,1} \right\|_{Z(\mathbb{R})} \leq \epsilon .
		\end{align*}
As a consequence of $\|U_{h, k}^\nu (t) - \Pi_{h_k}W^\nu(t - t_k^\nu)\|_{\mathcal{S}^1} \to 0$ as $k \to \infty$, there exists a constant $C_{\nu,\epsilon}$ such that
\begin{align*}
 & \left\| U_{h,k}^{\nu} \nabla_g e^{it\Delta_g} r_k^{J} \right\|_{L_{t,x}^{\frac{d+2}{d-1}}(S_{N_k^{\nu},h_k^{\nu}}^{R})} \\
 \leq& \left\| (U_{h,k}^{\nu} - \pi_{h_k^{\nu}} W^{\nu,1}) \nabla_g e^{it\Delta_g} r_k^{J} \right\|_{L_{t,x}^{\frac{d+2}{d-1}}(S_{N_k^{\nu},h_k^{\nu}}^{R})} + \left\| W^{\nu,1} \right\|_{L_{t,x}^{\frac{2(d+2)}{d-4}}} \left\| \nabla_g e^{it\Delta_g} r_k^{J} \right\|_{L_{t,x}^{2}(S_{N_k^{\nu},h_k^{\nu}}^{R})} \\
\lesssim&_{E_{\max},\eta} \epsilon + C_{\nu,\epsilon} \left\| \nabla_g e^{it\Delta_g} r_k^{J} \right\|_{L_{t,x}^{2}(S_{N_k^{\nu},h_k^{\nu}}^{R})},
		\end{align*}
		where $\left\| W^{\nu,1} \right\|_{L_{t,x}^{\frac{2(d+2)}{d-4}}}$ becomes $\left\| W^{\nu,1} \right\|_{L_{t,x}^{\infty}}$ for $d=4$.
	Next, we consider a Euclidean profile, choosing $v \in C_{0}^{\infty}(\mathbb{R} \times \mathbb{R}^{d})$ satisfying
	\begin{align*}
		\| v' - v \|_{Z(\mathbb{R})} \leq \epsilon,
	\end{align*}
	where $v'$ is the function from Lemma \ref{le.3v3}. We then repeat the estimation procedure as follows:
	\begin{align*}
		\left\| U_{e,k}^{\mu} \nabla_g e^{it\Delta_g} r_k^{J} \right\|_{L_{t,x}^{\frac{d+2}{d-1}}(S_{N_k^{\mu},h_k^{\mu}}^{R})} 
		\lesssim&_{E_{\max},\eta} \epsilon + \left\| \tilde{u}_k \nabla_g e^{it\Delta_g} r_k^{J} \right\|_{L_{t,x}^{\frac{d+2}{d-1}}(S_{N_k^{\mu},h_k^{\mu}}^{R})} \\ 
		\lesssim&_{E_{\max},\eta} \epsilon + \left\| \tilde{u}_k \right\|_{L_{t,x}^{\frac{2(d+2)}{d-4}}(S_{N_k^{\mu},h_k^{\mu}}^{R})} \left\| \nabla_g e^{it\Delta_g} r_k^{J} \right\|_{L_{t,x}^{2}(S_{N_k^{\mu},h_k^{\mu}}^{R})} \\ 
		\lesssim&_{E_{\max},\eta} \epsilon + C_{\mu,\epsilon} N_k \left\| \nabla_g e^{it\Delta_g} r_k^{J} \right\|_{L_{t,x}^{2}(S_{N_k^{\mu},h_k^{\mu}}^{R})},
	\end{align*}	
	where we use $\left\| \tilde{u}_k \right\|_{L_{t,x}^\infty}$ for $d=4$ instead of $\left\| \tilde{u}_k \right\|_{L_{t,x}^{\frac{2(d+2)}{d-4}}}$.
	Thus, we conclude for all cases, 	
	\begin{align*}
		\left\| U_k^{\gamma} \nabla_g e^{it\Delta_g} r_k^{J} \right\|_{L_{t,x}^{\frac{d+2}{d-1}}(S_{N_k^{\gamma},h_k^{\gamma}}^{R})} \lesssim_{E_{\max},\eta} \epsilon + C_{\alpha,\epsilon} N_k^{\gamma} \left\| \nabla_g e^{it\Delta_g} r_k^{J} \right\|_{L_{t,x}^{2}(S_{N_k^{\gamma},h_k^{\gamma}}^{R})}.
	\end{align*}	
	Finally, using both Lemma \ref{le.2v4} and \eqref{pr.4v51}, we obtain	
	\begin{align*}
		\underset{k \to \infty} \limsup \left\| U_k^{\gamma} \nabla_g e^{t\Delta_g} r_k^{J} \right\|_{L_{t,x}^{\frac{d+2}{d-1}}(S_{N_k^{\gamma},h_k^{\gamma}}^{R})} \lesssim_{E_{\max},\eta} \epsilon.
	\end{align*}	
	Because $\epsilon > 0$ can be chosen arbitrarily, \eqref{eq1.5v11} holds. Therefore, the proof of \eqref{eq1.5v10} is complete, and hence \eqref{eq1.5v9} is proved when $d \geq 4$.
	
	As $u_k(0)=U_k^{\mathrm{app}}(0)$, it follows from Sobolev's inequality and the conclusion of Proposition \ref{pr.2v6} that for all $k$ sufficiently large and for every interval $J$,	
	\begin{align*}
		\| u_k \|_{Z(J)} \lesssim \| u_k \|_{S^1(J)} \lesssim \| u_k - U_k^{\mathrm{app}} \|_{S^1(J)} + \| U_k^{\mathrm{app}} \|_{S^1(\mathbb{R})} \lesssim_{E_{\max},\eta} 1.
	\end{align*}
	As a result, $u_k$ is global for all $k$ large enough and has a uniformly bounded $Z(\mathbb{R})$-norm, which is in contradiction with Proposition \ref{pr.2v81}. This completes the proof.
	\end{proof}
	
\begin{lemma}[Linear evolution criterion]\label{le.5v1}
Let $M > 0$ be arbitrary. There exists $\delta > 0$ such that for any interval $J \subset \mathbb{R}$, if 
\begin{align*} 
\| \nabla_g \phi \|_{L^2(\mathbb{H}^d)} \leq M,\qquad \| e^{it\Delta_g} \phi \|_{Z(J)} \leq \delta, 
\end{align*} 
then, given any $t_0 \in J$, the maximal solution $(I,u)$ of \eqref{eq1.1v1} with initial condition $u(t_0) = e^{it_0\Delta_g}\phi$ obeys $J \subset I$ and the following inequalities hold
	\begin{align}\label{eq1.5v12} 
		\| u - e^{it\Delta_g} \phi \|_{S^1(J)} \leq \delta^{\frac{d}{d-2}},\qquad \| u \|_{S^1(J)} \leq C(M,\delta). 
	\end{align} 
	Similarly, when $J=(-\infty,T)$, there also exists a unique maximal solution $(I,u)$ with $J\subset I$ such that
	\begin{align*} 
		\lim_{t\to -\infty} \| \nabla_g ( u(t) - e^{it\Delta_g} \phi ) \|_{L^2(\mathbb{H}^d)} = 0,
	\end{align*} 
	and \eqref{eq1.5v12} is also satisfied in this case. Similar results hold in the Euclidean case, with $(\mathbb{R}^d,\delta_{ij})$ taking the place of $(\mathbb{H}^d,g)$. 
\end{lemma}
\begin{proof} 	
Proposition \ref{pr.2v6} directly implies the first part. Define $v=e^{it\Delta_g}\phi$. Then \eqref{pr.2v61} is satisfied, and applying the Strichartz estimates yields
\begin{align*} 
			\| \nabla_g v | v|^{\frac{4}{d-2}} \|_{L_t^2 L_x^{\frac{2d}{d+2}}(J\times\mathbb{H}^d)} \leq \| v \|_{Z(J)}^{\frac{4}{d-2}} \| \nabla_g e^{it\Delta_g} \phi \|_{L_t^{\frac{2(d+2)}{d-2}} L_x^{\frac{2d(d+2)}{d^2+4}}(J\times\mathbb{H}^d)} \lesssim M \delta^{\frac{4}{d-2}}.
	\end{align*}	
Hence \eqref{pr.2v62} holds. Applying Proposition \ref{pr.2v6} with $\rho=1$ concludes the proof. The second assertion is classical and follows from a fixed-point argument.	
	\end{proof}
	
The following is a description of a Euclidean nonlinear profile.
	\begin{lemma}[Euclidean nonlinear profile description]\label{le.5v2}
	 Let $\phi \in \dot{H}^1(\mathbb{R}^d)$ and $(N_k, t_k, h_k)_k \in \mathcal{F}_e$ be given. Let $U_k$ be the solution of \eqref{eq1.1v1} such that $U_k(0) = \Pi_{t_k, h_k}(T_{N_k} \phi)$.
		\begin{enumerate} 	
			\item[(i)] First, when $k$ is large enough, $U_k$ is defined for all times, lies in $C(\mathbb{R}:H^1)$, and satisfies
			\begin{align*} 
				\| U_k \|_{L^2(\mathbb{R})} \leq 2C\bigl( E_{\mathbb{R}^d}(\phi) \bigr).
			\end{align*} 
			\item[(ii)] Next, there exists a Euclidean solution $u \in C(\mathbb{R} : \dot{H}^1(\mathbb{R}^d))$ of
			\begin{align}\label{le.5v21} 
				(i\partial_t + \Delta) u = |u|^{\frac{4}{d-2}} u
			\end{align} 
			with scattering data $\phi^{\pm\infty}$ as defined in \eqref{eq1.3v4} such that, after passing to a subsequence, the following condition is satisfied: for any $\varepsilon > 0$, there exists $T(\phi,\varepsilon)$; then for all $T \geq T(\phi,\varepsilon)$, there is an $R(\phi,\varepsilon,T)$ such that whenever $R \geq R(\phi,\varepsilon,T)$, the following estimate holds
			\begin{align}\label{le.5v22} 
				\| U_k - \tilde{u}_k \|_{S^1( |t-t_k| \leq T N_k^{-2} )} \leq \varepsilon,
			\end{align} 
			for every sufficiently large $k$, with
			\begin{align*} 
				(\pi_{h_k^{-1}} \tilde{u}_k)(t,x) = N_k^{\frac{d-2}{2}} \eta\bigl( N_k \Psi_I^{-1}(x)/R \bigr) u\bigl( N_k \Psi_I^{-1}(x),\, N_k^2 (t-t_k) \bigr).
			\end{align*} 
			In addition, taking a subsequence, we have 
			\begin{align}\label{le.5v23} 
				\|U_k\|_{L_t^{\frac{2(d+2)}{d-2}} H_x^{1,\frac{2d(d+2)}{d^2+4}} \cap L_t^{\frac{2(d+2)}{d}} H_x^{1,\frac{2(d+2)}{d}}(\mathbb{H}^d\times\{N_k^2|t-t_k|\geq T\})} \leq \varepsilon
			\end{align} 
			and for any $\pm(t-t_k)\geq T N_k^{-2}$, 
			\begin{align}\label{le.5v24} 
				\| \nabla_g \bigl( U_k(t) - \Pi_{t_k-t, h_k} T_{N_k} \phi^{\pm\infty} \bigr) \|_{L^2} \leq \varepsilon,
			\end{align} 
			for all sufficiently large $k$ depending on $\phi$, $\varepsilon$, $T$, $R$. 
		\end{enumerate}
	\end{lemma}
	Lemma \ref{le.5v2} has been proved in \cite[Lemma 6.2]{Ionescu2012} for  $d=3$. After a careful computation, it can be extend to higher dimensions similarity.
	\begin{lemma}\label{le.5v3}
	Consider two profiles $\tilde{\phi}_{\mathcal{O}_k}$ and $\tilde{\psi}_{\mathcal{O}_k'}$ associated with orthogonal frames $\mathcal{O}$ and $\mathcal{O}'$. Let $U_k$ and $U'_k$ be the solutions of \eqref{eq1.1v1} satisfying $U_k(0) = \tilde{\phi}_{\mathcal{O}_k}$ and $U'_k(0) = \tilde{\psi}_{\mathcal{O}_k'}$. Assume further that if $\mathcal{O} \in \mathcal{F}_h$ (respectively $\mathcal{O}' \in \mathcal{F}_h$), then for both profiles, we have 	
	\[
	E(\tilde{\phi}_{\mathcal{O}_k}) < E_{\max} - \eta \quad (\text{resp. } E(\tilde{\psi}_{\mathcal{O}_k'}) < E_{\max} - \eta)
	\]
	for some $\eta > 0$. Then, up to a subsequence, the following conclusions hold
		\begin{enumerate}
			\item[(i)] for any $T \in \mathbb{R}$, we have
			\begin{align}\label{le.5v31}
				\left\langle \nabla_g U_k(T), \nabla_g U_k'(T) \right\rangle_{L^2\times L^2(\mathbb{H}^d)} \to 0, \quad \text{as } k\to\infty.
			\end{align}
			
			\item[(ii)]
			\begin{align}\label{le.5v32}
				\left\| U_k \nabla_g U_k' \right\|_{L_t^{\frac{d+2}{d-2}} L_x^{\frac{d(d+2)}{d^2-d+2}}(\mathbb{R}\times\mathbb{H}^d)} + \left\| U_k \nabla_g U_k' \right\|_{L_{t,x}^{\frac{d+2}{d-1}} \cap L_{t,x}^{\frac{d+2}{2(d-2)}}(\mathbb{R}\times\mathbb{H}^d)} 
			 + \left\| \nabla_g U_k \cdot \nabla_g U_k' \right\|_{L_{t,x}^{\frac{d+2}{d}}(\mathbb{R}\times\mathbb{H}^d)} \to 0 
			\end{align}
		\end{enumerate}
		$ \text{as } k\to\infty$.
	\end{lemma}
	
	\begin{proof}
		(i) We choose $\epsilon_0 > 0$ to be sufficiently small. First, consider the hyperbolic frame case. Let $W$ be the solution of \eqref{eq1.1v1} with initial condition $W(0) = \phi$. By the invariance of \eqref{eq1.1v1} under the $\pi$-action, we have $U_k(0) = \pi_{h_k}(W(T))$. Next, assume that $\mathcal{O} = (N_k, h_k)_k$ with $N_k \to \infty$ and that $\mathcal{O}$ is equivalent to $(N_k, h_k)_k$. By applying Lemma \ref{le.3v3}, we have that there exists some $R > 0$ such that	
\begin{align*} 
\| U_k(T) - \tilde{u}_k(T) \|_{H^1(\mathbb{H}^d)} \leq \frac{\epsilon_0}{2}, 
\end{align*} 
for sufficiently large $k$. More specifically, when $k$ is large enough, 
\begin{align*} 
\left\| U_k(T) - \pi_{h_k} T_{N_k}\bigl( \eta(\frac{\cdot}{R}) u(T_0,\cdot) \bigr) \right\|_{H^1(\mathbb{H}^d)} \leq \epsilon_0.
\end{align*} 
Thus, we can substitute $U_k(T)$ and $U_k'(T)$ with a linear profile in new frames $\mathcal{O}_T$ and $\mathcal{O}_T'$ (up to an error $\epsilon_0$) such that $\mathcal{O}_T$ and $\mathcal{O}_T'$ are orthogonal if and only if $\mathcal{O}$ and $\mathcal{O}'$ are orthogonal. Then \eqref{le.5v31} is a consequence of Lemma \ref{le.4v4} (ii).
		
		(ii) Let us demonstrate that the first norm in \eqref{le.5v32} decays; similar reasoning will apply to the second norm. Fix an arbitrary $\epsilon > 0$. Then, by Lemma \ref{le.3v3} when $U_k$ arises from a Euclidean frame, we deduce
		\begin{align*} 
			\| U_k \|_{S^1(\mathbb{R})} + \| U_k' \|_{S^1(\mathbb{R})} \leq M < \infty, 
		\end{align*} 
		and that there exist $R$ and $\delta$ such that
		\begin{align}\label{eq1.5v19} 
			&\|\nabla_g U_k\|_{L_t^{\frac{2(d+2)}{d-2}} L_x^{\frac{2d(d+2)}{d^2+4}} \cap L_{t,x}^{\frac{2(d+2)}{d}}\bigl((\mathbb{R}\times\mathbb{H}^d)\setminus S_{N_k,t_k,h_k}^{R}\bigr)} + \|U_k\|_{L_{t,x}^{\frac{2(d+2)}{d-2}}\bigl((\mathbb{R}\times\mathbb{H}^d)\setminus S_{N_k,t_k,h_k}^{R}\bigr)} \leq \epsilon,\nonumber\\ 
			&\sup_{\substack{S\in\mathbb{ R }\\h\in SO(d,1)} } \Bigl( \|\nabla_g U_k\|_{L_t^{\frac{2(d+2)}{d-2}} L_x^{\frac{2d(d+2)}{d^2+4}} \cap L_{t,x}^{\frac{2(d+2)}{d}}(S_{N_k,S,h}^{\delta})} + \|U_k\|_{L_{t,x}^{\frac{2(d+2)}{d-2}}(S_{N_k,S,h}^{\delta})} \Bigr) \leq \epsilon, 
		\end{align} 
		where
		\begin{align}\label{eq1.5v20} 
			S_{N,T,h}^{a} := \left\{ (t,x) \in \mathbb{R} \times \mathbb{H}^d : |t-T| \leq N^{-2}a^2 \text{ and } d_g(h^{-1}x,\mathbf{0}) \leq N^{-1}a \right\}. 
		\end{align} 
		The same $R$ and $\delta$ also yield a similar claim for $U'_k$. 
		
		In the case where $\frac{N_k}{N'_k} \to \infty$, for sufficiently large $k$, we obtain
		\begin{align*} 
			\left\| U_k \nabla_g U_k' \right\|_{L_t^{\frac{d+2}{d-2}} L_x^{\frac{d(d+2)}{d^2-d+2}}} 
			\leq& \left\| U_k \right\|_{L_{t,x}^{\frac{2(d+2)}{d-2}}} \left\| \nabla_g U_k' \right\|_{L_t^{\frac{2(d+2)}{d-2}} L_x^{\frac{2d(d+2)}{d^2+4}}(S_{N_k',t_k,h_k}^{\delta})} \\ 
			&+ \left\| U_k \right\|_{L_{t,x}^{\frac{2(d+2)}{d-2}}\bigl((\mathbb{R}_+\times\mathbb{H}^d)\setminus S_{N_k,t_k,h_k}^{R}\bigr)} \left\| \nabla_g U_k' \right\|_{L_t^{\frac{2(d+2)}{d-2}} L_x^{\frac{2d(d+2)}{d^2+4}}} \lesssim_{M} \epsilon . 
		\end{align*} 
		The situation where $\frac{N_k'}{N_k} \to \infty$ is analogous. Otherwise, we may assume that $C^{-1} \leq \frac{N_k}{N'_k} \leq C$ for all $k$. Then, for sufficiently large $k$, we can ensure that $S_{N_k,t_k,h_k}^R \cap S_{N_k',t_k',h_k'}^R = \emptyset$. Thus, by \eqref{eq1.5v19}, we obtain as before that
		\begin{align*} 
			\left\| U_k \nabla_g U_k' \right\|_{L_t^{\frac{d+2}{d-2}} L_x^{\frac{d(d+2)}{d^2-d+2}}} \lesssim_{M} \epsilon . 
		\end{align*} 
		So, in all cases, 
		\begin{align*} 
			\lim_{k\to\infty} \left\| U_k \nabla_g U_k' \right\|_{L_t^{\frac{d+2}{d-2}} L_x^{\frac{d(d+2)}{d^2-d+2}}} \lesssim_{M} \epsilon . 
		\end{align*} 
		Thus, the first term in \eqref{le.5v32} converges to 0.
	\end{proof}
	
Next, we turn to prove Theorem \ref{th1v1}.
	\begin{proof}[Proof of Theorem \ref{th1v1}]
	Suppose, for the sake of contradiction, that $E_{\max} < +\infty$. We first demonstrate that there exists a solution $u \in C((-T_*, T^*):H^1)$ to \eqref{eq1.1v1} satisfying
	\begin{align}\label{eq1.5v21} 
		E(u) = E_{\max}, \qquad \| u \|_{Z(-T_*,0)} = \| u \|_{Z(0,T^*)} = +\infty . 
	\end{align} 
	By our assumption, we can find a sequence of solutions $u_k$ defined on intervals $I_k = (-T_k, T^k)$ with $E(u_k) \leq E_{\max}$ and $\|u_k\|_{Z(I_k)} \to +\infty$. Choosing appropriate times $t_k \in (-T_k, T^k)$ fits the hypotheses of Proposition \ref{pr.2v8}. Hence, after passing to a subsequence, there exist isometries $h_k \in \mathbb{G}$ such that $\pi_{h_k}(u_k(t_k)) \to w_0$ strongly in $H^1$. 
	
	Now let $u \in C((-T_*,T^*) : H^1(\mathbb{H}^d))$ be the maximal solution of \eqref{pr.2v51} with initial data $w_0$, as stated in Proposition \ref{pr.2v6}. By the stability result of Proposition \ref{pr.2v6}, if $\|u\|_{Z(0,T^*)} < +\infty$, then we would have $T^* = +\infty$ and $\|u_k\|_{Z(I_k,+\infty)} \leq C(\|u\|_{Z(0,+\infty)})$, which is impossible. A similar argument shows $\|u\|_{Z(-T_*,0)} = +\infty$, establishing \eqref{eq1.5v21}. 
	
	Next, we argue that the solution $u$ obtained above can be extended globally. By Proposition \ref{pr.2v5}, it is enough to find a $\delta > 0$ such that for every time $t \in (-T_*, T^*)$, 
	\begin{align*} 
		\| u \|_{Z((t-\delta, t+\delta) \cap (-T_*, T^*))} \leq 1 . 
	\end{align*} 	
If this fails, there would exist a sequence $\delta_k \to 0$ and times $t_k \in (-T_* + \delta_k, T^* - \delta_k)$ with
\begin{align}\label{eq1.5v22} 
		\| u \|_{Z(t_k-\delta_k, t_k+\delta_k)} \geq 1 . 
	\end{align} 
	We apply Proposition \ref{pr.2v8} to the constant sequence $u_k = u$. Up to a subsequence, we obtain isometries $h_k \in \mathbb{G}$ such that $\pi_{h_k}(u(t_k)) \to w$ strongly in $H^1$ for some $w$. Let $z$ be the maximal nonlinear solution with initial data $w$. The local theory provided by Proposition \ref{pr.2v5} yields a $\delta > 0$ such that	
	\begin{align*} 
		\| z \|_{Z(-\delta,\delta)} \leq \frac{1}{2}. 
	\end{align*} 
	Proposition \ref{pr.2v6} then implies $\|u\|_{Z(t_k-\delta_k,t_k+\delta_k)} \leq 1/2 + o_k(1)$, contradicting \eqref{eq1.5v22}. Thus, we have shown that when $E_{\max} < \infty$, there exists a global solution $u \in C(\mathbb{R}:H^1)$ of \eqref{eq1.1v1} such that
	\begin{align*} 
		E(u) = E_{\max}, \qquad \| u \|_{Z(-\infty,0)} = \| u \|_{Z(0,\infty)} = +\infty . 
	\end{align*} 	
	We now claim that for some $\delta > 0$,	
	\begin{align*} 
		\left\| u(t) \right\|_{L^{\frac{2d}{d-2}}} \geq \delta \quad \text{for all } t . 
	\end{align*} 
	Indeed, if this were false, we could pick a sequence $t_k \in (0, \infty)$ with $u(t_k) \to 0$ in $L^{\frac{2d}{d-2}}$. Applying Proposition \ref{pr.2v8} to this sequence and passing to a subsequence yields isometries $h_k \in \mathbb{G}$ such that $\pi_{h_k}(u(t_k)) \to w$ in $H^1$ with $w=0$. But this contradicts the conservation of energy. 
	
	Finally, we have reached a contradiction with the Morawetz estimate \eqref{pr.2v71}. Hence, $E_{\max} = +\infty$ as desired.
\end{proof}

\end{document}